\IfFileExists{siamart251216.cls}{\documentclass{siamart251216}}{\documentclass[11pt]{article}\usepackage{amsthm}}

\usepackage[left=1.08in,right=1.08in,top=1in,bottom=1in,marginparwidth=1.0in,marginparsep=0.06in]{geometry}
\usepackage{amsmath,amssymb,amsfonts,bm}
\usepackage{mathtools}
\usepackage{booktabs}
\usepackage{graphicx}
\usepackage{float}
\usepackage{xcolor}
\usepackage{enumitem}
\usepackage{hyperref}
\usepackage[numbers,sort&compress]{natbib}
\usepackage[T1]{fontenc}
\usepackage{lmodern}
\IfFileExists{microtype.sty}{\usepackage{microtype}}{}
\hypersetup{colorlinks=true, linkcolor=blue, citecolor=blue, urlcolor=blue}
\providecommand{\email}[1]{\href{mailto:#1}{#1}}
\allowdisplaybreaks
\makeatletter
\@ifundefined{theorem}{\newtheorem{theorem}{Theorem}[section]}{}
\@ifundefined{lemma}{\newtheorem{lemma}[theorem]{Lemma}}{}
\@ifundefined{proposition}{\newtheorem{proposition}[theorem]{Proposition}}{}
\@ifundefined{corollary}{\newtheorem{corollary}[theorem]{Corollary}}{}
\@ifundefined{definition}{\newtheorem{definition}[theorem]{Definition}}{}
\@ifundefined{newsiamremark}{%
  
  \newtheorem{assumption}[theorem]{Assumption}
}{%
  \newsiamremark{remark}{Remark}
  \newsiamremark{assumption}{Assumption}
}
\makeatother

\usepackage{cleveref}
\crefname{assumption}{Assumption}{Assumptions}
\Crefname{assumption}{Assumption}{Assumptions}
\crefname{definition}{Definition}{Definitions}
\Crefname{definition}{Definition}{Definitions}
\crefname{lemma}{Lemma}{Lemmas}
\Crefname{lemma}{Lemma}{Lemmas}
\crefname{proposition}{Proposition}{Propositions}
\Crefname{proposition}{Proposition}{Propositions}
\crefname{corollary}{Corollary}{Corollaries}
\Crefname{corollary}{Corollary}{Corollaries}
\crefname{theorem}{Theorem}{Theorems}
\Crefname{theorem}{Theorem}{Theorems}
\crefname{remark}{Remark}{Remarks}
\Crefname{remark}{Remark}{Remarks}

\newcommand{\E}{\mathbb{E}}

\newcommand{\N}{\mathcal{N}}
\newcommand{\R}{\mathbb{R}}
\newcommand{\dd}{\mathrm{d}}
\newcommand{\tr}{\mathrm{Tr}}
\newcommand{\diag}{\mathrm{diag}}
\newcommand{\Cov}{\mathrm{Cov}}

\newcommand{\KL}{\mathrm{KL}}
\newcommand{\ESS}{\mathrm{ESS}}

\newcommand{\alphat}{\alpha_t}
\newcommand{\gammat}{\gamma_t}
\newcommand{\lambdat}{\lambda_t}
\newcommand{\KOU}{K_t^{\mathrm{OU}}}
\newcommand{\rhoOU}{\rho_{y,t}}

\newcommand{\nor}[1]{\left\|#1\right\|}
\newcommand{\inner}[2]{\left\langle #1,#2\right\rangle}

\newcommand{\op}{\mathrm{op}}
\newcommand{\eff}{\mathrm{eff}}

\DeclareRobustCommand{\TWD}{\ifmmode\mathrm{TWD}\else\textnormal{\textsc{TWD}}\fi}
\DeclareRobustCommand{\TSI}{\ifmmode\mathrm{TSI}\else\textnormal{\textsc{TSI}}\fi}
\DeclareRobustCommand{\BLEND}{\ifmmode\mathrm{ScalarBlend}\else\textnormal{\textsc{Scalar Blend}}\fi}
\DeclareRobustCommand{\MBLEND}{\ifmmode\mathrm{MatrixBlend}\else\textnormal{\textsc{Matrix Blend}}\fi}
\DeclareRobustCommand{\USBLEND}{\ifmmode\mathrm{UniformScalarBlend}\else\textnormal{\textsc{Uniform Scalar Blend}}\fi}
\DeclareRobustCommand{\UMBLEND}{\ifmmode\mathrm{UniformMatrixBlend}\else\textnormal{\textsc{Uniform Matrix Blend}}\fi}
\DeclareRobustCommand{\LFGI}{\ifmmode\mathrm{LFGI}\else\textnormal{\textsc{LFGI}}\fi}
\DeclareRobustCommand{\SNIS}{\ifmmode\mathrm{SNIS}\else\textnormal{\textsc{SNIS}}\fi}

\newcommand{\sMBLEND}{\widehat s_{\rm MB}}

\newcommand{\sLFGI}{\widehat s_{\LFGI}}

\newcommand{\Gstar}{G_\star}
\newcommand{\Ghat}{\widehat G}
\newcommand{\Gainstar}{\Delta\mathcal R_\star}
\newcommand{\lambdaridge}{\lambda_{\rm ridge}}
\newcommand{\Ns}{N_{\rm s}}
\newcommand{\Ng}{N_{\rm g}}
\newcommand{\Gstareps}{G_{\star,\varepsilon}}
\newcommand{\Cebd}{C_{\eb\delta}}
\newcommand{\Cdd}{C_{\delta\delta}}
\newcommand{\Debd}{\Delta_{\eb\delta}}
\newcommand{\Ddd}{\Delta_{\delta\delta}}
\newcommand{\Lambdapole}{\Lambda_{\rm pole}}
\newcommand{\Hhat}{\widehat H}
\newcommand{\AhatN}{\widehat A_N}
\newcommand{\eb}{e_b}
\newcommand{\OmegaG}{\Omega}
\newcommand{\Psiop}{\Psi_t}

\DeclareMathOperator*{\argmin}{arg\,min}

\newcommand{\qpf}{q^{\mathrm{PF}}}
\newcommand{\Epf}{E^{\mathrm{PF}}}
\newcommand{\logZ}{\log Z}
\newcommand{\GN}{\mathrm{GN}}
\newcommand{\bank}{\mathcal B}
\newcommand{\tmin}{t_{\min}}
\newcommand{\tmax}{t_{\max}}
\newcommand{\phibase}{\varphi}
\newcommand{\tildep}{\widetilde p_0}

\title{Laplace--Fisher Gate Identities for\\
Optimal Matrix-Gated Blended Score Estimation}

\author{%
 Alois Duston%
 \thanks{The Oden Institute for Computational Engineering and Sciences,
 The University of Texas at Austin (\email{alois.duston@utexas.edu}).}%
 \and
 Tan Bui-Thanh%
 \thanks{The Oden Institute for Computational Engineering and Sciences, and
 the Department of Aerospace Engineering \& Engineering Mechanics,
 The University of Texas at Austin (\email{tanbui@oden.utexas.edu}).}%
}

\date{\today}

\begin{document}
\maketitle

\begin{abstract}
Sampling from an unnormalized target density by reversing an Ornstein--Uhlenbeck diffusion requires the score of each noise-perturbed marginal law.  Two exact identities are available: Tweedie's identity and a target-score identity, each yielding unbiased finite-reference score estimators for the OU-marginal score.  Score estimators induced by scalar blends of Tweedie and TSI score estimators can reduce variance, but they are too rigid for singular or strongly anisotropic targets.  We formulate blended score estimation as a conditional risk-minimization problem over matrix valued blending coefficients, referred to as gates.  Our central result is to show the optimal matrix valued gate for blended score estimation is given
\[
    \Gstar(y,t)
    =
    \alphat^2
    \left(
        \alphat^2 I_d
        +
        \gammat\,
        \E[H_0(X_0)\mid Y_t=y]
    \right)^{-1},
    \qquad
    H_0=-\nabla^2\log p_0 .
\]
Here $\alphat=e^{-t}$ and $\gammat=1-e^{-2t}$ are the OU coefficients, and the conditional expectation is under the OU posterior of $X_0$ given $Y_t=y$.  We call this formula the \emph{Laplace--Fisher Gate Identity} (\LFGI{}).  Because the Tweedie--TSI disagreement has conditional mean zero, the gate changes the score-estimator variance but not its expected value.  We derive the variance-optimal matrix gate, record the Gaussian special case, and establish finite-reference consistency and stability bounds for estimating the gate from weighted reference samples.

We then use the finite-reference LFGI score estimator for normalized density evaluation in Bayesian inverse problems.  In regimes where MCMC pilot samples and derivative information are already available, LFGI uses those byproducts to construct a normalized surrogate for the posterior density.  The resulting surrogate supplies information that the MCMC samples alone do not provide: posterior-energy evaluation, model-evidence estimation, and downstream density-based diagnostics.  On a PDE-constrained inverse-problem benchmark, the LFGI surrogate improves posterior-density calibration and sampling diagnostics relative to the other tested score-estimator classes.  Experiments using LFGI with known model evidence check absolute evidence calibration in both Gaussian and non-Gaussian settings.

\end{abstract}

\newpage
\noindent The table records notation used repeatedly across multiple sections.  Quantities used only inside a single construction are defined where they appear.
\begin{table}[H]
\centering
\label{tab:notation}
\begingroup
\footnotesize
\setlength{\tabcolsep}{4pt}
\renewcommand{\arraystretch}{0.92}
\begin{tabular}{@{}p{0.32\linewidth}p{0.64\linewidth}@{}}
\toprule
Notation & Meaning \\
\midrule
$d,I_d$ & ambient dimension and $d\times d$ identity matrix. \\
$X_0,Y_t\in\R^d$; $x,y$ & clean state, OU-corrupted state, and realized values. \\
$\alphat,\gammat$ & OU coefficients, $\alphat=e^{-t}$ and $\gammat=1-e^{-2t}$. \\
$\KOU(y\mid x),\rhoOU$ & OU transition density and OU posterior law $p_0(x\mid Y_t=y)$. \\
$p_0,p_t$ & target density and OU-marginal density. \\
$s_0(x),s_t(y)$ & clean score $\nabla_x\log p_0(x)$ and OU-marginal score $\nabla_y\log p_t(y)$. \\
$H_0(x),H(y,t)$ & observed information $-\nabla_x^2\log p_0(x)$ and its OU-posterior average $\E[H_0(X_0)\mid Y_t=y]$. \\
$\bank_N,\bank_{\rm s},\bank_{\rm g}$ & reference bank and independent score/gate banks, with sizes $N,\Ns,\Ng$. \\
$\widetilde w_i(y,t),w_i(y,t)$ & unnormalized and normalized OU weights attached to reference samples. \\
$b(x;y,t),c(x;t),\delta(x;y,t),\eb(x;y,t)$ & Tweedie signal, target-score signal, disagreement $\delta=c-b$, and Tweedie residual $\eb=b-s_t$. \\
$G(y,t),\Gstar(y,t),\Ghat(y,t)$ & matrix gate, population-optimal gate, and empirical gate. \\
$\Psiop(P)$ & Hessian-average gate map $\alphat^2(\alphat^2I_d+\gammat P)^{-1}$. \\
$\Cdd(y,t),\Cebd(y,t)$ & disagreement covariance $\E[\delta\delta^\top\mid Y_t=y]$ and residual-disagreement moment $\E[\eb\delta^\top\mid Y_t=y]$. \\
$\mathcal R(G;y,t),\Gainstar(y,t)$ & conditional trace risk and population-optimal risk reduction $\mathcal R(0;y,t)-\mathcal R(\Gstar;y,t)$. \\
$\tildep,Z,\qpf$ & unnormalized target density, normalizing constant, and probability-flow density. \\
\bottomrule
\end{tabular}
\endgroup
\end{table}

We first fix the notation and terminology used throughout the paper.  The expectation convention is as follows.  For a fixed query $(y,t)$, meaning a noisy state and diffusion time, $\rhoOU$ denotes the OU posterior law $p_0(\cdot\mid Y_t=y)$.  Expressions such as $\E[f(X_0)\mid Y_t=y]$, $\E[f\mid y,t]$, and $\Cov(f,g\mid y,t)$ are expectations and covariances under this law.  The moment matrices $\Cdd(y,t)$ and $\Cebd(y,t)$ are the corresponding conditional second and cross moments.  Expectations under any other law are subscripted, for example $\E_{p_0}$, $\E_{p_t}$, $\E_{\qpf}$, or $\E_q$; finite-reference averages are denoted by hats and OU weights rather than by $\E$.

We use \emph{population} for the exact object defined by these expectations before replacing them by a finite reference bank; for example, the population-optimal gate minimizes the exact conditional risk under $\rhoOU$.  We use a 'reference bank', or bank for short, denoted $\bank_N=\{X_i\}_{i=1}^N$, to mean a finite collection of reference samples, typically drawn independently from $p_0$, on which the OU weights are evaluated at the query $(y,t)$.  A finite-reference or empirical object replaces exact conditional expectations by weighted averages over such a bank.  In the experiments, score bank and gate bank denote independent reference banks used, respectively, for score-signal averaging and gate construction.

\section{Introduction}
\label{sec:introduction}

\subsection{Score estimation under noise}

Diffusion and score-based samplers generate samples by corrupting an initial target law and then integrating a reverse-time dynamics whose drift depends on the score of the corrupted marginal law \citep{anderson1982reverse,sohldickstein2015nonequilibrium,ho2020denoising,song2021score}.  The object of interest is a target distribution $p_0$ on $\R^d$, together with the Ornstein--Uhlenbeck forward process that maps a clean state $X_0\sim p_0$ to a noisy state $Y_t$.  The marginal of $Y_t$ is denoted by $p_t$, and the score needed by the reverse sampler is
\[
    s_t(y)=\nabla_y\log p_t(y).
\]
The estimator-level problem is to estimate $s_t(y)$ accurately at the states $y$ and noise levels visited by the reverse-time integrator.

The same OU-marginal score also supports the probability-flow density construction used later.  Given a score estimator $\hat{s}_t$, path integration defines the normalized probability-flow density $\qpf$ approximating the target density $p_0$. In the inverse-problem setting below, $\qpf$ is constructed from MCMC pilot byproducts and derivative information and is evaluated on anisotropic non-Gaussian posterior geometry; the probability-flow construction, evidence identities, and Gauss--Newton gates are defined in \cref{sec:pf-density-evaluation-body,sec:gn-gates-body}.

We focus on a setting in which the clean score $s_0(x)=\nabla_x\log p_0(x)$, and eventually its observed information $H_0(x)=-\nabla_x^2\log p_0(x)$, are available or can be evaluated to a desired accuracy.  This includes many scientific sampling problems, Bayesian inverse problems, and controlled synthetic regimes in which evaluating $s_0(x)$ is not the bottleneck \citep{robert2004montecarlo,stuart2010inverse}.  Even when $s_0$ and $H_0$ are available, the score of the corrupted marginal is not directly available.  It is a conditional expectation under the OU posterior law $p_0(x\mid Y_t=y)$.  The quality of a sampler therefore depends on the quality of a conditional score estimator.

\subsection{Two exact identities and one variance problem}

For the Ornstein--Uhlenbeck process, there are two exact conditional-expectation identities for the OU-marginal score $s_t(y)=\nabla_y\log p_t(y)$ at a query $(y,t)$: the Tweedie/denoising identity and the target-score identity \citep{efron2011tweedie,vincent2011connection,debortoli2024tsm}.  For a clean state $x$ and query $(y,t)$, define the Tweedie signal and target-score signal by
\[
    b(x;y,t):=\frac{\alphat x-y}{\gammat},
    \qquad
    c(x;t):=\frac{s_0(x)}{\alphat}.
\]
Then
\[
    s_t(y)=\E[b(X_0;y,t)\mid Y_t=y],
    \qquad
    s_t(y)=\E[c(X_0;t)\mid Y_t=y].
\]
The first is the Tweedie identity. It writes $s_t(y)$ as the conditional average of the OU denoising signal $b(X_0;y,t)=(\alphat X_0-y)/\gammat$.  The second is the target-score identity. It writes the same OU-marginal score $s_t(y)$ as the conditional average of the clean-score signal $c(X_0;t)=s_0(X_0)/\alphat$ under $p_0(x\mid Y_t=y)$.  Both are exact conditional-expectation identities, so neither identity is biased for $s_t(y)$ when the conditional expectation is evaluated exactly.

With the notation convention above, a finite-reference estimator computes the conditional expectations in these two identities by OU-weighted averages over a reference bank.  The difficulty is how to combine the two weighted averages at each query $(y,t)$ so as to minimize the finite-reference variance of the score estimate.  A self-normalized importance sampling (SNIS) estimator of the Tweedie identity can become poorly conditioned at small diffusion times because the residual is scaled by $\gammat^{-1}$ \citep{owen2013mc,liu2001combined,kong1994sequential}.  A target-score-identity estimator can be better conditioned at low noise, but it can carry large variance when the clean score is singular, highly anisotropic, or poorly aligned with the conditional posterior.  Thus the existence of two exact identities does not by itself solve the sampling problem.  It creates a local estimator-selection problem.  At a given state $y$ and time $t$, the goal is to combine these identities to minimize conditional mean-squared score error.

\subsection{Why scalar blending is insufficient}

The Tweedie and target-score identities can be combined before the weighted average over the reference bank is taken.  For a clean state $x$ and query $(y,t)$, write
\[
    \delta(x;y,t):=c(x;t)-b(x;y,t).
\]
Because the Tweedie signal $b$ and target-score signal $c$ have the same conditional mean, $\E[\delta(X_0;y,t)\mid Y_t=y]=0$.  A blended score signal is therefore
\[
    z_G(x;y,t)
    :=
    b(x;y,t)+G(y,t)\delta(x;y,t),
\]
where $G(y,t)$ is fixed after conditioning on the query $(Y_t=y,t)$.  Averaging $z_G(X_i;y,t)$ with the same OU weights used for the Tweedie and TSI estimators gives a score estimator with the same conditional mean as either identity.  The gate $G$ changes variance, not the target conditional mean.  The choices $G=0$ and $G=I_d$ recover the Tweedie and TSI signals.  A scalar blend restricts $G(y,t)=g(y,t)I_d$, so every direction receives the same fraction of the correction from Tweedie toward TSI.  We denote this restricted estimator by \BLEND{}.  Allowing a matrix gate gives \MBLEND{}, where the correction can depend on direction.

The singular Gaussian example shows why the scalar restriction is too restrictive.  If
\[
    p_0=\N(0,P^{-1}),
    \qquad
    P=\diag(\lambda_1,\ldots,\lambda_d),
\]
then the population-optimal matrix gate is diagonal in this basis and satisfies
\[
    e_j^\top \Gstar(y,t)e_j
    =
    \psi_t(\lambda_j)
    :=
    \frac{\alphat^2}{\alphat^2+\gammat\lambda_j}.
\]
Thus $\psi_t(\lambda_j)$ is the Gaussian attenuation applied to the TSI--Tweedie correction in coordinate $j$.  A scalar blend would force all diagonal entries to equal one number $g(y,t)$.  A scalar coefficient cannot match all entries when $P$ is ill-conditioned.  High-curvature directions with large $\lambda_j$ have small $\psi_t(\lambda_j)$ and stay closer to Tweedie, whereas low-curvature directions with small $\lambda_j$ have $\psi_t(\lambda_j)$ closer to one and move closer to TSI.  We therefore optimize over matrix gates $G(y,t)$.

The direction-dependent mismatch also matters for density evaluation.  A scalar-gated score field used in the probability-flow estimator can accumulate path-likelihood error along the probability-flow trajectory.  The mismatch between scalar attenuation and the eigenvalue-dependent filter $\psi_t(\lambda_j)$ can therefore produce large error in the normalized density and in evidence diagnostics built from that density.  See the Darcy-flow density benchmark in \cref{subsec:density-experiments}.

\subsection{Main contribution}

We formulate Tweedie--TSI blending as a conditional risk minimization problem for a matrix gate.  Here a gate is a measurable matrix coefficient $G(y,t)\in\R^{d\times d}$ multiplying the zero-mean disagreement signal between the Tweedie and TSI score signals.  Since this disagreement signal has conditional mean zero, every such gate leaves the conditional mean equal to $s_t(y)$, and the gate-selection problem is variance and risk control.  The optimal gate can be constructed in two equivalent ways, depending on which target information is available.  With target scores and reference samples, the direct moment construction expresses the gate through conditional moments of the score disagreement: the cross-covariance between the target-score signal and the disagreement signal, and the conditional covariance of the disagreement signal itself.  The finite estimator replaces those conditional covariances by self-normalized weighted averages.  When target Hessians are also available, OU Fisher identities rewrite the same second moments through the conditional average of the target Hessian, so the same gate can be obtained by estimating that Hessian average and applying the associated matrix-resolvent map.  The two constructions have different finite-reference errors: centered primal regression estimates a conditional covariance ratio, while LFGI estimates a conditional Hessian average before inversion.

The application is normalized density evaluation rather than merely another sampling benchmark.  In the inverse-problem regime where scores, Hessian or Gauss--Newton information, and posterior pilot samples are available, additional MCMC from the pilot bank is a natural sampling baseline.  LFGI instead uses those pilot byproducts to build a probability-flow score estimator that is held fixed after construction and has closed-form divergence, yielding an exactly normalized surrogate for the target posterior density.  Absolute evidence identities are exact only when this probability-flow surrogate equals the target density; away from that ideal case, evidence quantities built from the surrogate are calibration diagnostics.  The Darcy-flow density benchmark in \cref{subsec:density-experiments} reports posterior-energy calibration (\cref{app:density-metrics}, Def.~\ref{def:central-density}) and effective-sample-size overlap diagnostics (Def.~\ref{def:evidence-diagnostics}), while the analytic Gaussian, shared-geometry mixture, and misaligned-curvature mixture calibration targets in \cref{subsec:known-z-calibration} give known-normalization evidence checks.

\section{Related Work}
\label{sec:related-work}

Diffusion and score-based generative models sample by reversing a noising process whose drift depends on a time-dependent score \citep{sohldickstein2015nonequilibrium,ho2020denoising,song2021score}.  For unnormalized targets, recent diffusion-path samplers estimate intermediate scores by Monte Carlo, inner MCMC, rejection/reference procedures, or auxiliary-particle schemes rather than by fully amortized score training \citep{huang2024rdmc,grenioux2024slips,he2024zeroth,noble2025learned,delmoral2006smc,wu2025rdsmc,young2026dpsmc}.  We use those sampler-level frameworks as context and isolate the local OU-posterior estimator-design problem. We then ask which matrix gate applied to the two exact score identities minimizes conditional trace risk.

The identities blended here are the Tweedie/denoising identity \citep{efron2011tweedie,vincent2011connection} and the target-score identity \citep{debortoli2024tsm}; they induce the two finite-reference score estimators of the OU-marginal score that scalar or matrix schedules later blend. Prior and concurrent work combines the resulting score estimators through scalar, diagonal, or matrix schedules estimated globally or along path marginals \citep{debortoli2024tsm,he2025reversekl,cordero2025sampling,kahouli2025cvsm,ko2025latent,young2026dpsmc}.  Our contribution is the corresponding pointwise analysis of matrix gates.  For each $(y,t)$, the zero-mean Tweedie--TSI disagreement is an OU-posterior-score control variate.  The risk-optimal gate can be expressed through conditional moments of this score disagreement; Fisher--Stein identities then rewrite those same moments through the conditional Hessian average $H(y,t)$, yielding an equivalent curvature-based matrix-resolvent gate.  Although second-order information is classical in Laplace approximations and Hessian-informed MCMC \citep{tierney1986laplace,robert2004montecarlo,girolami2011riemann,Bui-ThanhGirolami14,MartinWilcoxBursteddeEtAl12}, here curvature is not a replacement target density; it is the analytic representation of the risk-optimal gate between identities with the same conditional mean.

A parallel line learns time-dependent scores, drifts, or path controls for unnormalized sampling \citep{akhound2024idem,phillips2024pdds,vargas2024transport,richter2024improved,chen2025sequential,zhang2022pathintegral,havens2025adjoint}.  The estimator studied here is non-amortized and reference-based: the bank approximates OU conditional expectations of known target-side quantities rather than fitting a global score, kernel field, or transport map \citep{hyvarinen2005score,sriperumbudur2017infinite,arbel2018kcef,wenliang2019deepkef,zhou2020nonparametricscore,liu2016svgd,liu2016ksd,chwialkowski2016kernelgof,korba2021ksddescent,arbel2019mmdflow}.  The density-evaluation application connects to continuous normalizing flows, probability-flow likelihoods, and Bayesian evidence estimators such as bridge sampling, annealed importance sampling (AIS), thermodynamic integration, nested sampling, and transport/flow surrogates \citep{chen2018neuralode,grathwohl2019ffjord,song2021score,song2021maximum,meng1996bridge,neal2001annealed,gelman1998normalizing,skilling2006nested}.  LFGI uses target derivatives and Hessian or Gauss--Newton information evaluated at target reference samples to construct a self-normalized probability-flow surrogate for the target density.

\section{Theory: From Exact Identities to Optimal Matrix Gates}

\label{sec:theory}

The presentation follows the score-based sampling viewpoint: a forward diffusion defines OU marginals $p_t$, and a reverse sampler requires the score $s_t=\nabla\log p_t$ \citep{anderson1982reverse,song2021score}.  The only estimator-level question is how to compute this score accurately from a reference bank and evaluations of $s_0$ and $H_0$.

\subsection{Score-based sampling with the Ornstein--Uhlenbeck process}
\label{subsec:ou-process}

The Ornstein--Uhlenbeck process provides the forward corruption used throughout the analysis.  The transition kernel is Gaussian.  The marginal of $Y_t$ is an explicit Gaussian convolution of $p_0$, and the posterior over clean states given a noisy query point is available up to normalization.

\begin{definition}[OU forward process]
\label[definition]{def:ou-process}
Let $X_0\sim p_0$ and let $\xi\sim\N(0,I_d)$ be independent.  Define
\begin{equation*}
    Y_t=\alphat X_0+\sqrt{\gammat}\,\xi,
    \qquad
    \alphat=e^{-t},
    \qquad
    \gammat=1-e^{-2t}.
\end{equation*}
The transition density is
\begin{equation*}
    \KOU(y\mid x)
    =
    (2\pi\gammat)^{-d/2}
    \exp\left(
        -\frac{\nor{y-\alphat x}^2}{2\gammat}
    \right).
\end{equation*}
The marginal of $Y_t$ and its score are
\begin{equation}
\label{eq:pt-score}
    p_t(y)=\int \KOU(y\mid x)p_0(x)\,\dd x,
    \qquad
    s_t(y)=\nabla_y\log p_t(y).
\end{equation}
\end{definition}

Equivalently, $Y_t$ solves the forward OU SDE
\begin{equation*}
    \dd Y_t=-Y_t\,\dd t+\sqrt{2}\,\dd W_t,
    \qquad
    Y_0\sim p_0.
\end{equation*}
The reverse-time sampler associated with this forward process requires $s_t(y)$ along its trajectory \citep{anderson1982reverse,song2021score}.  Written for integration from $\tmax$ down to $\tmin$, the exact reverse SDE and its probability-flow ODE are
\begin{align}
\label{eq:ou-reverse-sde}
    \dd \bar Y_t
    &=
    \{-\bar Y_t-2s_t(\bar Y_t)\}\,\dd t
    +\sqrt{2}\,\dd \bar W_t,\\
\frac{\dd \bar y_t}{\dd t}
    &=
    -\bar y_t-s_t(\bar y_t).\notag
\end{align}
Thus samples are obtained by replacing $s_t$ by an estimator $\widehat s(\cdot,t)$ in \cref{eq:ou-reverse-sde} and discretizing the resulting reverse dynamics.  Density evaluation uses the same estimated score in the inverse probability-flow path: if $z_{\tmin}=x$ and $\dot z_t=-z_t-\widehat s(z_t,t)$, then
\begin{equation}
\label{eq:pf-density-intro}
    \log \qpf(x)
    =
    \log\phibase(z_{\tmax})
    -
    \int_{\tmin}^{\tmax}
    \{d+\nabla\!\cdot\widehat s(z_t,t)\}\,\dd t.
\end{equation}
The estimator-level objective in the rest of the paper is to estimate $s_t$ accurately enough that these reverse-SDE and probability-flow substitutions give useful samples and density values.  The formula for $p_t$ in \cref{eq:pt-score} is usually not tractable enough to differentiate in $y$ directly, so we work with conditional expectations under the forward posterior.

\begin{definition}[OU posterior]
\label[definition]{def:ou-posterior}
For fixed $(y,t)$ with $p_t(y)>0$, define
\begin{equation}
\label{eq:rho-ou}
    \rhoOU(x)
    =
    p_0(x\mid Y_t=y)
    =
    \frac{\KOU(y\mid x)p_0(x)}{p_t(y)}.
\end{equation}
\end{definition}

The posterior $\rhoOU$ is the conditional density $p_0(x\mid Y_t=y)$ with respect to which all moments such as $\E[f(X_0)\mid Y_t=y]$ are taken.  In finite-reference implementations, such a moment is approximated by normalized weights proportional to $\KOU(y\mid X_i)$ applied to reference samples $X_i\sim p_0$ \citep{owen2013mc,liu2001combined,kong1994sequential}.

\begin{assumption}[Regularity for the OU identities and LFGI theorem]
\label[assumption]{ass:regularity}
The target density $p_0$ is positive and twice continuously differentiable on its support.  The conditional law $\rhoOU$ has sufficient decay for the integration-by-parts steps used in the target-score identity \eqref{eq:tsi-identity} and in \cref{lem:Cebd-universal,lem:fisher-H}.  All conditional moments in \eqref{eq:tweedie-identity}, \eqref{eq:tsi-identity}, and \cref{lem:Cebd-universal,lem:fisher-H,thm:lfgi} are finite.  Whenever one of those results uses an inverse, the relevant matrix is assumed nonsingular; otherwise the corresponding Moore--Penrose statement is used.
\end{assumption}

\subsection{The Tweedie and target-score identities}
\label{subsec:twd-tsi}

The two identities used in this paper are best viewed as two score signals whose conditional mean is the same OU-marginal score $s_t(y)$.  The first signal depends only on the OU residual $y-\alphat x$.  The second signal depends on the target score $s_0(x)$.

\begin{definition}[Tweedie and TSI score signals]
\label[definition]{def:bcd}
Define
\begin{equation*}
\begin{aligned}
    b(x;y,t)
    &:=
    -\frac{y-\alphat x}{\gammat}
    =
    \frac{\alphat x-y}{\gammat},
    \\
    c(x;t)
    &:=
    \frac{s_0(x)}{\alphat},
    \qquad
    s_0(x)=\nabla_x\log p_0(x),
    \\
    \delta(x;y,t)
    &:=
    c(x;t)-b(x;y,t).
\end{aligned}
\end{equation*}
We call $b$ the Tweedie signal, $c$ the target-score (TSI) signal, and $\delta$ the disagreement signal.
\end{definition}

The first named identity is the \emph{Tweedie identity}, also called the denoising score identity for the OU corruption \citep{efron2011tweedie,vincent2011connection}:
\begin{equation}
\label{eq:tweedie-identity}
    s_t(y)
    =
    \E[b(X_0;y,t)\mid Y_t=y].
\end{equation}
It expresses the time-$t$ score as the posterior average of the Tweedie signal $b(X_0;y,t)=(\alphat X_0-y)/\gammat$.

The second named identity is the \emph{target-score identity} \citep{debortoli2024tsm}:
\begin{equation}
\label{eq:tsi-identity}
    s_t(y)
    =
    \E[c(X_0;t)\mid Y_t=y]
    =
    \frac{1}{\alphat}
    \E[s_0(X_0)\mid Y_t=y].
\end{equation}
It averages the clean-score signal $c(X_0;t)=s_0(X_0)/\alphat$ under the OU posterior $p_0(x\mid Y_t=y)$.  It is an identity for the same quantity as Tweedie, namely the OU-marginal score $s_t(y)$, but the finite-reference estimators average different score signals: $c(X_0;t)$ for TSI and $b(X_0;y,t)$ for Tweedie.

Subtracting \eqref{eq:tweedie-identity} from \eqref{eq:tsi-identity} gives the zero-mean disagreement identity
\begin{equation}
\label{eq:delta-zero-mean}
    \E[\delta(X_0;y,t)\mid Y_t=y]=0.
\end{equation}
This zero-mean property is the algebraic reason that blending can reduce variance without changing the target conditional mean, as in standard control-variate constructions \citep{owen2013mc,robert2004montecarlo}.  Any conditional multiple of $\delta$ is a valid control variate.

\subsection{Posterior-score control variate structure}

The disagreement signal satisfies the posterior-score relation displayed in \cref{lem:posterior-score-disagreement}: $\delta(x;y,t)=\alphat^{-1}\nabla_x\log p_0(x\mid Y_t=y)$.

\begin{lemma}[The disagreement is the OU posterior score]
\label[lemma]{lem:posterior-score-disagreement}
Let
\[
    q(x;y,t):=\nabla_x\log\rhoOU(x).
\]
Then
\begin{equation*}
    q(x;y,t)=\alphat\delta(x;y,t).
\end{equation*}
Equivalently,
\[
    \delta(x;y,t)
    =
    \frac{1}{\alphat}
    \nabla_x\log p_0(x\mid Y_t=y).
\]
\end{lemma}

Proof deferred to \cref{app:proof-lem-posterior-score-disagreement}.

The preceding lemma gives the control-variate interpretation directly.  The disagreement $\delta=c-b$ has conditional mean zero and is proportional to the OU posterior score.  Thus adding $G\delta$ to the Tweedie signal is a matrix control variate.   Because $\delta=\alphat^{-1}\nabla_x\log p_0(x\mid Y_t=y)$, its covariance $\Cdd=\E[\delta\delta^\top\mid Y_t=y]$ is the OU-posterior Fisher information up to the factor $\alphat^{-2}$ \citep{stein1972bound,vaart1998asymptotic}.

\subsection{Matrix gates}

We now replace the scalar coefficient $g(y,t)$ by a matrix $G(y,t)$ acting on the disagreement signal $\delta(x;y,t)$.  This is the matrix version of scalar estimator blending and control-variate score scheduling \citep{debortoli2024tsm,kahouli2025cvsm,young2026dpsmc}.

\begin{definition}[Gated score signal]
\label[definition]{def:gated-score}
For a measurable matrix gate $G(y,t)\in\R^{d\times d}$, define
\begin{equation*}
    z_G(x;y,t)
    =
    b(x;y,t)+G(y,t)\delta(x;y,t).
\end{equation*}
\end{definition}

The conditional-mean identity is immediate from the zero conditional mean of the disagreement.  For every measurable matrix $G(y,t)$ that is fixed after conditioning on $(Y_t=y,t)$,
\begin{equation}
\label{eq:measurable-gate-target}
    \E[z_G(X_0;y,t)\mid Y_t=y]
    =
    s_t(y).
\end{equation}
This follows from the two identities above because $\E[\delta(X_0;y,t)\mid Y_t=y]=0$.  Thus the choice of $G$ changes the conditional risk, not the conditional mean.

\section{The Matrix-Gate Risk Problem}
\label{sec:operator-risk}

With the conditional mean fixed, the gate is chosen by local variance control.  At fixed $(y,t)$, choose the matrix gate by minimizing the conditional trace risk $R(G;y,t)$; the first-order condition is the normal equation $G\Cdd+\Cebd=0$.  All quantities are spatially varying in $(y,t)$, so the optimal gate can adapt simultaneously to position, noise level, and anisotropic posterior geometry.

\subsection{Conditional risk and normal equations}

\begin{definition}[Centered Tweedie residual]
\label[definition]{def:eb}
Define
\begin{equation*}
    \eb(x;y,t)=b(x;y,t)-s_t(y).
\end{equation*}
\end{definition}

The residual $\eb$ is the deviation of the Tweedie signal from its conditional mean.  By \eqref{eq:delta-zero-mean}, $\delta$ has zero conditional mean, so the gate chooses a control variate that cancels as much of $\eb$ as possible.

\begin{definition}[Local conditional trace risk]
\label[definition]{def:local-risk}
At fixed $(y,t)$, with expectation taken over $X_0\sim\rhoOU$, define
\begin{equation*}
    \mathcal R(G;y,t)
    =
    \E\left[
        \nor{\eb(X_0;y,t)+G(y,t)\delta(X_0;y,t)}^2
        \mid Y_t=y
    \right].
\end{equation*}
Define the conditional moment matrices
\begin{equation*}
    \Cdd(y,t)=\E[\delta\delta^\top\mid Y_t=y],
    \qquad
    \Cebd(y,t)=\E[\eb\delta^\top\mid Y_t=y].
\end{equation*}
\end{definition}

The matrix $\Cdd$ measures the conditional second moment of the available control variate, while $\Cebd$ measures its cross moment with the centered Tweedie residual $\eb$.  These two matrices define the ordinary least-squares normal equation for the best matrix coefficient.

\begin{proposition}[Matrix-gate normal equation]
\label[proposition]{prop:normal-equation}
At fixed $(y,t)$, the unconstrained minimizer
\[
    \Gstar(y,t)\in\argmin_{G\in\R^{d\times d}}\mathcal R(G;y,t)
\]
satisfies
\begin{equation}
\label{eq:normal-equation}
    \Gstar \Cdd+\Cebd=0.
\end{equation}
If $\Cdd$ is nonsingular, then
\begin{equation*}
    \Gstar=-\Cebd \Cdd^{-1}.
\end{equation*}
\end{proposition}

Proof deferred to \cref{app:proof-prop-normal-equation}.

\begin{definition}[Normal residual]
\label[definition]{def:normal-residual}
Define
\begin{equation*}
    \OmegaG(G;y,t)
    :=
    \E[(\eb+G\delta)\delta^\top\mid Y_t=y]
    =
    G\Cdd+\Cebd.
\end{equation*}
\end{definition}

The normal residual $\Omega_G$ in \cref{def:normal-residual} measures the cross moment between the remaining gated score residual and the disagreement signal.  It is zero exactly when the remaining gated residual is orthogonal to every direction in the conditional span of $\delta$, equivalently every direction with nonzero contribution under $\Cdd$.

\begin{proposition}[Residual orthogonality and excess risk]
\label[proposition]{prop:residual-risk}
For any solution $\Gstar$ of \cref{eq:normal-equation},
\begin{equation*}
    \OmegaG(\Gstar;y,t)=0.
\end{equation*}
Moreover, for any matrix $G$,
\begin{equation}
\label{eq:excess-risk}
    \mathcal R(G;y,t)-\mathcal R(\Gstar;y,t)
    =
    \tr\left(
        (G-\Gstar)\Cdd(G-\Gstar)^\top
    \right).
\end{equation}
If $\Cdd$ is nonsingular, then
\begin{equation*}
    \mathcal R(G;y,t)-\mathcal R(\Gstar;y,t)
    =
    \tr\left(
        \OmegaG(G;y,t)\Cdd^{-1}\OmegaG(G;y,t)^\top
    \right).
\end{equation*}
\end{proposition}

Proof deferred to \cref{app:proof-prop-residual-risk}.

By \cref{eq:excess-risk}, a gate perturbation accrues excess score risk through the norm weighted by the disagreement covariance $\Cdd$.  Gate error matters only through $\nor{(G-\Gstar)\Cdd^{1/2}}_F^2$; eigendirections of $\Cdd$ with negligible weight in this norm do not need to be learned accurately.

\begin{definition}[Population-optimal risk reduction]
\label[definition]{def:population-risk-reduction}
At fixed $(y,t)$, \emph{population} means that expectations are taken under the full conditional law $\rhoOU=p_0(\cdot\mid Y_t=y)$, rather than over a finite reference bank.  Define the population-optimal risk reduction over the no-gate Tweedie baseline by
\begin{equation*}
    \Gainstar(y,t)
    :=
    \mathcal R(0;y,t)-\mathcal R(\Gstar;y,t).
\end{equation*}
For an implementable gate $\widehat G$, define its plug-in excess and captured population-optimal fraction by
\begin{equation*}
    \mathcal E(\widehat G;y,t)
    :=
    \mathcal R(\widehat G;y,t)-\mathcal R(\Gstar;y,t),
    \qquad
    \mathrm{cap}(\widehat G;y,t)
    :=
    1-\frac{\mathcal E(\widehat G;y,t)}{\Gainstar(y,t)}
\end{equation*}
whenever $\Gainstar(y,t)>0$.
\end{definition}

\Cref{eq:excess-risk} shows that the finite-reference criterion is the excess score risk $\mathcal E(\widehat G;y,t)$, or equivalently the population-risk-reduction condition $\mathcal E(\widehat G;y,t)\le \eta\,\Gainstar(y,t)$, not ordinary matrix-norm closeness to $\Gstar$.  The exact excess risk is
\[
    \mathcal E(\widehat G;y,t)
    =
    \nor{(\widehat G-\Gstar)\Cdd(y,t)^{1/2}}_F^2.
\]
Thus an estimated gate is useful precisely when its error is small in the eigendirections weighted by $\Cdd$ in the excess-risk norm.  Equivalently, $\widehat G$ captures at least a $(1-\eta)$ fraction of the population-optimal risk reduction if
\[
    \mathcal E(\widehat G;y,t)
    \le
    \eta\,\Gainstar(y,t).
\]
The finite-reference bounds below therefore control the same quantity: the excess score risk $\mathcal E(\widehat G;y,t)$ in \cref{eq:excess-risk}.

\subsection{Scalar, diagonal, and full matrix classes}

The admissible gate class determines how much anisotropic structure can be expressed.  A scalar gate gives one global interpolation coefficient.  A diagonal gate gives coordinate-wise interpolation in a fixed basis.  A full matrix gate can rotate and scale the disagreement signal in any direction.

\begin{definition}[Gate classes]
\label[definition]{def:gate-classes}
Define
\[
    \mathcal G_{\rm sc}:=\{gI_d:g\in\R\},
    \qquad
    \mathcal G_{\rm diag}:=\{G:G \text{ diagonal in a fixed basis}\},
    \qquad
    \mathcal G_{\rm op}:=\R^{d\times d}.
\]
\end{definition}

For these nested gate classes, the corresponding normal equations reduce to familiar scalar and coordinatewise formulas.  If $\tr \Cdd>0$, the optimal scalar gate $G=g_\star I_d$ has
\[
    g_\star=-\frac{\tr \Cebd}{\tr \Cdd}.
\]
In a fixed coordinate basis, if $(\Cdd)_{jj}>0$, the optimal diagonal gate $G=\diag(g_1,\ldots,g_d)$ satisfies
\[
    g_j=-\frac{(\Cebd)_{jj}}{(\Cdd)_{jj}}.
\]
Because the classes are nested, the best full matrix gate can only improve the conditional trace risk relative to the best diagonal gate, and the best diagonal gate can only improve it relative to the best scalar gate:
\[
    \inf_{G\in\mathcal G_{\rm op}}\mathcal R(G)
    \le
    \inf_{G\in\mathcal G_{\rm diag}}\mathcal R(G)
    \le
    \inf_{G\in\mathcal G_{\rm sc}}\mathcal R(G).
\]
The Gaussian case gives the simplest example showing that this extra expressivity affects the risk minimizer.

\subsection{Spatially uniform control-variate gates}
\label{subsec:spatially-uniform-gates}

The risk problem above is pointwise in $(y,t)$, so its population minimizer may be a spatially varying matrix gate $G_\star(y,t)$.  Several existing control-variate score estimators instead use spatially uniform schedules: the blending coefficient depends on diffusion time, but not on the query location.  We include these schedules as external baselines for the experiments below, using the scalar spatially uniform control-variate schedule from control-variate score matching \citep{kahouli2025cvsm} and the DPSMC spatially uniform matrix control-variate schedule \citep[Sec.~3.4, Props.~2--3]{young2026dpsmc}, both written here in the target-score-identity-to-Tweedie convention.  The formulas in this subsection are the expected-variance control-variate schedules derived in those references, specialized to the OU notation of \cref{sec:theory} and converted to the manuscript's Tweedie--TSI gate convention.

Let
\[
    I_\pi := \E_{p_0}[s_0(X)s_0(X)^\top]
\]
be the target-score second moment.  In the convention where the estimator is written as a correction from the target-score identity toward Tweedie, the spatially uniform scalar schedule is
\begin{equation}
\label{eq:s-unif-sc}
\begin{aligned}
    a_{\rm sc}(t)
    &=
    \frac{\gammat\operatorname{Tr}(I_\pi)}{\alphat^2 d+\gammat\operatorname{Tr}(I_\pi)},\\
    s_{\rm unif\text{-}sc}(y,t)
    &=s_{\TSI}(y,t)+a_{\rm sc}(t)\{s_{\TWD}(y,t)-s_{\TSI}(y,t)\}.
\end{aligned}
\end{equation}
The corresponding spatially uniform matrix schedule is
\begin{equation}
\label{eq:s-unif-mat}
\begin{aligned}
    A_{\rm mat}(t)
    &=
    \gammat I_\pi(\alphat^2 I_d+\gammat I_\pi)^{-1},\\
    s_{\rm unif\text{-}mat}(y,t)
    &=s_{\TSI}(y,t)+A_{\rm mat}(t)\{s_{\TWD}(y,t)-s_{\TSI}(y,t)\}.
\end{aligned}
\end{equation}
This convention is the reverse of the spatially varying gate notation $z_G=b+G(c-b)$ used above: the scalar schedule corresponds to $G=(1-a_{\rm sc})I_d$, and the matrix schedule corresponds to $G=I_d-A_{\rm mat}$.  Thus the formulas are algebraically the same control variates, expressed as corrections from TSI toward Tweedie rather than from Tweedie toward TSI.
The finite-reference versions replace $I_\pi$ by the corresponding empirical target-score second moment on the gate bank.  These gates are statistically easy spatially uniform moment-estimation problems.  They can capture spatially uniform anisotropy, but they cannot represent spatially varying or mode-dependent curvature.  The spatially varying LFGI gate replaces the spatially uniform moment $I_\pi$ by the conditional observed-information average $H(y,t)$ inside the resolvent map.  Thus the uniform baselines separate the benefit of matrix-valued blending from the additional benefit of spatially varying Hessian-informed gating.

\subsection{Why matrix gates are needed: exact Gaussian calculation}
\label{subsec:gaussian-gate-calculation}

The Gaussian case reveals why scalar blends become problematic when the target density has anisotropic curvature.  It removes nonlinear target effects and leaves only the spectral structure of the optimal gate.  If even this model requires direction-dependent attenuation, then scalar blending cannot be the risk-minimizing solution within the scalar class.

\begin{definition}[Singular Gaussian family]
\label[definition]{def:singular-gaussian}
Let
\[
    p_0^\varepsilon=\N(0,P_\varepsilon^{-1}),
    \qquad
    P_\varepsilon=U\diag(\lambda_1^\varepsilon,\ldots,\lambda_d^\varepsilon)U^\top.
\]
A singular limit is a family with
\[
    \kappa(P_\varepsilon)
    =
    \frac{\lambda_{\max}(P_\varepsilon)}
         {\lambda_{\min}(P_\varepsilon)}
    \to\infty
\]

or with at least one $\lambda_j^\varepsilon\to\infty$.
\end{definition}

\begin{proposition}[Gaussian optimal gate]
\label[proposition]{prop:gaussian-gate}
For $p_0=\N(m,P^{-1})$ with $P\succ0$, the optimal full matrix gate is independent of $y$ and equals
\begin{equation*}
    \Gstar(t)
    =
    \alphat^2(\alphat^2 I_d+\gammat P)^{-1}.
\end{equation*}
If $Pu_j=\lambda_j u_j$, then
\begin{equation*}
    \Gstar(t)u_j
    =
    \psi_t(\lambda_j)u_j,
    \qquad
    \psi_t(\lambda)=
    \frac{\alphat^2}{\alphat^2+\gammat\lambda}
    =
    \frac{1}{1+\lambdat\lambda},
    \qquad
    \lambdat:=\frac{\gammat}{\alphat^2}.
\end{equation*}
\end{proposition}

This is also the constant-Hessian Gaussian special case of \cref{thm:lfgi}, but the proof in \cref{app:proof-prop-gaussian-gate} uses direct Gaussian conditioning and the normal equation rather than invoking LFGI.

\begin{corollary}[Scalar blending cannot resolve anisotropic stiffness]
\label[corollary]{cor:scalar-failure}
Suppose $P$ has at least two eigenvalues $\lambda_i\neq\lambda_j$ and both corresponding eigendirections carry nonzero disagreement variance.  Then no scalar gate $gI_d$ can equal the Gaussian optimal gate unless
\[
    \psi_t(\lambda_i)=\psi_t(\lambda_j).
\]
In particular, scalar blending cannot simultaneously match high- and low-curvature directions in the singular Gaussian limit.
\end{corollary}

Proof deferred to \cref{app:proof-cor-scalar-failure}.

In this Gaussian case, the population risk minimizer is the spectral filter $\psi_t$.  For example, when $\gammat/\alphat^2=1$, a direction with $\lambda=100$ has attenuation $1/101$, while a direction with $\lambda=0.01$ has attenuation about $0.99$.  A single scalar coefficient cannot reproduce both values.  The Gaussian example is not a failure case for centered primal regression; \cref{prop:gaussian-centered-cancellation} states the special centered-regression cancellation for this case.

The remaining question is whether the full matrix gate can be estimated in the regimes where it is needed.  \Cref{sec:primal-regression} specifies the centered primal matrix estimator used in \cref{sec:validation-plan} and isolates the non-Gaussian residual that makes this estimator difficult at finite reference count.

\section{Centered Primal Matrix Regression and the Non-Gaussian Failure Regime}
\label{sec:primal-regression}

The Gaussian spectral filter shows that scalar gates can be structurally insufficient, but it does not decide whether the full matrix gate is easy to estimate.  The Hessian-free construction estimates the normal-equation quotient directly:
\[
    \Gstar=-\Cebd \Cdd^{-1}.
\]
Its finite-reference version replaces the conditional moments by centered weighted OU-SNIS moments, regressing $b_i-\bar b$ on $\delta_i-\bar\delta$.  This construction has an exact Gaussian cancellation, but in non-Gaussian targets its error is controlled by the residual--disagreement cross moment in \cref{eq:centered-regression-residual-risk}.

\subsection{Centered primal matrix gate estimation}
\label{subsec:centered-primal-gate}

Fix $(y,t)$ and let $\{(x_i,w_i)\}_{i=1}^{N}$ be a finite OU-SNIS approximation to the posterior $\rhoOU$.  Define the score signals
\[
    b_i=b(x_i;y,t),
    \qquad
    c_i=c(x_i;t),
    \qquad
    \delta_i=c_i-b_i,
\]
and weighted means
\[
    \bar b=\sum_i w_i b_i,
    \qquad
    \bar\delta=\sum_i w_i \delta_i.
\]
The centered weighted-regression gate used in the validation benchmarks of \cref{sec:validation-plan} is
\begin{equation}
\label{eq:centered-primal-gate}
    \widehat G_{\rm cen}
    =
    -\widehat C_{\eb\delta}\,\bigl(\widehat C_{\delta\delta}+\lambdaridge I_d\bigr)^{-1},
\end{equation}
where $\lambdaridge\ge0$ is a Tikhonov regularization parameter and
\begin{equation}
\label{eq:centered-covariances}
    \widehat C_{\eb\delta}
    =
    \sum_i w_i (b_i-\bar b)(\delta_i-\bar\delta)^\top,
    \qquad
    \widehat C_{\delta\delta}
    =
    \sum_i w_i (\delta_i-\bar\delta)(\delta_i-\bar\delta)^\top.
\end{equation}
With exact conditional expectations, $\E[b\mid y,t]=s_t(y)$ and $\E[\delta\mid y,t]=0$, so
\[
    \E[\eb\delta^\top\mid y,t]=\Cebd,
    \qquad
    \Cov(\delta,\delta\mid y,t)=\Cdd.
\]
Thus the centered regression target is the population normal-equation gate $\Gstar=-\Cebd \Cdd^{-1}$.  Centering does not change the exact conditional objective.  It is the natural weighted-sample implementation because $s_t(y)$ is unavailable and must be replaced by the query-wise Tweedie mean $\bar b$.

We call this query-wise centered weighted least-squares gate the centered primal matrix blend, or Matrix Blend.  The independent score/gate-bank score estimator used in the validation benchmarks of \cref{sec:validation-plan} is defined in \cref{sec:finite-reference}, after the finite-reference bank convention has been introduced.

\begin{theorem}[Risk-weighted error identity for centered primal regression]
\label[theorem]{thm:direct-plugin-ill-conditioning}
Fix $(y,t)$ and assume $\Cdd\succ0$.  Let
\[
    \widehat C_{\delta\delta}=\Cdd+\Ddd,
    \qquad
    \widehat C_{\eb\delta}=\Cebd+\Debd,
    \qquad
    \widehat G_{\rm cen}=-\widehat C_{\eb\delta}\widehat C_{\delta\delta}^{-1},
\]
with $\widehat C_{\delta\delta}$ invertible and with the ridge omitted from the displayed algebra.  Then
\begin{equation}
\label{eq:direct-moment-risk-exact}
    \mathcal R(\widehat G_{\rm cen};y,t)-\mathcal R(\Gstar;y,t)
    =
    \nor{
        (\Debd+\Gstar\Ddd)
        \widehat C_{\delta\delta}^{-1}\Cdd^{1/2}
    }_F^2.
\end{equation}
If, in addition,
\[
    r_{\delta\delta}:=\nor{\Cdd^{-1/2}\Ddd \Cdd^{-1/2}}_{\op}<1,
\]
then
\begin{equation}
\label{eq:direct-moment-risk-bound}
    \mathcal R(\widehat G_{\rm cen};y,t)-\mathcal R(\Gstar;y,t)
    \le
    \frac{1}{(1-r_{\delta\delta})^2}
    \nor{(\Debd+\Gstar\Ddd)\Cdd^{-1/2}}_F^2.
\end{equation}
Consequently, if the right-hand side of \cref{eq:direct-moment-risk-bound} is at most $\eta\Gainstar(y,t)$, then
\[
    \mathcal R(0;y,t)-\mathcal R(\widehat G_{\rm cen};y,t)
    \ge
    (1-\eta)\Gainstar(y,t).
\]

\end{theorem}

Proof deferred to \cref{app:proof-thm-direct-plugin-ill-conditioning}.

The risk-reduction inequality in \cref{thm:direct-plugin-ill-conditioning} gives the following sample-size condition for the centered-primal gate.  Here, and in the auxiliary bimodal calculation in \cref{subsec:misaligned-bimodal-gmm,subsec:bimodal-primal-failure-scale}, $N_{\eff}(y,t)$ denotes the effective sample size of the query-specific OU-SNIS weights; for iid conditional samples $N_{\eff}=N$, while for normalized weights $w_i(y,t)$ we use the standard inverse-squared-weight diagnostic $N_{\eff}(y,t)=(\sum_i w_i(y,t)^2)^{-1}$ \citep{kong1994sequential,owen2013mc,martino2017effective}.  If, with probability at least $1-\delta$,
\[
    r_{\delta\delta}\le \frac12,
    \qquad
    \nor{(\Debd+\Gstar\Ddd)\Cdd^{-1/2}}_F^2
    \le
    \frac{V_{\rm cen}(y,t)\log(1/\delta)}{N_{\eff}(y,t)},
\]
then $\widehat G_{\rm cen}$ captures at least a $(1-\eta)$ fraction of the population-optimal risk reduction provided
\[
    N_{\eff}(y,t)
    \gtrsim
    \frac{V_{\rm cen}(y,t)}{\eta\Gainstar(y,t)}
    \log\frac1\delta.
\]
This criterion is conditional on the OU posterior law $\rhoOU=p_0(\cdot\mid Y_t=y)$ and on the finite-reference sampling law used to form $\widehat G_{\rm cen}$; it is an effective-sample-size condition, not a universal lower bound.  The relevant comparison is whether the centered-regression noise scale $V_{\rm cen}$ is larger than the Hessian-concentration scale required by LFGI in the non-Gaussian worked examples.

\subsection{Gaussian cancellation and the non-Gaussian residual}
\label{subsec:non-gaussian-residual}

Define the optimal-gate residual
\begin{equation*}
    r_\star(x;y,t)
    :=
    b(x;y,t)-s_t(y)+\Gstar(y,t)\delta(x;y,t).
\end{equation*}
By \cref{eq:normal-equation},
\begin{equation*}
    \E[r_\star(X_0;y,t)\delta(X_0;y,t)^\top\mid Y_t=y]=0.
\end{equation*}
The residual $r_\star$ is the pointwise part of $b(X_0;y,t)-s_t(y)$ that is not removed by the best linear map $\Gstar\delta(X_0;y,t)$.  It vanishes in the Gaussian cancellation result \cref{prop:gaussian-centered-cancellation} and is the term whose empirical cross moment with $\delta$ drives the centered-primal finite-reference error.

\begin{proposition}[Gaussian pointwise cancellation]
\label[proposition]{prop:gaussian-centered-cancellation}
For $p_0=\N(m,P^{-1})$ and fixed $(y,t)$,
\begin{equation*}
b(x;y,t)-s_t(y)
    =
    -\Gstar(t)\delta(x;y,t)
    \qquad
    \text{for every }x.
\end{equation*}
Equivalently, $r_\star\equiv0$.  Consequently, every centered empirical sample satisfies
\[
    b_i-\bar b
    =
    -\Gstar(t)(\delta_i-\bar\delta),
\]
so $\widehat C_{\eb\delta}=-\Gstar\widehat C_{\delta\delta}$ exactly.  Therefore $\widehat G_{\rm cen}=\Gstar$ on the empirical disagreement span, and it equals $\Gstar$ as a full matrix whenever $\widehat C_{\delta\delta}$ has full rank.
\end{proposition}

Proof deferred to \cref{app:proof-prop-gaussian-centered-cancellation}.

Thus the singular Gaussian is not a failure case for the centered estimator.  The finite-reference difficulty for centered primal regression must therefore come from regimes where $r_\star\neq0$.

For centered regression, the coupled error appearing in \cref{eq:direct-moment-risk-exact} can be written as the empirical residual--disagreement cross moment
\begin{equation*}
    \widehat C_{r_\star\delta}
    :=
    \widehat C_{\eb\delta}+\Gstar\widehat C_{\delta\delta}
    =
    \sum_i w_i
    \bigl(r_{\star,i}-\bar r_\star\bigr)
    (\delta_i-\bar\delta)^\top,
\end{equation*}
where $\bar r_\star=\sum_iw_i r_{\star,i}$.  Hence
\begin{equation}
\label{eq:centered-regression-residual-risk}
    \mathcal R(\widehat G_{\rm cen})-\mathcal R(\Gstar)
    =
    \nor{\widehat C_{r_\star\delta}\,\widehat C_{\delta\delta}^{-1}\Cdd^{1/2}}_F^2.
\end{equation}
Ill-conditioning of $\widehat C_{\delta\delta}$ matters only through a nonzero residual--disagreement cross moment.  Equation \eqref{eq:centered-regression-residual-risk} shows that low rank or ill-conditioning of the disagreement covariance is not, by itself, the failure mode.  It becomes damaging only when a nonzero empirical residual--disagreement cross moment is present for the inverse to amplify.  Indeed, by submultiplicativity,
\begin{equation}
\label{eq:centered-residual-amplification-bound}
    \mathcal R(\widehat G_{\rm cen})-\mathcal R(\Gstar)
    \le
    \nor{\widehat C_{r_\star\delta} \Cdd^{-1/2}}_F^2
    \nor{\Cdd^{1/2}\widehat C_{\delta\delta}^{-1}\Cdd^{1/2}}_{\op}^2,
\end{equation}
with the usual additional ridge/projection term when \cref{eq:centered-primal-gate} uses $\lambdaridge>0$.  The factor $\|\widehat C_{r_\star\delta}\Cdd^{-1/2}\|_F^2$ measures the residual--disagreement cross moment.  The operator-norm factor $\|\Cdd^{1/2}\widehat C_{\delta\delta}^{-1}\Cdd^{1/2}\|_{\op}^2$ measures amplification by the empirical inverse of $\widehat C_{\delta\delta}$.  The Gaussian cancellation above sets $\widehat C_{r_\star\delta}$ to zero, so a nearly singular $\Cdd$ has nothing to amplify in the $\Cdd$-weighted excess-risk norm.  In stiff non-Gaussian targets, both factors can be active.

The same residual-covariance mechanism explains the misaligned-stiffness behavior used in the validation examples.  In misaligned Gaussian mixtures, the population gate is spatially varying in $(y,t)$ because component precisions have different stiff eigenspaces, and centered primal regression must estimate a residual-coupled covariance quotient whose finite-reference error is amplified by the empirical inverse of $\widehat C_{\delta\delta}$.  The auxiliary two-component calculation in \cref{subsec:misaligned-bimodal-gmm,subsec:bimodal-primal-failure-scale} records the closed-form spatially varying gate and the residual-leakage scale used to interpret the misaligned-GMM validation target.

\Cref{sec:lfgi} derives the population gate $\Gstar(y,t)$ from the conditional Hessian average $H(y,t)$ instead of from the covariance quotient $-\Cebd\Cdd^{-1}$.

\section{Laplace--Fisher Gate Identity}
\label{sec:lfgi}

LFGI solves the full-matrix local trace-risk problem $\min_G R(G;y,t)$ from \cref{def:local-risk}, whose normal equation gives $\Gstar=-\Cebd\Cdd^{-1}$, but estimates the conditional Hessian average $H(y,t)$ rather than the moment quotient $-\Cebd\Cdd^{-1}$.  Instead of estimating the residual-coupled regression moments and inverting the empirical disagreement covariance, we use Fisher--Stein identities to rewrite the normal equation as a resolvent of the OU-conditional average target observed information. With exact conditional moments, the covariance quotient $-\Cebd\Cdd^{-1}$ and the Hessian-average map $\Psiop(H)=\alphat^2(\alphat^2I_d+\gammat H)^{-1}$ are equivalent, but they require different finite-reference estimates.

\subsection{Fisher--Stein identities for the normal equation}

The reduction of \cref{eq:normal-equation} to a resolvent uses the fact that the disagreement signal is the posterior score.  Therefore its second moment is a posterior Fisher information, and the corresponding Fisher information can be converted into an expected negative Hessian by Bartlett's identity \citep{vaart1998asymptotic}.

\begin{definition}[Target observed information and conditional average]
\label[definition]{def:H0-H}
Define
\begin{equation*}
    H_0(x):=-\nabla_x^2\log p_0(x),
\end{equation*}
and
\begin{equation*}
    H(y,t):=
    \E[H_0(X_0)\mid Y_t=y].
\end{equation*}
\end{definition}

\begin{lemma}[Disagreement covariance is posterior Fisher information]
\label[lemma]{lem:M-fisher}
Let
\[
    \mathcal I(\rhoOU)
    :=
    \E_{\rhoOU}
    [
        q(X)q(X)^\top
    ],
    \qquad
    q=\nabla_x\log\rhoOU.
\]
Then
\begin{equation*}
    \Cdd(y,t)=\frac{1}{\alphat^2}\mathcal I(\rhoOU).
\end{equation*}
Moreover,
\begin{equation}
\label{eq:bartlett}
    \mathcal I(\rhoOU)
    =
    \E_{\rhoOU}[-\nabla_x^2\log\rhoOU(X)].
\end{equation}
\end{lemma}

Proof deferred to \cref{app:proof-lem-m-fisher}.

\begin{lemma}[Universal residual--disagreement cross moment]
\label[lemma]{lem:Cebd-universal}
Under \cref{ass:regularity},
\begin{equation}
\label{eq:Cebd-universal}
    \Cebd(y,t)=-\frac{1}{\gammat}I_d.
\end{equation}
\end{lemma}

Proof deferred to \cref{app:proof-lem-cebd-universal}.

Under \cref{ass:regularity}, \cref{eq:Cebd-universal} gives the explicit identity $\Cebd=-\gammat^{-1}I_d$.  Therefore the normal-equation formula $\Gstar=-\Cebd\Cdd^{-1}$ leaves $\Cdd$ as the only target-dependent moment; \cref{lem:fisher-H} then rewrites $\Cdd$ in terms of $H(y,t)$.

\begin{lemma}[Posterior Fisher information equals averaged observed information plus OU curvature]
\label[lemma]{lem:fisher-H}
Under \cref{ass:regularity},
\begin{equation*}
    \mathcal I(\rhoOU)
    =
    H(y,t)+\frac{\alphat^2}{\gammat}I_d.
\end{equation*}
Consequently,
\begin{equation}
\label{eq:M-H}
    \Cdd(y,t)
    =
    \frac{1}{\alphat^2}H(y,t)+\frac{1}{\gammat}I_d
    =
    \frac{1}{\gammat\alphat^2}
    \left(
        \alphat^2I_d+\gammat H(y,t)
    \right).
\end{equation}
\end{lemma}

Proof deferred to \cref{app:proof-lem-fisher-h}.

The expression for $\Cdd$ in \cref{eq:M-H} shows that the normal-equation covariance is already a shifted Hessian average.  The LFGI estimator estimates the conditional Hessian average $H(y,t)$ and then applies the map $\Psiop(H)=\alphat^2(\alphat^2I_d+\gammat H)^{-1}$.

\subsection{The exact optimal gate theorem}

The preceding identities identify both matrices in the population normal equation: $\Cebd=-\gammat^{-1}I_d$ and $\Cdd=(\gammat\alphat^2)^{-1}(\alphat^2I_d+\gammat H)$.  Substituting them into $\Gstar=-\Cebd\Cdd^{-1}$ gives the Hessian-resolvent gate.

\begin{theorem}[Laplace--Fisher Gate Identity]
\label[theorem]{thm:lfgi}
Under \cref{ass:regularity}, fix $(y,t)$ and assume that the averaged Hessian shift
\[
    A(y,t):=\alphat^2I_d+\gammat\E[H_0(X_0)\mid Y_t=y]
\]
is invertible.  Then the matrix gate that minimizes the conditional score-estimation risk in \cref{def:local-risk} is
\begin{equation}
\label{eq:lfgi}
    \Gstar(y,t)
    =
    \alphat^2
    \left(
        \alphat^2I_d+
        \gammat\,\E[H_0(X_0)\mid Y_t=y]
    \right)^{-1}.
\end{equation}
Equivalently,
\begin{equation}
\label{eq:Psi}
    \Gstar(y,t)=\Psiop(H(y,t)),
    \qquad
    \Psiop(P):=\alphat^2(\alphat^2I_d+\gammat P)^{-1}.
\end{equation}
\end{theorem}

Proof deferred to \cref{app:proof-thm-lfgi}.

The name \emph{Laplace--Fisher} is meant in this limited sense.  The Fisher part is the identity in \cref{lem:M-fisher,lem:fisher-H}: the disagreement covariance is the posterior Fisher information of the OU conditional law, equivalently the averaged observed information plus the OU curvature term in \cref{eq:M-H}.  The Laplace part refers to the same local-curvature object used by Laplace approximations: a Hessian or observed-information matrix is shifted by the diffusion curvature before inversion.  Here that curvature is not used to replace the target by a Gaussian approximation; it gives the risk-minimizing gate between two exact score identities.

LFGI is not a different risk objective.  It is the same risk-minimizing matrix gate as the centered primal normal equation, written as the Hessian-average map $\Psiop(H)=\alphat^2(\alphat^2I_d+\gammat H)^{-1}$.  The centered primal estimator computes covariance moments of $(b,\delta)$ and inverts the empirical disagreement covariance.  LFGI estimates $H$ and applies the map $\Psiop$. With exact conditional moments, the moment quotient $-\Cebd\Cdd^{-1}$ and the Hessian-average formula $\Psiop(H)$ agree, but at finite $N$ they require estimating different quantities once the Gaussian pointwise cancellation is broken.

The corresponding population score signal is the gated signal from \cref{def:gated-score} with $G=\Gstar$:
\[
    z_{\LFGI}(x;y,t)=b(x;y,t)+\Gstar(y,t)\delta(x;y,t).
\]
Equation \eqref{eq:measurable-gate-target} gives $\E[z_{\LFGI}(X_0;y,t)\mid Y_t=y]=s_t(y)$, so LFGI remains an estimator of the OU-marginal score $s_t(y)$.  By \eqref{eq:normal-equation} and \cref{thm:lfgi}, the same gate minimizes the conditional trace risk in \cref{def:local-risk} over matrices $G\in\R^{d\times d}$.

\subsection{Spectral attenuation}

The resolvent form makes the gate interpretable as a curvature-dependent filter.  The active directions are determined by the eigenvalues of the conditional observed information, not by a fixed coordinate basis.

If $H(y,t)u_j=\lambda_j(y,t)u_j$, then the resolvent formula gives the spectral attenuation rule
\[
    \Gstar(y,t)u_j
    =
    \frac{\alphat^2}{\alphat^2+\gammat\lambda_j(y,t)}u_j.
\]
Thus directions with $\gammat\lambda_j\gg\alphat^2$ are attenuated toward Tweedie, while directions with $\gammat\lambda_j\ll\alphat^2$ remain close to TSI.  Equivalently, the active curvature threshold at time $t$ is
\[
    \lambda_{\rm act}(t)
    :=
    \frac{\alphat^2}{\gammat}
    =
    \lambdat^{-1}.
\]
Curvatures near $\lambda_{\rm act}(t)$ are most sensitive to gate error.  The scale $\lambda_{\rm act}(t)=\alphat^2/\gammat$ determines the attenuation rule above: at each $t$, LFGI asks whether a conditional curvature direction is large or small relative to this threshold, and it chooses the corresponding amount of correction automatically.

\section{Finite-Reference LFGI}
\label{sec:finite-reference}

The population LFGI gate depends on the conditional average $H(y,t)$, which is generally unavailable in closed form.  The finite-reference Hessian-average estimator approximates $H(y,t)$ by SNIS, and the finite-reference LFGI gate estimator applies the resolvent map to that average. The implementable score estimator uses two independent reference banks: a score bank for the Tweedie and TSI signals, and a gate bank for the Hessian average.  The bank split is not needed to define the population gate.  It makes the random gate a function only of the gate bank.  After conditioning on the gate bank and the query $(y,t)$, the gate is fixed with respect to the score-bank reference samples, which is the independence needed for \eqref{eq:measurable-gate-target}.  The same OU posterior weights are used within each bank \citep{owen2013mc,liu2001combined,kong1994sequential}.

\subsection{OU-SNIS Hessian averaging}

Let
\[
    \bank_{\rm s}=\{X^{\rm s}_i\}_{i=1}^{\Ns},
    \qquad
    \bank_{\rm g}=\{X^{\rm g}_i\}_{i=1}^{\Ng},
    \qquad
    X^{\rm s}_i,X^{\rm g}_i\overset{\rm iid}{\sim}p_0
\]
be independent score and gate banks.  For $\ell\in\{\mathrm{s},\mathrm{g}\}$, define unnormalized and normalized OU weights
\begin{equation*}
    \widetilde w^{\ell}_i(y,t)=\KOU(y\mid X^{\ell}_i),
    \qquad
    w^{\ell}_i(y,t)=
    \frac{\widetilde w^{\ell}_i(y,t)}
         {\sum_j\widetilde w^{\ell}_j(y,t)}.
\end{equation*}
These weights approximate the OU posterior $\rhoOU$ at the query point $(y,t)$ within each bank.  Applying the gate-bank weights to $H_0(X^{\rm g}_i)$ gives the empirical conditional observed information.

\begin{definition}[Finite-reference Hessian average]
\label[definition]{def:Hhat}
Define
\begin{equation*}
    \Hhat^{\rm g}_{\Ng}(y,t)
    =
    \sum_{i=1}^{\Ng}w^{\rm g}_i(y,t)H_0(X^{\rm g}_i).
\end{equation*}
When the gate bank is the only finite-reference quantity, we write this average generically as $\Hhat_N$.
\end{definition}

\begin{definition}[Finite-reference LFGI gate]
\label[definition]{def:finite-lfgi-gate}
Define
\begin{equation*}
    \Ghat^{\rm g}_{\Ng}(y,t)
    =
    \alphat^2
    \left(
        \alphat^2I_d+\gammat\Hhat^{\rm g}_{\Ng}(y,t)
    \right)^{-1}.
\end{equation*}
Again, $\Ghat_N$ denotes the corresponding gate when the gate bank is the only finite-reference quantity.
\end{definition}

\begin{definition}[Independent-bank finite-reference LFGI score estimator]
\label[definition]{def:finite-lfgi-score}
Define the score-bank SNIS averages
\begin{equation*}
    \widehat b^{\rm s}(y,t)=\sum_i w^{\rm s}_i(y,t)b(X^{\rm s}_i;y,t),
    \qquad
    \widehat c^{\rm s}(y,t)=\sum_i w^{\rm s}_i(y,t)c(X^{\rm s}_i;t),
    \qquad
    \widehat\delta^{\rm s}(y,t)=\widehat c^{\rm s}(y,t)-\widehat b^{\rm s}(y,t).
\end{equation*}
The implementable independent-bank LFGI score estimator is
\begin{equation*}
    \sLFGI(y,t)
    =
    \widehat b^{\rm s}(y,t)
    +
    \Ghat^{\rm g}_{\Ng}(y,t)\widehat\delta^{\rm s}(y,t).
\end{equation*}
\end{definition}

The same bank convention defines the finite-reference Matrix Blend estimator used in the validation benchmarks of \cref{sec:validation-plan}.  This baseline estimates the gate from centered covariances rather than from Hessians.  Its gate-bank reference samples form \(\widehat G_{\rm cen}^{\rm g}\) by \cref{eq:centered-primal-gate,eq:centered-covariances}, and its score-bank reference samples form \(\widehat b^{\rm s}\) and \(\widehat\delta^{\rm s}\) as above:
\begin{equation*}
    \sMBLEND(y,t)
    =
    \widehat b^{\rm s}(y,t)
    +
    \widehat G_{\rm cen}^{\rm g}(y,t)\widehat\delta^{\rm s}(y,t).
\end{equation*}
Thus Matrix Blend and LFGI target the same population gate $\Gstar(y,t)$, but Matrix Blend estimates the empirical covariance quotient in \cref{eq:centered-primal-gate,eq:centered-covariances}, whereas LFGI estimates the empirical Hessian average $\widehat H_N(y,t)$ before applying $\Psiop$.

The LFGI gate is therefore an adaptive function of an independently drawn local Hessian bank rather than a separately trained network or a global time-dependent coefficient.  Using the same score bank for $\widehat b^{\rm s}$ and $\widehat c^{\rm s}$ is harmless because their difference is the zero-mean control variate being averaged.  The independence requirement is specifically between the random gate and the score-signal reference samples to which that gate is applied.

\subsection{Consistency}

Fix $(y,t)$; the gate-bank draw is the only randomness in the consistency statement.  The self-normalized Hessian average $\widehat H_N(y,t)$, and therefore the plug-in gate $\widehat G_N(y,t)$, is generally biased at finite bank size.  The basic asymptotic property is consistency: as the gate bank grows, $\widehat H_N(y,t)$ converges to the OU-conditional Hessian average $H(y,t)$, and $\widehat G_N(y,t)$ converges to the population LFGI gate $\Gstar(y,t)$ \citep{vaart1998asymptotic,owen2013mc}.

\begin{theorem}[OU-SNIS consistency for the LFGI gate]
\label[theorem]{thm:snis-lfgi-consistency}
Assume
\[
    0<p_t(y)<\infty,
    \qquad
    \E_{p_0}\left[
        \KOU(y\mid X)\nor{H_0(X)}
    \right]<\infty,
\]
and assume the aggregate covariance $\Cdd(y,t)$ is nonsingular, equivalently $A(y,t)=\alphat^2I_d+\gammat H(y,t)$ is nonsingular.  Then
\begin{equation*}
    \Hhat_N(y,t)
    \xrightarrow[N\to\infty]{a.s.}
    H(y,t),
\end{equation*}
and, for all sufficiently large $N$ almost surely,
\begin{equation*}
    \Ghat_N(y,t)
    \xrightarrow[N\to\infty]{a.s.}
    \Gstar(y,t).
\end{equation*}
\end{theorem}

Proof deferred to \cref{app:proof-thm-snis-lfgi-consistency}.

At finite reference size, the Hessian-average, gate, and score estimators are built from self-normalized weighted averages. Therefore finite-reference Hessian averages and score-bank averages are generally biased at order $O(N^{-1})$ for each bank, while the corresponding Hessian-average, gate, and score estimators are consistent under the stated moment conditions.  Independent score and gate banks remove the additional finite-reference coupling that would arise from applying a random gate to the same reference samples used to construct it.

The finite-reference score estimator is consistent because the score-bank averages converge to $s_t(y)$ and $0$, while the gate-bank estimator converges to $\Gstar(y,t)$.  Under the hypotheses of \cref{thm:snis-lfgi-consistency} and the corresponding first-moment conditions for $b$ and $c$, as $\Ng\to\infty$ and $\Ns\to\infty$,
\[
    \widehat b^{\rm s}(y,t)\to s_t(y),
    \qquad
    \widehat c^{\rm s}(y,t)\to s_t(y),
    \qquad
    \widehat\delta^{\rm s}(y,t)\to0
\]
almost surely, and therefore
\[
    \sLFGI(y,t)\to s_t(y)
\]
almost surely.  At a fixed $(y,t)$, the empirical disagreement average $\widehat\delta^{\rm s}$ vanishes as the score bank grows.  Gate accuracy matters at finite reference count because it controls finite-reference variance and risk, which is why the finite-reference count analysis below focuses on risk capture rather than only on asymptotic consistency.

The preceding limits give $\widehat G_N(y,t)\to\Gstar(y,t)$ and $\widehat s_{\rm LFGI}(y,t)\to s_t(y)$.  For sampling, however, the relevant question is finite-reference stability: how much Hessian averaging error can the resolvent tolerate before the risk benefit of the population gate is lost?

\subsection{Gate stability under Hessian averaging error}

Writing
\[
    \Delta_H=\Hhat_N-H,
    \qquad
    A=\alphat^2I_d+\gammat H,
    \qquad
    \AhatN=\alphat^2I_d+\gammat\Hhat_N,
\]
the standard resolvent identity gives the ordinary operator-norm perturbation estimate
\[
    \nor{\Ghat_N-\Gstar}_{\op}
    \le
    \alphat^2\gammat
    \nor{\AhatN^{-1}}_{\op}
    \nor{\Delta_H}_{\op}
    \nor{A^{-1}}_{\op}
\]
whenever $A$ and $\AhatN$ are nonsingular.  In particular, if
\[
    \gammat\nor{\Delta_H}_{\op}
    \le
    \theta_A\,\lambda_{\min}(A),
    \qquad
    0<\theta_A<1,
\]
then $\nor{\Ghat_N-\Gstar}_{\op}/\nor{\Gstar}_{\op}\le\theta_A/(1-\theta_A)$.  This bound is useful but still expressed in an ordinary operator norm.  \Cref{prop:risk-weighted-lfgi-perturbation} bounds the excess risk in \cref{eq:excess-risk} by the Hessian-average error $\Delta_H=\Hhat_N-H$.

\begin{proposition}[Risk-weighted LFGI perturbation]
\label[proposition]{prop:risk-weighted-lfgi-perturbation}
Assume $A(y,t)\succ0$ and define
\[
    \Delta_H=\Hhat_N-H,
    \qquad
    \epsilon_H
    :=
    \gammat\nor{A^{-1/2}\Delta_HA^{-1/2}}_{\op}.
\]
If $\epsilon_H<1$, then $\AhatN\succ0$ and
\begin{equation*}
    \mathcal R(\Ghat_N;y,t)-\mathcal R(\Gstar;y,t)
    \le
    \alphat^4
    \left(\frac{\epsilon_H}{1-\epsilon_H}\right)^2
    \nor{A^{-1}}_{\op}\,\tr\left(A^{-1}\Cdd\right).
\end{equation*}
In particular, if $\epsilon_H\le1/2$, then
\begin{equation*}
    \mathcal R(\Ghat_N;y,t)-\mathcal R(\Gstar;y,t)
    \le
    4\alphat^4\epsilon_H^2
    \nor{A^{-1}}_{\op}\,\tr\left(A^{-1}\Cdd\right).
\end{equation*}
\end{proposition}

Proof deferred to \cref{app:proof-prop-risk-weighted-lfgi-perturbation}.

The perturbation bound controls $R(\widehat G_N;y,t)-R(\Gstar;y,t)$ through $\epsilon_H$ and the trace factor $\nor{A^{-1}}_{\op}\tr(A^{-1}\Cdd)$.  The gate need not be learned accurately in every matrix direction; it must be learned in the $\Cdd$-weighted excess-risk norm after relative Hessian preconditioning by $A^{-1/2}$.  This bound is the input to the sample-size condition in \cref{sec:reference-complexity}.

\section{Reference Complexity and Hessian Conditioning}
\label{sec:reference-complexity}

For Matrix Blend, which uses no Hessians, \cref{eq:centered-regression-residual-risk} shows that the finite-reference error is controlled by $\widehat C_{r_\star\delta}\widehat C_{\delta\delta}^{-1}\Cdd^{1/2}$.  For finite-reference LFGI, the estimated quantity is the conditional Hessian average $H(y,t)$.  We ask when $\mathcal E(\widehat G_N;y,t)\le\eta\,\Gainstar(y,t)$, where $\Gainstar=R(0;y,t)-R(\Gstar;y,t)$ is the population-optimal risk reduction defined in \cref{def:population-risk-reduction}.  LFGI estimates the OU-conditional observed information through the averaged shifted matrix
\[
    A(y,t)=\alphat^2I_d+\gammat H(y,t).
\]
We call a query $(y,t)$ \emph{pole-separated} when the averaged shifted matrix $A(y,t)$ is uniformly separated from singularity on directions with non-negligible weight in the $\Cdd$-weighted excess-risk norm.  In the PSD case $H(y,t)\succeq0$, this separation is automatic because $A(y,t)\succeq\alphat^2 I_d$.  The favorable LFGI condition is therefore not simply ``PSD because PSD is safer.''  It is \emph{the inequality in \cref{eq:lfgi-risk-capture-condition}}: Hessian averaging error is small after preconditioning by $A^{-1/2}$, and the remaining amplification is measured by the pole factor $\Lambdapole=\nor{A^{-1}}_{\op}\tr(A^{-1}\Cdd)$.  The sufficient-condition comparison depends on two finite-reference quantities: residual-coupled covariance inversion for Matrix Blend and relative Hessian concentration for LFGI.  It is not a calibrated predictor of empirical advantage.  For compact notation, define the conservative risk-weighted pole factor
\begin{equation}
\label{eq:Lambda-pole}
    \Lambdapole(y,t)
    :=
    \nor{A(y,t)^{-1}}_{\op}\,\tr\left(A(y,t)^{-1}\Cdd(y,t)\right).
\end{equation}

Combining \cref{prop:risk-weighted-lfgi-perturbation} with the definition of $\Gainstar=R(0;y,t)-R(\Gstar;y,t)$ gives the following risk-capture condition.  For $0<\eta<1$, if
\begin{equation}
\label{eq:lfgi-risk-capture-condition}
    \alphat^4
    \left(\frac{\epsilon_H}{1-\epsilon_H}\right)^2
    \Lambdapole(y,t)
    \le
    \eta\Gainstar(y,t),
\end{equation}
then
\begin{equation*}
    \mathcal R(0;y,t)-\mathcal R(\Ghat_N;y,t)
    \ge
    (1-\eta)\Gainstar(y,t).
\end{equation*}
Thus the finite-reference LFGI gate captures at least a $(1-\eta)$ fraction of the population-optimal risk reduction whenever the perturbation term in \eqref{eq:lfgi-risk-capture-condition} is at most $\eta\Gainstar(y,t)$.

The constant-Hessian case provides the exact tie between the two matrix estimators.  When $H_0$ is constant, LFGI has no Hessian-averaging error and centered primal regression has the Gaussian pointwise cancellation from \cref{prop:gaussian-centered-cancellation}.  Both matrix estimators therefore recover the same gate; the finite-reference statement is collected in \cref{subsec:constant-hessian-gaussian-tie}.  The comparison between LFGI and centered primal regression is consequently a non-Gaussian finite-reference question, not a Gaussian one.

\subsection{Variable-Hessian targets, PSD regimes, and multiplicative stability}

For non-Gaussian targets, the Hessian is not constant, and LFGI must estimate an OU-conditional Hessian average.  The relevant concentration scale is not the raw variation of $H_0(X)$; it is the variation after preconditioning by $A=\alphat^2I_d+\gammat H$.  The relevant finite-reference quantity is the relative Hessian fluctuation $\gamma_t\nor{A^{-1/2}(\widehat H_N-H)A^{-1/2}}_{\op}$, which replaces the residual-coupled centered-regression term described in \cref{subsec:non-gaussian-residual}.

\begin{assumption}[Relative conditional Hessian concentration]
\label[assumption]{ass:hessian-concentration}
At fixed $(y,t)$, suppose $A\succ0$ and the finite-reference Hessian average satisfies, with probability at least $1-\delta$,
\begin{equation}
\label{eq:relative-hessian-concentration}
    \epsilon_H
    =
    \gammat\nor{A^{-1/2}(\Hhat_N-H)A^{-1/2}}_{\op}
    \le
    C\left(
        \sqrt{\frac{v_A^2\log(2d/\delta)}{N_{\eff}(y,t)}}
        +
        \frac{R_A\log(2d/\delta)}{N_{\eff}(y,t)}
    \right),
\end{equation}
where $v_A^2$ and $R_A$ are dimensionless variance and envelope parameters for the relative Hessian fluctuations
\[
    \gammat A^{-1/2}(H_0(X)-H)A^{-1/2},
    \qquad X\sim p_0(\cdot\mid Y_t=y).
\]
For iid conditional samples, $N_{\eff}=N$.  For OU-SNIS reference banks, $N_{\eff}$ denotes the local effective sample size
\begin{equation*}
    N_{\eff}(y,t)
    =
    \frac{1}{\sum_{i=1}^{N}w_i(y,t)^2}.
\end{equation*}
\end{assumption}

The iid conditional-sample version of \cref{ass:hessian-concentration} is a standard matrix Bernstein statement and is proved in \cref{app:proof-hessian-concentration-bernstein} \citep{tropp2012user,vershynin2018high,wainwright2019high}.  The self-normalized OU reference-bank version used in implementations is kept as the explicit effective-sample-size hypothesis in \cref{ass:hessian-concentration}.  We do not claim a sharp standalone SNIS matrix Bernstein theorem beyond that hypothesis.

A simple sufficient condition makes the relative fluctuation parameters in \cref{ass:hessian-concentration} interpretable.  Suppose $H(y,t)\succeq0$ and, for some $\beta_H\ge0$,
\[
    (1-\beta_H)H(y,t)
    \preceq
    H_0(X)
    \preceq
    (1+\beta_H)H(y,t)
    \qquad
    \text{for }X\sim p_0(\cdot\mid Y_t=y)\text{ a.s.}
\]
Then the relative Hessian fluctuation
\[
    \Xi_H(X)
    :=
    \gammat A(y,t)^{-1/2}
    \bigl(H_0(X)-H(y,t)\bigr)
    A(y,t)^{-1/2}
\]
satisfies $\|\Xi_H(X)\|_{\op}\le\beta_H$ almost surely.  Indeed, conjugating
$-\beta_H H\preceq H_0(X)-H\preceq\beta_H H$ by $\gammat^{1/2}A^{-1/2}$ and using
$0\preceq\gammat A^{-1/2}H A^{-1/2}\preceq I_d$ places every eigenvalue of $\Xi_H(X)$ in $[-\beta_H,\beta_H]$.  Thus one may take $R_A\le\beta_H$ and $v_A^2\le\beta_H^2$ in \cref{ass:hessian-concentration}; this relative concentration scale is independent of the condition number of $H(y,t)$.

\begin{theorem}[Risk-weighted sample complexity for LFGI]
\label[theorem]{thm:hessian-reference-condition}
Assume \cref{ass:hessian-concentration}, let $0<\eta<1$, and ignore the lower-order linear Bernstein term for readability.  For OU-SNIS banks, \cref{ass:hessian-concentration} is used here as the explicit effective-sample-size hypothesis.  A sufficient condition for finite-reference LFGI to capture at least a $(1-\eta)$ fraction of the population-optimal risk reduction with probability at least $1-\delta$ is
\begin{equation}
\label{eq:hessian-reference-sufficient}
    N_{\eff}(y,t)
    \gtrsim
    \frac{
        \alphat^4 v_A^2\Lambdapole(y,t)
    }{
        \eta\Gainstar(y,t)
    }
    \log\frac{2d}{\delta},
\end{equation}
together with the small-perturbation condition $\epsilon_H\le1/2$.  Including the Bernstein envelope term gives the sufficient condition
\begin{equation}
\label{eq:hessian-reference-sufficient-full}
    C\left(
        \sqrt{\frac{v_A^2\log(2d/\delta)}{N_{\eff}(y,t)}}
        +
        \frac{R_A\log(2d/\delta)}{N_{\eff}(y,t)}
    \right)
    \le
    \frac12
    \min\left\{
        1,
        \sqrt{\frac{\eta\Gainstar(y,t)}{4\alphat^4\Lambdapole(y,t)}}
    \right\}.
\end{equation}
\end{theorem}

Proof deferred to \cref{app:proof-thm-hessian-reference-condition}.

The theorem should be read as a sufficient condition, not as a complete characterization of all successful cases.  It identifies the quantities that control the finite-reference excess risk of LFGI: relative Hessian concentration, the pole factor $\Lambdapole$, and the population-optimal risk reduction $\Gainstar$ that must dominate the perturbation.  In PSD or pole-separated regimes, the trace factor in \cref{eq:Lambda-pole} cancels exactly:
\[
    \tr\left(A(y,t)^{-1}\Cdd(y,t)\right)
    =
    \frac{d}{\gammat\alphat^2},
    \qquad
    \Lambdapole(y,t)
    =
    \frac{d}{\gammat\alphat^2}
    \|A(y,t)^{-1}\|_{\op}.
\]
If $H(y,t)\succeq-\mu_t(y)I_d$ and $a_t(y):=\alphat^2-\gammat\mu_t(y)>0$, then $A(y,t)\succeq a_t(y)I_d$ and hence
\[
    \Lambdapole(y,t)
    \le
    \frac{d}{\gammat\alphat^2a_t(y)}.
\]
In the sufficiently PSD case $H\succeq0$, this gives
\[
    \alphat^4\Lambdapole(y,t)
    \le
    \frac{d}{\gammat}.
\]
Thus PSD or pole-separated targets do not pay for raw stiff curvature through $\tr(\Cdd)$ in the LFGI perturbation bound.  The resolvent cancels that stiffness before the remaining operator-norm pole factor is applied.

Combining the PSD cancellation with the multiplicative-stability scale $v_A^2\le\beta_H^2$ and $R_A\le\beta_H$ gives a compact sufficient condition.  Ignoring the lower-order Bernstein envelope term, finite-reference LFGI captures at least a $(1-\eta)$ fraction of the population-optimal risk reduction with probability at least $1-\delta$ whenever
\[
    N_{\eff}(y,t)
    \gtrsim
    \frac{\beta_H^2 d}{\eta\,\gammat\,\Gainstar(y,t)}
    \log\frac{2d}{\delta},
\]
together with the small-perturbation condition $\epsilon_H\le1/2$.  Including the Bernstein envelope term, it is sufficient that
\[
    C\left(
        \sqrt{\frac{\beta_H^2\log(2d/\delta)}{N_{\eff}(y,t)}}
        +
        \frac{\beta_H\log(2d/\delta)}{N_{\eff}(y,t)}
    \right)
    \le
    \frac12
    \min\left\{
        1,
        \sqrt{\frac{\eta\,\gammat\,\Gainstar(y,t)}{4d}}
    \right\}.
\]
The constant-Hessian case corresponds to $\beta_H=0$, recovering the exact calculation in \eqref{eq:constant-hessian-lfgi-recovery}: LFGI has $\widehat G_N=\Gstar$ from any nonempty Hessian bank, and centered primal regression recovers $\Gstar$ on the empirical disagreement span.  When $\beta_H=O(1)$, the bound no longer claims a special learnability advantage from Hessian stability alone; the remaining denominator $\Gainstar(y,t)$ records whether there is meaningful population-optimal risk reduction to capture.

The resulting finite-reference comparison is between the residual-coupled Matrix Blend term $\widehat C_{r_\star\delta}\widehat C_{\delta\delta}^{-1}\Cdd^{1/2}$ from \cref{eq:centered-regression-residual-risk} and the relative Hessian-average error $\epsilon_H$ from \cref{eq:relative-hessian-concentration}.  The auxiliary comparison in \cref{subsec:residual-coupling-regime} records the corresponding sufficient-condition guide and the GMM-side Hessian calculation.  The validation sweep therefore checks score RMSE and the risk-weighted gate-capture diagnostic directly rather than fitting the sufficient-condition ratio.

\section{Probability-Flow Density Evaluation and Evidence Diagnostics}

\label{sec:pf-density-evaluation-body}

A fixed finite-reference score estimator also defines a normalized density through the probability-flow ODE: after the bank is frozen, the estimated reverse vector field induces a deterministic change of variables from the Gaussian base to the low-noise endpoint.

\subsection{Setting: what a pilot run does and does not provide}
\label{subsec:density-setting}
A Bayesian inverse problem supplies an unnormalized posterior
\[
    \tildep(x)=\exp\{-\Phi(x)\}\,p_{\rm prior}(x),
    \qquad
    Z=\int \tildep(x)\,\dd x,
    \qquad
    p_0(x)=\tildep(x)/Z .
\]
An adjoint-gradient MALA pilot produces samples $x_i$ together with byproducts already computed during the run: $\log\tildep(x_i)$, $s_0(x_i)=\nabla\log\tildep(x_i)$, and either $H_0(x_i)$ or a structured precision proxy such as $P_i^{\GN}$.  The pilot is therefore a strong source of samples, but it does not by itself provide a normalized pointwise density, $\log Z$, or proposal/target overlap certificates.  The purpose of the fixed finite-reference LFGI score estimator is to convert these MCMC byproducts into those density-level objects.

\begin{definition}[Pilot bank for density evaluation]
\label[definition]{def:pilot-bank-density}
A density-evaluation bank is a tuple
\[
    \bank=\{x_i,\ell_i,s_i,P_i,a_i\}_{i=1}^N,
    \qquad
    \ell_i=\log\tildep(x_i),\quad s_i=s_0(x_i),
\]
where $P_i$ is either $H_0(x_i)$ or a measurable PSD proxy, and $a_i$ are optional external log weights.  The default MALA-pilot application uses $a_i\equiv0$ for the score/gate bank and reserves a disjoint MALA-EVAL bank for density assessment.
\end{definition}

The density-evaluation setup uses bank splitting for the same reason the finite-reference theory separates score and gate banks.  The finite-reference theory in \cref{sec:finite-reference} uses independent score and gate banks as a clean sufficient condition.  The density-evaluation setup in \cref{subsec:density-experiments,app:pilot-split-setup} allows shared, prefix, or independent score/gate banks, while reported $\log\qpf$ diagnostics are evaluated on held-out MALA-EVAL samples.  This split-bank construction removes in-sample calibration artifacts; tied-bank variants are reported only as cost or ablation diagnostics.

\subsection{The probability-flow density of a fixed finite-reference score estimator}
\label{subsec:pf-density-frozen-field}

Once the bank and score rule are fixed, the reverse probability-flow ODE is an ordinary deterministic flow.  Its endpoint law is therefore an explicit normalized surrogate density.

\begin{definition}[Probability-flow law induced by a fixed score estimator]
\label[definition]{def:frozen-pf-surrogate}
Fix a pilot bank $\bank$ and a finite-reference score estimator $\widehat s(\cdot,t)$, for example Tweedie, scalar blend, matrix blend, or LFGI.  Let $\qpf$ be the law at time $\tmin$ obtained by integrating the deterministic reverse-OU probability-flow ODE initialized from the standard Gaussian base law $\phibase$ at $\tmax$.  For a query $x$, $\log\qpf(x)$ is evaluated by integrating the inverse flow from $x$ to $\tmax$ and accumulating the divergence by Liouville's formula.
\end{definition}

By the standard Liouville change-of-variables formula for invertible ODE flows \citep{chen2018neuralode,grathwohl2019ffjord,song2021score}, suppose $v_t(x)=-x-\widehat s(x,t)$ is continuously differentiable and has an invertible flow from $\tmin$ to $\tmax$.  If $z_t$ solves
\[
    \dot z_t=v_t(z_t),\qquad z_{\tmin}=x,
\]
and, with $d$ denoting the ambient dimension in $\R^d$, the implementation accumulator is
\[
    \mathcal A_{\tmax}(x)=\int_{\tmin}^{\tmax}\bigl(d+\nabla\!\cdot\widehat s(z_t,t)\bigr)\,\dd t,
\]
then
\begin{equation}
\label{eq:pf-change-of-variables-density}
    \log\qpf(x)=\log\phibase(z_{\tmax})-\mathcal A_{\tmax}(x).
\end{equation}
Consequently $\qpf$ is normalized by construction: fixed-score-estimator error changes which normalized density is produced, but it does not introduce a separate unknown normalizing constant.  No integrability or zero-curl assumption on $\widehat s$ is required.

\subsection{Closed-form divergence for finite-reference LFGI}
\label{subsec:lfgi-density-divergence}

The fixed probability-flow density requires the divergence of the frozen score field along the inverse flow.  For finite-reference LFGI this divergence can be evaluated by differentiating the OU weights and the Hessian-resolvent gate directly.

\begin{proposition}[Closed-form divergence of the finite-reference LFGI score estimator]

\label[proposition]{prop:closed-form-lfgi-divergence-density}
For a finite-reference LFGI score estimator
write $w_i(y,t)$ for the normalized OU weights on the score-estimator bank, $P_i=H_0(X_i)$, $\widehat H=\sum_i w_iP_i$, $\widehat b=\sum_i w_i b_i$, $\widehat c=\sum_i w_i c_i$, and $\widehat\delta=\widehat c-\widehat b$.  Define the centered quantities $\Delta b_i=b_i-\widehat b$ and $\Delta P_i=P_i-\widehat H$, and define $\widehat C_{bb}=\sum_i w_i\Delta b_i\Delta b_i^\top$ and $\widehat C_{cb}=\sum_i w_i(c_i-\widehat c)\Delta b_i^\top$. With
\[
    \widehat s(y,t)=\widehat b(y,t)+\widehat G(y,t)\widehat\delta(y,t),
    \qquad
    \widehat G=\alphat^2(\alphat^2 I_d+\gammat\widehat H)^{-1},
\]
let
\[
    J_b=\widehat C_{bb}-\gammat^{-1}I_d,
    \qquad
    J_{\widehat\delta}=\widehat C_{cb}-J_b,
    \qquad
    T_{auv}=\sum_i w_i(\Delta b_i)_a(\Delta P_i)_{uv}.
\]
Then
\[
    \nabla\!\cdot\widehat s
    =\tr(J_b)+\langle \widehat G,J_{\widehat\delta}^\top\rangle_F
      -\frac{\gammat}{\alphat^2}
      \sum_{a,u,v}\widehat G_{au}\,T_{auv}\,(\widehat G\widehat\delta)_v .
\]
For independent score and gate banks, the same differentiation rule applies with the corresponding bank averages separated.  The formula costs $O(Nd^2)$ moment accumulation plus $O(d^3)$ resolvent work per query time step and uses no Hutchinson trace estimator in the LFGI density path.
\end{proposition}

The proof is a direct SNIS differentiation and resolvent-calculus argument and is given in \cref{app:proof-closed-form-lfgi-divergence-density}.  For each coordinate $a$, $\partial_{y_a}w_i=w_i(\Delta b_i)_a$, giving $J_b=\partial_y\widehat b$, $J_{\widehat\delta}=\partial_y\widehat\delta$, and $T_a=\partial_{y_a}\widehat H$.  Differentiating $\widehat G=\alphat^2(\alphat^2I+\gammat\widehat H)^{-1}$ yields $\partial_{y_a}\widehat G=-(\gammat/\alphat^2)\widehat G T_a\widehat G$; tracing $\partial_y[\widehat b+\widehat G\widehat\delta]$ gives the stated formula.  The full index proof is in \cref{app:proof-closed-form-lfgi-divergence-density}.

Tweedie also admits an analytic divergence, and scalar blend can be evaluated with deterministic coordinate finite differences.  The density comparison is therefore primarily an estimator-accuracy comparison, where estimator accuracy means agreement of the fixed probability-flow energy with target posterior-energy diagnostics on held-out points rather than accuracy of the divergence formula alone.  The closed-form LFGI divergence is still essential for the application because it makes the LFGI density path exact with respect to the fixed finite-reference score estimator and avoids stochastic Hutchinson trace noise in the path likelihood.

Once the score estimator is fixed, probability-flow density evaluation consists of three steps.
\begin{enumerate}[leftmargin=2em,itemsep=0.2em]
\item Build a fixed finite-reference score estimator from a MALA pilot bank using one of four score-estimator constructions: the Tweedie score estimator, a scalar Tweedie--TSI score-estimator blend, a Matrix Blend gate, or an LFGI Hessian-resolvent gate.
\item For held-out evaluation states $x_j$, integrate the inverse probability-flow ODE over $[\tmin,\tmax]$ with Heun/trapezoidal accumulation of $d+\nabla\!\cdot\widehat s$.
\item Return $\log\qpf(x_j)$, estimated energies $\Epf(x_j)=-\log\qpf(x_j)$, pointwise log-ratio certificates $\log\tildep(x_j)-\log\qpf(x_j)$, and ESS diagnostics for the resulting correction weights.
\end{enumerate}


\subsection{From score risk to density diagnostics}
\label{subsec:density-guarantees}

After $\qpf$ has been constructed, the comparison with the unnormalized target is made through log ratios, evidence identities, and correction weights.

For any sufficiently regular fixed score estimator $\widehat s$, the probability-flow construction defines an exactly normalized density $\qpf$ intended to approximate the target density $p_0$ or posterior density under study.  The quality of $\qpf$, as measured by log-ratio diagnostics, KL-type controls, or held-out energy calibration, depends on how closely $\widehat s$ tracks the true OU marginal score along the path.  Score-error controls of this type underlie likelihood and probability-flow analyses of score-based and flow models \citep{song2021score,song2021maximum,lu2022maximum,chen2023probability,benton2023error}.  We do not elevate this transfer to a formal assumption because it is not invoked by a theorem below.  Instead, it is the working regularity condition used to interpret the density diagnostics in this section.

The working probability-flow transfer condition used to interpret these diagnostics is the following.  Let $p_t$ denote the OU marginal path of the target and let $\qpf_t$ be the probability-flow path induced by a fixed estimator $\widehat s(\cdot,t)$.  For the time interval and tail class used in the density-evaluation benchmarks of \cref{subsec:density-experiments,subsec:known-z-calibration}, suppose there exist a finite constant $C_{\rm PF}$ and a nonnegative weight $w(t)$ such that
\begin{equation}
\label{eq:pf-score-error-transfer}
    \KL\bigl(p_{\tmin}\,\|\,\qpf_{\tmin}\bigr)
    \le
    C_{\rm PF}\int_{\tmin}^{\tmax}
    w(t)\,\E_{p_t}\bigl\|\widehat s(Y_t,t)-\nabla\log p_t(Y_t)\bigr\|^2\,\dd t
    +\KL(p_{\tmax}\,\|\,\phibase).
\end{equation}
This condition is the probability-flow analogue of the usual drift-error control used in likelihood and transport analyses \citep{song2021maximum,lu2022maximum,chen2023probability,benton2023error}: the endpoint term records the finite $\tmax$ mismatch between the OU-smoothed target and the Gaussian base, while the integral records score error along the path.  Under \eqref{eq:pf-score-error-transfer}, the fixed-time excess-risk bound from \cref{prop:risk-weighted-lfgi-perturbation} and \eqref{eq:lfgi-risk-capture-condition} transfers to a normalized-density comparison after integration over time.  The conclusion is conditional on pilot coverage: it controls $\qpf$ on the region covered by the pilot/reference bank and does not assert discovery of posterior modes absent from the pilot.

\subsection{Log ratios, correction weights, and evidence identities}
\label{subsec:density-log-ratios-evidence}

For a held-out query $x$, the error bookkeeping starts from the log-ratio diagnostic:
\[
    \Delta_{\qpf}(x)=\log\tildep(x)-\log\qpf(x).
\]
If $\qpf=p_0$ exactly, then $\Delta_{\qpf}(x)=\log Z$ for every $x$.  In practice its variation over held-out posterior points reflects several error sources: fixed-score-estimator error, finite-bank SNIS fluctuation in the score estimator, numerical quadrature error in the probability-flow solve, low-time smoothing/truncation error, and any Hessian or Gauss--Newton proxy error.  We use this decomposition as an accounting device rather than as a separate theorem, and report the corresponding empirical diagnostics: energy rank correlation, affine slope, affine $R^2$, slope-normalized residual RMSE, pointwise log-density/log-ratio diagnostics, and ESS; see Appendix \cref{app:density-metrics}, Defs.~\ref{def:central-density}--\ref{def:evidence-diagnostics}.

The probability-flow density $\qpf$ induced by the fixed finite-reference score estimator can also be used through correction weights. When samples $Y_j\sim\qpf$ are drawn from this density, the usual self-normalized importance weights
\[
    W_j\propto \frac{\tildep(Y_j)}{\qpf(Y_j)}
\]
produce asymptotically exact posterior expectations for integrable test functions:
\[
    \E_{p_0}f
    =
    \frac{\E_{\qpf}\{f(Y)\tildep(Y)/\qpf(Y)\}}
         {\E_{\qpf}\{\tildep(Y)/\qpf(Y)\}}.
\]
The realized ESS of the importance weights $\tildep(Y)/\qpf(Y)$ is therefore an overlap diagnostic for using $\qpf$ as a proposal; see \cref{app:density-metrics}, Def.~\ref{def:evidence-diagnostics}.  On held-out MALA-pilot points, the same log ratios are used as pointwise density audits; they should be approximately constant if the surrogate density agrees with the posterior density up to the unknown normalizing constant.

The normalized probability-flow density $\qpf$ gives one evidence identity under the proposal $\qpf$ and another under the posterior $p_0$. If $Y\sim\qpf$, then
\[
    Z=\E_{\qpf}\left[\frac{\tildep(Y)}{\qpf(Y)}\right],
\]
whereas if $X\sim p_0$, then
\[
    \frac1Z=\E_{p_0}\left[\frac{\qpf(X)}{\tildep(X)}\right].
\]
When both posterior-pilot points and surrogate samples are available, these two identities can be combined with standard bridge-sampling estimators.  The average of $\log\tildep(x)-\log\qpf(x)$ over held-out posterior points is an exact $\log Z$ estimator only in the ideal case $\qpf=p_0$; otherwise it is a calibration diagnostic and a biased plug-in estimate whose bias is visible through the residual diagnostics above.

The reported density-calibration diagnostics are the affine calibration and correction-weight quantities in \cref{app:density-metrics}, Defs.~\ref{def:central-density}--\ref{def:evidence-diagnostics}: affine slope, affine $R^2$, slope-normalized RMSE, raw log-ratio scale, and ESS.  These diagnostics check whether $-\log\qpf$ is an affine proxy for $-\log\tildep$ on the evaluated posterior bulk.  The absolute evidence claim is checked separately on the known-$Z$ analytic inverse problems in \cref{subsec:known-z-calibration}: a Gaussian positive control and a four-component non-Gaussian mixture posterior.

\section{Fisher--Gauss--Newton Gates for PDE Inverse Problems}
\label{sec:gn-gates-body}
\subsection{Fisher--Gauss--Newton gates for inverse problems}
\label{subsec:density-gn-gates}

In PDE inverse problems, exact observed-information averages can be indefinite or expensive, while Fisher--Gauss--Newton precision proxies are PSD and often already available from adjoint calculations.  Using such a proxy in the LFGI gate trades exact Hessian averaging for a pole-separated curvature surrogate.

Concretely, consider a least-squares inverse problem with posterior energy
\[
    E(x):=-\log\tildep(x)
    =\frac12\|F(x)-y_{\rm obs}\|_{\Gamma^{-1}}^2
    +\frac12\|x-m_{\rm pr}\|_{C_{\rm pr}^{-1}}^2+\text{const}.
\]
Write the residual, forward Jacobian, and target observed information as
\[
    r(x)=F(x)-y_{\rm obs},
    \qquad J(x)=\nabla F(x),
    \qquad H_0(x)=-\nabla^2\log\tildep(x)=\nabla^2E(x).
\]
Then
\[
    H_0(x)=C_{\rm pr}^{-1}+J(x)^\top\Gamma^{-1}J(x)+R_{\rm nl}(x),
    \qquad
    R_{\rm nl}(x)=\sum_k [\Gamma^{-1}r(x)]_k\,\nabla^2F_k(x),
\]
where the sum is over observation coordinates.  Here and throughout this section, the Gauss--Newton Hessian approximation means the PSD prior-plus-linearized-likelihood precision proxy
\[
    P^{\GN}(x):=C_{\rm pr}^{-1}+J(x)^\top\Gamma^{-1}J(x),
\]
that is, the exact observed information with the residual-weighted second-derivative term $R_{\rm nl}$ dropped.  Given a gate bank, LFGI--GN replaces the raw Hessian average by
\[
    \widehat H_{\GN}(y,t)
    :=\sum_i w_i^{\rm g}(y,t)P^{\GN}(X_i^{\rm g}),
    \qquad
    \widehat G_{\GN}(y,t)
    :=\alphat^2
    \bigl(\alphat^2I_d+\gammat\widehat H_{\GN}(y,t)\bigr)^{-1}.
\]
The posterior score and energy remain those of the nonlinear target; the Gauss--Newton approximation is used only to construct the curvature gate.

\begin{lemma}[PSD conditioning for Gauss--Newton precision gates]
\label[lemma]{lem:gn-psd-conditioning}
Suppose the bank precision proxies satisfy $P_i^{\GN}\succeq\lambda_0 I_d$ for some $\lambda_0>0$.  Then their OU-SNIS average satisfies $\widehat H_{\GN}(y,t)\succeq\lambda_0 I_d$, the LFGI resolvent
\[
    \widehat A_{\GN}=\alphat^2 I_d+\gammat\widehat H_{\GN}
\]
satisfies $\widehat A_{\GN}\succeq (\alphat^2+\gammat\lambda_0)I_d$, and the gate spectrum lies in
\[
    0<\lambda(\widehat G_{\GN})\le
    \frac{\alphat^2}{\alphat^2+\gammat\lambda_0}<1.
\]
Thus Gauss--Newton gates require no regularization/projection to avoid resolvent poles, unlike raw indefinite Hessian averages off the posterior mode.
\end{lemma}

Proof deferred to \cref{app:proof-gn-proxy-density}.

\begin{proposition}[Proxy-Hessian gate perturbation]
\label[proposition]{prop:proxy-hessian-gate-density}
Let $H=\E[H_0(X_0)\mid Y_t=y]$, let $\widetilde H=H+\Delta$ be a proxy conditional precision, and define
\[
    A=\alphat^2 I_d+\gammat H,
    \qquad
    \widetilde A=\alphat^2 I_d+\gammat\widetilde H .
\]
If $A\succ0$ and
\[
    \varepsilon_{\rm prox}=\left\|\gammat A^{-1/2}\Delta A^{-1/2}\right\|_{\op}<1,
\]
then, for $\Gstar=\alphat^2A^{-1}$ and $\widetilde G=\alphat^2\widetilde A^{-1}$,
\[
    \|\widetilde G-\Gstar\|_{\op}
    \le
    \alphat^2\|A^{-1/2}\|_{\op}^2
    \frac{\varepsilon_{\rm prox}}{1-\varepsilon_{\rm prox}},
\]
and the corresponding local excess risk is bounded by the excess-risk identity in \cref{eq:excess-risk} with $\widetilde G-\Gstar$.
\end{proposition}

Proof deferred to \cref{app:proof-gn-proxy-density}.

The conditional-mean identity does not depend on the Gauss--Newton approximation.  By \eqref{eq:measurable-gate-target}, any measurable gate leaves the blended Tweedie--TSI signal with conditional mean $s_t(y)$.  The Gauss--Newton choice is therefore a risk/conditioning decision: it drops sign-indefinite residual second-derivative fluctuations in exchange for a PSD precision whose resolvent is stable and whose divergence can be differentiated exactly.

The decomposition above also explains the approximation error.  Near a well-fit posterior region, $R_{\rm nl}$ is small or mean-zero to first order under the model, while its sign-indefinite fluctuations can dominate the variance and pole risk of a finite SNIS Hessian average.  This is the standard Gauss--Newton approximation logic \citep[Ch.~10]{nocedal2006numerical}.  The GN gate therefore introduces the usual Gauss--Newton proxy bias while replacing the indefinite residual-Hessian term by a PSD proxy.

\pagebreak

\section{Controlled Validation and Density-Evaluation Benchmarks}
\label{sec:validation-plan}
The validation benchmarks are organized by the failure of scalar Tweedie--TSI coefficients to match the direction-dependent risk-minimizing matrix gate in \cref{cor:scalar-failure}, the centered-primal residual coupling in \cref{subsec:bimodal-primal-failure-scale}, and the density-evaluation construction in \cref{sec:pf-density-evaluation-body}.  The singular Gaussian calculation is used analytically to show that scalar blending of score estimators is structurally insufficient, but it is not used as a separation experiment because centered primal regression has an exact Gaussian cancellation after centering.  The main controlled example is instead the misaligned singular-subspace Gaussian mixture: it requires a spatially varying matrix gate and produces the residual-coupled centered-primal finite-reference error computed in \cref{subsec:bimodal-primal-failure-scale}.  The $d=8$ mixture in \cref{subsec:exp-misaligned-gmm-d8} tests this mechanism directly, and the $d=24$ mixture in \cref{subsec:exp-misaligned-gmm-d24} checks whether it survives a moderate dimension increase.  \Cref{subsec:exp-funnel-d10} adds an explicitly non-Gaussian, non-GMM Neal funnel in $d=10$.  The Darcy-flow posterior in \cref{subsec:density-experiments} tests the density-evaluation application in the intended scientific-computing regime and separates robust posterior-bulk calibration from a tail-sensitive local Gaussian density fit.

\subsection{Comparison estimators and shared evaluation setup}
\label{subsec:comparison-estimators-setup}

The first column of \cref{tab:comparison-estimators} fixes the estimator labels used in the result tables.  The Uniform Scalar Blend and Uniform Matrix Blend rows are finite-reference versions of the spatially uniform estimators in \cref{eq:s-unif-sc,eq:s-unif-mat}; they correspond respectively to the scalar control-variate score-matching schedule and the DPSMC matrix control-variate schedule.  The remaining blended rows are query-dependent estimators: Scalar Blend, Matrix Blend, LFGI, and LFGI--GN estimate gates at each query $(y,t)$.  This comparison separates time-only control-variate scheduling from spatially varying finite-reference gates.

For a query $(y,t)$, let $w_i^{\rm s}(y,t)$ and $w_i^{\rm g}(y,t)$ denote normalized OU weights on independent score and gate banks.  The score bank forms
\[
    \widehat b^{\rm s}(y,t)=\sum_i w_i^{\rm s}(y,t)b(X_i^{\rm s};y,t),
    \qquad
    \widehat c^{\rm s}(y,t)=\sum_i w_i^{\rm s}(y,t)c(X_i^{\rm s};t),
    \qquad
    \widehat\delta^{\rm s}=\widehat c^{\rm s}-\widehat b^{\rm s}.
\]
For the gate bank, write $b_i^{\rm g}=b(X_i^{\rm g};y,t)$, $c_i^{\rm g}=c(X_i^{\rm g};t)$, $\delta_i^{\rm g}=c_i^{\rm g}-b_i^{\rm g}$, and let bars denote the corresponding gate-bank weighted means.  The scalar least-squares coefficient used by Scalar Blend is
\[
    \widehat g_{\rm SB}^{\rm g}(y,t)
    =
    \operatorname{clip}_{[0,1]}
    \left(
    -\frac{\sum_i w_i^{\rm g}\inner{b_i^{\rm g}-\bar b^{\rm g}}{\delta_i^{\rm g}-\bar\delta^{\rm g}}}
    {\sum_i w_i^{\rm g}\|\delta_i^{\rm g}-\bar\delta^{\rm g}\|^2+\varepsilon_{\rm SB}}
    \right),
\]
with the numerical denominator floor recorded in Appendix \cref{app:validation-run-setup}.  The matrix least-squares gate $\widehat G_{\rm cen}^{\rm g}$ is the centered primal gate in \cref{eq:centered-primal-gate,eq:centered-covariances} computed on the gate bank.

\begin{table}[H]
\centering
\footnotesize
\setlength{\tabcolsep}{3.2pt}
\renewcommand{\arraystretch}{1.17}
\begin{tabular}{@{}p{0.19\linewidth}p{0.27\linewidth}p{0.54\linewidth}@{}}
\toprule
Method label & Finite-reference score field & Gate construction \\
\midrule
Tweedie &
$\widehat s_{\rm TWD}=\widehat b^{\rm s}$ &
No gate; this is the OU-SNIS Tweedie signal. \\

Uniform Scalar Blend &
$\widehat s_{\rm USB}=\widehat c^{\rm s}+\widehat a_{\rm sc}^{\rm g}(t)(\widehat b^{\rm s}-\widehat c^{\rm s})$ &
Finite-reference version of the spatially uniform scalar estimator in \cref{eq:s-unif-sc}, with $\widehat a_{\rm sc}^{\rm g}(t)$ estimated from the gate-bank target-score second moment. \\

Scalar Blend &
$\widehat s_{\rm SB}=\widehat b^{\rm s}+\widehat g_{\rm SB}^{\rm g}\widehat\delta^{\rm s}$ &
Clipped local scalar least-squares coefficient estimated from the gate-bank Tweedie--TSI disagreement. \\

Uniform Matrix Blend &
$\widehat s_{\rm UMB}=\widehat c^{\rm s}+\widehat A_{\rm mat}^{\rm g}(t)(\widehat b^{\rm s}-\widehat c^{\rm s})$ &
Finite-reference version of the spatially uniform matrix estimator in \cref{eq:s-unif-mat}, with $\widehat A_{\rm mat}^{\rm g}(t)$ estimated from the gate-bank target-score second moment. \\

Matrix Blend &
$\widehat s_{\rm MB}=\widehat b^{\rm s}+\widehat G_{\rm cen}^{\rm g}\widehat\delta^{\rm s}$ &
Centered primal weighted-regression gate from \cref{eq:centered-primal-gate}.  This estimator uses no Hessians and targets the same population gate as LFGI. \\

LFGI (ours) &
$\widehat s_{\rm LFGI}=\widehat b^{\rm s}+\widehat G_{\rm LFGI}^{\rm g}\widehat\delta^{\rm s}$ &
$\widehat G_{\rm LFGI}^{\rm g}=\alphat^2(\alphat^2I_d+\gammat\widehat H^{\rm g})^{-1}$ with $\widehat H^{\rm g}=\sum_i w_i^{\rm g}H_0(X_i^{\rm g})$.  The controlled synthetic benchmarks use the raw Hessian average without truncation or PSD projection. \\

LFGI--GN (ours) &
$\widehat s_{\rm GN}=\widehat b^{\rm s}+\widehat G_{\rm GN}^{\rm g}\widehat\delta^{\rm s}$ &
$\widehat G_{\rm GN}^{\rm g}=\alphat^2(\alphat^2I_d+\gammat\widehat P_{\rm GN}^{\rm g})^{-1}$ with $\widehat P_{\rm GN}^{\rm g}=\sum_i w_i^{\rm g}P_i^{\rm GN}$.  This is the inverse-problem density-evaluation variant. \\
\bottomrule
\end{tabular}
\caption{Estimator labels and finite-reference score fields used in the validation and density-evaluation tables.  The score bank forms the Tweedie and target-score averages; the independent gate bank forms the uniform schedules and the spatially varying least-squares or Hessian gates.}
\label{tab:comparison-estimators}
\end{table}

Each sampling row substitutes the corresponding finite-reference score estimator for $s_t$ in the reverse SDE \eqref{eq:ou-reverse-sde} and uses the shared discretization recorded in Appendix \cref{app:validation-run-setup}.  Density-evaluation rows freeze the fitted score field and substitute it into the probability-flow density path \eqref{eq:pf-density-intro}, using the divergence and evidence-diagnostic construction from \cref{sec:pf-density-evaluation-body} and the density-evaluation setups in Appendix \cref{app:darcy-pilot-setup,app:pilot-split-setup}.  To match the independence condition in \eqref{eq:measurable-gate-target} and \cref{def:finite-lfgi-score}, the matrix-gate benchmarks use independent score and gate banks.  The two GMM fixed-budget tables and the funnel table use $\Ns=\Ng=250$, $N_{\rm gen}=N_{\rm test}=12000$, $300$ reverse steps, $t_{\min}=0.01$, $t_{\max}=3.0$, and double precision; the exact OU weights, Heun update, regularization, and repeated-run configuration are recorded in Appendix \cref{app:validation-run-setup}.  The reference-count sweep varies the matched score/gate bank size while holding the sampler and metric definitions fixed.  When the exact OU-marginal score is unavailable, the score-RMSE reference for Appendix \cref{app:metrics}, Def.~\ref{def:score-rmse}, is a high-reference SNIS--Tweedie estimate from held-out samples.

The fixed-budget sampling tables use a consistent metric order: Sliced KS, MMD, KDE NLL, and score RMSE or high-reference score-RMSE proxy.  Qualitatively, these check robust one-dimensional marginal agreement, kernel discrepancy, held-out KDE likelihood, and noisy-score error, respectively; the formal definitions are in Appendix \cref{app:metrics}, Defs.~\ref{def:sliced-ks}--\ref{def:score-rmse}.  Density-evaluation diagnostics check posterior-energy rank correlation, central affine calibration, residual scale, pointwise probability-flow likelihood, and correction-weight overlap; their formal definitions are in Appendix \cref{app:density-metrics}, Defs.~\ref{def:central-density}--\ref{def:evidence-diagnostics}.  The auxiliary rows ``GT floor'' and ``MAP--Laplace,'' when present, are not learned score-estimator rows: GT floor is an independent reference-sample floor for sample-quality metrics, and MAP--Laplace is the density proxy described in \cref{subsec:density-experiments}.

\subsection{Validation goals}

The validation suite tests three linked claims.  First, once scalar interpolation has been ruled out by the singular-Gaussian spectral calculation, the misaligned GMM checks whether a spatially varying matrix gate is needed when component responsibilities rotate the stiff subspace, including against spatially uniform scalar and matrix control-variate schedules.  Second, the GMM and gate-capture sweeps test whether finite-reference LFGI estimates that matrix gate more reliably than the centered primal Matrix Blend at matched bank size.  Third, the Darcy-flow and known-normalization benchmarks test whether the same finite-reference score construction can be frozen into a probability-flow density surrogate for inverse-problem posterior diagnostics.  Robust central density metrics are computed on the central posterior-energy band defined in Appendix \cref{app:density-metrics}, Def.~\ref{def:central-density}, rather than on extreme held-out energy tails.  Curl is used only as the finite-difference skew-Jacobian diagnostic defined in Appendix \cref{app:metrics}, Def.~\ref{def:curl}.

\vspace{-0.5em}

\subsection{Score-risk and gate-capture sweeps}
\label{subsec:validation-score-gate-sweeps}

The two diagnostics closest to the theory are score RMSE for the conditional score-risk claim and gate-capture error for the excess-risk geometry in \cref{eq:excess-risk,eq:lfgi-risk-capture-condition}.  We use time-averaged noisy-score RMSE (Appendix \cref{app:metrics}, Def.~\ref{def:score-rmse}) as the primary reference-count sweep.  To test the finite-reference gate claim, we directly compare LFGI gate estimation against centered primal matrix gate estimation using the gate-capture errors in Appendix \cref{app:metrics}, Def.~\ref{def:gate-capture}.  Downstream sample-quality tables use the sampling metrics summarized in \cref{subsec:comparison-estimators-setup}.

The first controlled validation is the score-RMSE reference-count sweep.  The reference-count sweep varies a matched bank size $N$ with $\Ns=\Ng=N$ and compares Tweedie, Uniform Scalar Blend, Scalar Blend, Uniform Matrix Blend, Matrix Blend, and LFGI on the $d=8$ misaligned singular-subspace GMM.  The score-RMSE diagnostic (see Appendix \cref{app:metrics}, Def.~\ref{def:score-rmse}) evaluates each finite-reference estimator against the exact OU-marginal score available for the analytic target, averaged over the displayed diffusion-time grid.  This is the most direct estimator-level test of the theory: it asks whether the LFGI gate yields a lower conditional score error than spatially uniform blending, spatially varying scalar interpolation, and centered primal matrix estimation at the same reference count.  Endpoint sample metrics are reported in the fixed-budget tables below, where they are interpreted as downstream sampler diagnostics rather than as the primary evidence for the local score-risk claim.

\begin{figure}[H]
\centering
\includegraphics[width=0.5\textwidth]{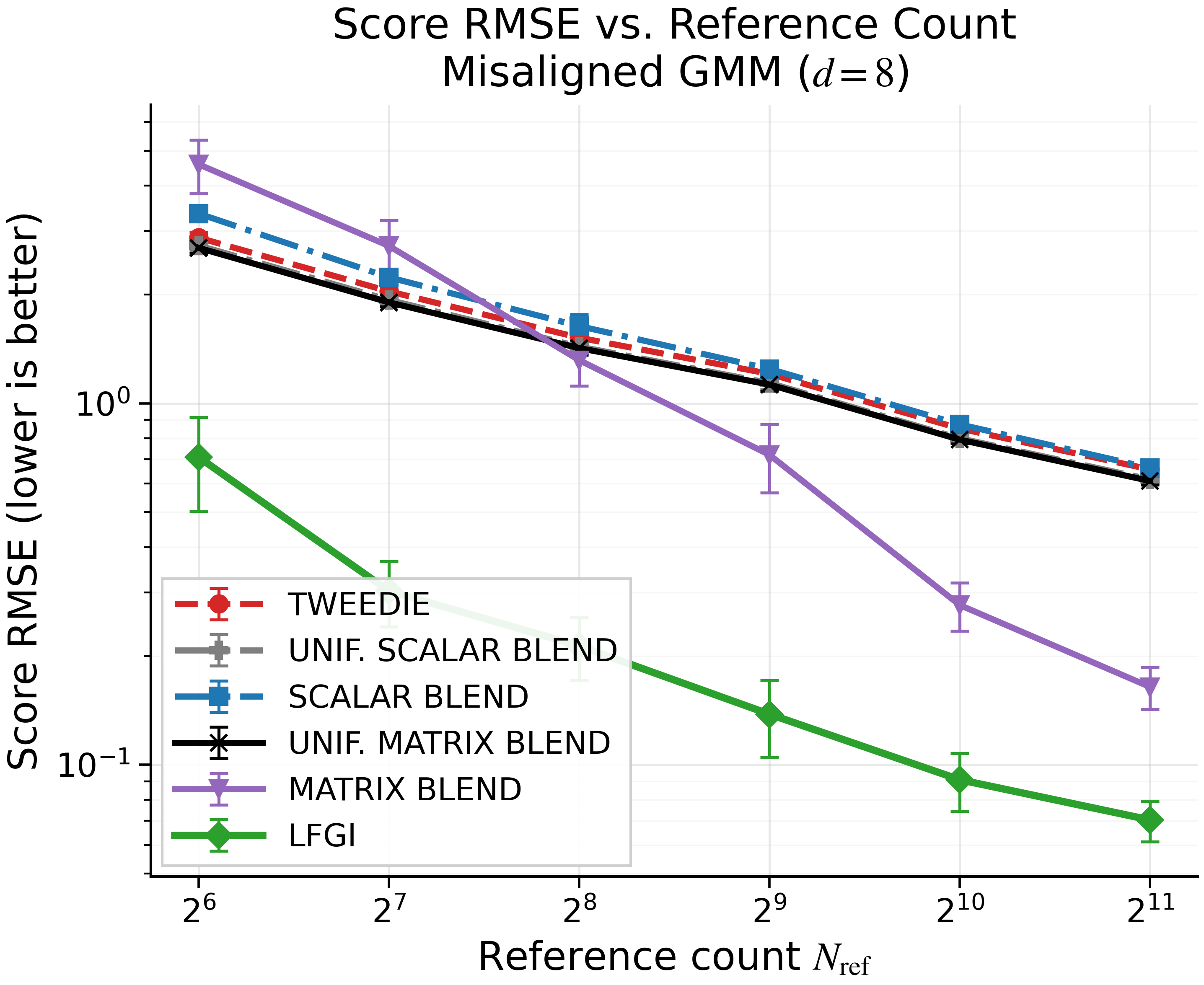}
\caption{Reference-count score-RMSE sweep on the $d=8$ misaligned singular-subspace GMM.  The metric is the time-averaged noisy-score RMSE (Appendix \cref{app:metrics}, Def.~\ref{def:score-rmse}).  LFGI remains below the other learned estimators across the displayed reference-bank sizes.}
\label{fig:validation-score-rmse-sweep}
\end{figure}

The second controlled validation is the gate-capture finite-reference count sweep.  The gate-estimation sweep holds the target, query-time grid, and score-bank construction fixed while varying the gate-bank size.  We compare the LFGI estimator, which applies $\Psiop$ to $\widehat H_N$, against centered primal matrix estimation of the same population gate.  The risk-weighted gate error is the excess-risk quadratic form $\operatorname{tr}\{(\widehat G-\Gstar)\Cdd(\widehat G-\Gstar)^\top\}$.  The main error geometry is the risk-weighted norm induced by the disagreement covariance $\Cdd=\E[(c-b)(c-b)^\top\mid Y_t=y]$, because this is the geometry that controls excess conditional score risk.  The risk-weighted and Frobenius gate-capture diagnostics are defined in Appendix \cref{app:metrics}, Def.~\ref{def:gate-capture}.  The granular time-by-time sweep is reported in Appendix \cref{app:gate-capture-breakdown}; across the displayed times, the LFGI estimator reaches a substantially smaller gate error at the same gate-bank size.

Together, \cref{fig:validation-score-rmse-sweep} and Appendix \cref{app:gate-capture-breakdown} provide the direct estimator-level validation.
The score-RMSE sweep tests the score estimator itself, and the gate-capture sweep tests the exact excess-risk geometry that appears in the finite-reference theory.  We do not use \cref{eq:relative-advantage-regime} as an empirical calibration plot: that inequality is a sufficient-condition guide for the distinct error mechanisms, not a sharp predictor.

\vspace{-0.7em}
\subsection{\texorpdfstring{Experiment I: misaligned singular-subspace GMM in $d=8$}{Experiment I: misaligned singular-subspace GMM in d=8}}
\label{subsec:exp-misaligned-gmm-d8}

The two-component calculation in \cref{subsec:misaligned-bimodal-gmm} isolates the mechanism; the validation target keeps that mechanism but uses $K=8$ components to make the multimodal sampling problem nontrivial.  The target has $K=8$ components in $d=8$, intrinsic rank $3$, component radius $3$, and normal scale $\sigma_\perp=0.035$.  Each component has PSD precision, so the Hessian-average map remains pole-separated, but the local stiff eigenspaces of the component precisions, and hence of $\Gstar(y,t)$, rotate across components.

\begin{table}[H]
\centering
\small
\begin{tabular}{@{}lcccc@{}}
\toprule
Method & Sliced KS $\downarrow$ & MMD $\downarrow$ & NLL $\downarrow$ & Score RMSE $\downarrow$ \\
\midrule
Tweedie & $0.0732{\pm}0.0109$ & $(2.60{\pm}0.42){\times}10^{-3}$ & $6.144{\pm}0.022$ & $7.316{\pm}0.950$ \\
Uniform Scalar Blend & $0.0714{\pm}0.0096$ & $(2.50{\pm}0.46){\times}10^{-3}$ & $6.128{\pm}0.018$ & $6.280{\pm}0.684$ \\
Scalar Blend & $0.0725{\pm}0.0102$ & $(3.00{\pm}0.47){\times}10^{-3}$ & $6.246{\pm}0.028$ & $6.778{\pm}0.788$ \\
Uniform Matrix Blend & $0.0726{\pm}0.0088$ & $(2.80{\pm}0.48){\times}10^{-3}$ & $6.135{\pm}0.016$ & $6.472{\pm}0.451$ \\
Matrix Blend & $0.0697{\pm}0.0068$ & $(3.20{\pm}0.71){\times}10^{-3}$ & $6.174{\pm}0.029$ & $7.449{\pm}0.430$ \\
LFGI (ours) & $\mathbf{0.0649{\pm}0.0087}$ & $\mathbf{(1.80{\pm}0.53){\times}10^{-3}}$ & $\mathbf{6.045{\pm}0.014}$ & $\mathbf{2.853{\pm}0.171}$ \\
GT floor & $0.0513{\pm}0.0053$ & $(0.577{\pm}0.204){\times}10^{-3}$ & $6.054{\pm}0.009$ & -- \\
\bottomrule
\end{tabular}
\caption{Misaligned singular-subspace GMM results in $d=8$, averaged over $5$ runs with $\Ns=\Ng=250$.  Columns follow \cref{subsec:comparison-estimators-setup}; entries are mean $\pm$ one standard deviation.  Bold values mark the best learned estimator in each column.}
\label{tab:misaligned-gmm-d8-results}
\end{table}
This experiment tests whether the direction-dependent attenuation rule from the Gaussian calculation remains useful when local mixture responsibilities vary, and whether spatially uniform matrix blending or centered primal matrix estimation closes the gap to LFGI.  LFGI reduces score RMSE to $2.853\pm0.171$, compared with $6.280\pm0.684$ for Uniform Scalar Blend, $6.472\pm0.451$ for Uniform Matrix Blend, $6.778\pm0.788$ for Scalar Blend, $7.316\pm0.950$ for Tweedie, and $7.449\pm0.430$ for Matrix Blend.  It also gives the best learned-estimator Sliced KS, MMD, NLL, and score RMSE.  The uniform gates improve some metrics relative to their local finite-reference counterparts, but they do not close the score-risk gap to LFGI.  The heatmaps in \cref{fig:misaligned-gmm-d8-heatmaps} show the same qualitative pattern: LFGI preserves the anisotropic component geometry more faithfully than the global and scalar baselines and avoids the finite-reference degradation of the centered primal matrix gate.

\begin{figure}[H]
\centering
\includegraphics[width=.99\textwidth]{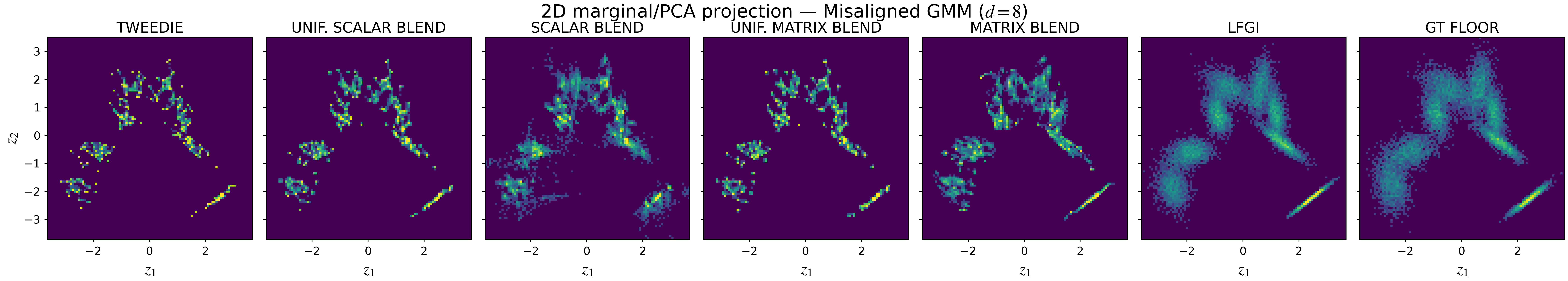}
\caption{Misaligned singular-subspace GMM, $K=8$, $d=8$, rank $3$: two-dimensional marginal/PCA projections.  Panels follow the method titles.}
\label{fig:misaligned-gmm-d8-heatmaps}
\end{figure}
\subsection{\texorpdfstring{Experiment II: misaligned singular-subspace GMM in $d=24$}{Experiment II: misaligned singular-subspace GMM in d=24}}
\label{subsec:exp-misaligned-gmm-d24}

The second GMM target keeps the same family and increases the ambient dimension to $d=24$, with intrinsic rank $4$, component radius $4.5$, and normal scale $\sigma_\perp=0.035$.  The example remains a controlled PSD/pole-separated mixture, but the higher dimension makes the singular-score moment-estimation problem substantially harder.

\begin{table}[H]
\centering
\small
\begin{tabular}{@{}lcccc@{}}
\toprule
Method & Sliced KS $\downarrow$ & MMD $\downarrow$ & NLL $\downarrow$ & Score RMSE $\downarrow$ \\
\midrule
Tweedie & $0.0689{\pm}0.0108$ & $(2.90{\pm}0.30){\times}10^{-3}$ & $18.791{\pm}0.030$ & $10.157{\pm}1.231$ \\
Uniform Scalar Blend & $0.0717{\pm}0.0043$ & $(3.40{\pm}0.93){\times}10^{-3}$ & $18.765{\pm}0.025$ & $9.823{\pm}0.961$ \\
Scalar Blend & $0.0782{\pm}0.0114$ & $(3.70{\pm}0.44){\times}10^{-3}$ & $18.966{\pm}0.060$ & $11.580{\pm}0.971$ \\
Uniform Matrix Blend & $0.0770{\pm}0.0110$ & $(3.40{\pm}0.98){\times}10^{-3}$ & $18.773{\pm}0.032$ & $9.650{\pm}0.508$ \\
Matrix Blend & $0.0808{\pm}0.0087$ & $(4.00{\pm}0.88){\times}10^{-3}$ & $18.835{\pm}0.046$ & $11.755{\pm}0.574$ \\
LFGI (ours) & $\mathbf{0.0627{\pm}0.0113}$ & $\mathbf{(1.60{\pm}0.53){\times}10^{-3}}$ & $\mathbf{18.701{\pm}0.010}$ & $\mathbf{4.470{\pm}0.199}$ \\
GT floor & $0.0527{\pm}0.0050$ & $(0.571{\pm}0.112){\times}10^{-3}$ & $18.740{\pm}0.005$ & -- \\
\bottomrule
\end{tabular}
\caption{Same fixed-budget setup and column convention as \cref{tab:misaligned-gmm-d8-results}, now for $d=24$ with intrinsic rank $4$.}
\label{tab:misaligned-gmm-d24-results}
\end{table}
The $d=24$ result serves as a compact scaling check.  LFGI reduces score RMSE to $4.470\pm0.199$, compared with $9.650\pm0.508$ for Uniform Matrix Blend, $9.823\pm0.961$ for Uniform Scalar Blend, $10.157\pm1.231$ for Tweedie, $11.755\pm0.574$ for Matrix Blend, and $11.580\pm0.971$ for Scalar Blend.  It also gives the best learned-estimator Sliced KS, MMD, NLL, and score RMSE, so the lower score-RMSE is accompanied by better sample-quality diagnostics at the larger dimension.  The qualitative PCA panels and extended interpretation are moved to \cref{app:d24-qualitative-panels} to preserve main-text space for the density-evaluation application.

\subsection{\texorpdfstring{Experiment III: Neal funnel in $d=10$}{Experiment III: Neal funnel in d=10}}
\label{subsec:exp-funnel-d10}

The Neal-funnel experiment in \cref{subsec:exp-funnel-d10} adds an explicitly non-Gaussian and non-mixture target.  The Neal funnel is defined hierarchically by a global coordinate $x_1\sim \N(0,\sigma_f^2)$ and conditional coordinates $x_{2:d}\mid x_1\sim\N(0,\exp(x_1)I_{d-1})$.  The displayed run uses $d=10$ and $\sigma_f^2=6$.  Unlike the GMM examples, the target is not assembled from locally Gaussian mixture components with fixed PSD precisions.  Its conditional scale changes exponentially across the funnel coordinate: the negative-$x_1$ neck is stiff and nearly singular, while the positive-$x_1$ mouth is diffuse.  This tests whether LFGI remains useful when $H_0$ is state-dependent and not generated by the PSD-GMM structure of the preceding validation targets.  The exact target construction is recorded in \cref{app:target-funnel-d10}.

The funnel run uses the same independent-bank sampler setup as the GMM tables.  Table~\ref{tab:funnel-d10-results} reports mean $\pm$ one standard deviation over five independent runs, with the high-reference score-RMSE proxy used because the exact OU-marginal score is not available in closed form.

The funnel also clarifies the scope of the pole-separation theory in \cref{sec:reference-complexity}.  The raw funnel observed information has indefinite regions, so the PSD argument used for the GMM targets does not apply.  The finite-reference LFGI theorems in \cref{prop:risk-weighted-lfgi-perturbation,thm:hessian-reference-condition} instead control the shifted conditional matrix $A(y,t)=\alphat^2I_d+\gammat H(y,t)$ through the pole factor in \cref{eq:Lambda-pole}, the relative Hessian error $\epsilon_H$, and the risk-capture inequality in \cref{eq:lfgi-risk-capture-condition}.  The auxiliary audit in \cref{app:funnel-pole-diagnostics} evaluates these quantities on the funnel query distribution.  Across the six audited times in \cref{tab:funnel-pole-audit-summary}, the observed nonpositive-$A$ rate and active-pole rate are both zero.  The strict finite-reference certificate is not uniform: the pass rate falls to $0.510$ at $t=0.3$, the largest reported $90\%$ quantile of $\epsilon_H$ is $0.916$, and the largest reported $90\%$ quantile of the perturbation-to-gain ratio in \cref{eq:lfgi-risk-capture-condition} is $127.0$.  Thus the funnel run is not certified by the strict finite-reference sufficient condition at all audited times, but the audit does not show the shifted-resolvent pole failure mode.
\vspace{-1.0em}
\begin{figure}[H]
\centering
\includegraphics[width=0.99\textwidth]{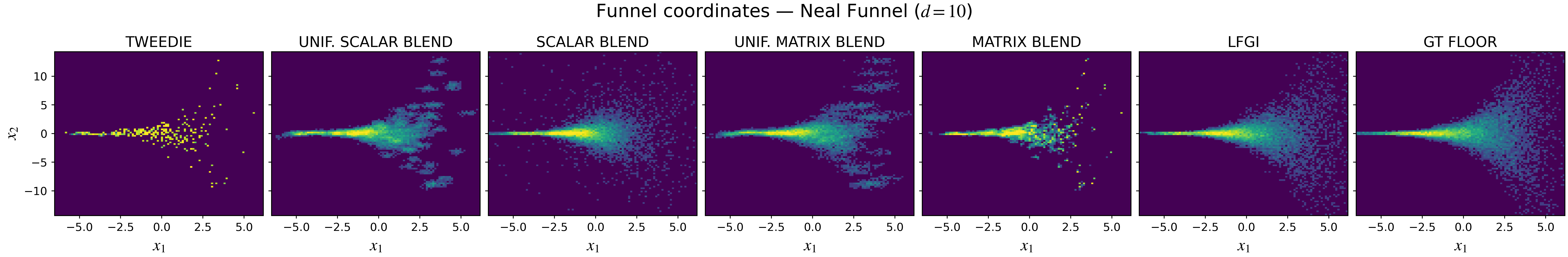}
\caption{Neal funnel in $d=10$: two-dimensional funnel-coordinate histograms.  Panels follow the method titles; robust coordinate limits show the narrow neck and widening mouth simultaneously.}
\label{fig:funnel-d10-heatmaps}
\end{figure}

\vspace{-1.0em}

\begin{table}[H]
\centering
\small
\begin{tabular}{@{}lcccc@{}}
\toprule
Method & Sliced KS $\downarrow$ & MMD $\downarrow$ & NLL $\downarrow$ & Score RMSE proxy $\downarrow$ \\
\midrule
Tweedie & $0.0721{\pm}0.0061$ & $(3.90{\pm}0.57){\times}10^{-3}$ & $223.4{\pm}28.1$ & $148.7{\pm}72.3$ \\
Uniform Scalar Blend & $0.0688{\pm}0.0051$ & $(3.80{\pm}0.86){\times}10^{-3}$ & $212.3{\pm}24.5$ & $62.1{\pm}45.8$ \\
Scalar Blend & $0.0902{\pm}0.0104$ & $(6.70{\pm}2.30){\times}10^{-3}$ & $181.8{\pm}25.5$ & $9.308{\pm}0.681$ \\
Uniform Matrix Blend & $0.0670{\pm}0.0051$ & $(3.40{\pm}1.20){\times}10^{-3}$ & $209.2{\pm}24.1$ & $54.1{\pm}13.4$ \\
Matrix Blend & $0.0805{\pm}0.0098$ & $(4.80{\pm}1.20){\times}10^{-3}$ & $218.8{\pm}27.3$ & $123.9{\pm}20.7$ \\
LFGI (ours) & $\mathbf{0.0614{\pm}0.0053}$ & $\mathbf{(1.50{\pm}0.50){\times}10^{-3}}$ & $\mathbf{152.4{\pm}20.8}$ & $\mathbf{9.280{\pm}1.838}$ \\
GT floor & $0.0528{\pm}0.0082$ & $(0.727{\pm}0.078){\times}10^{-3}$ & $145.4{\pm}22.3$ & -- \\
\bottomrule
\end{tabular}
\caption{Same fixed-budget reporting convention as \cref{tab:misaligned-gmm-d8-results}, now on Neal's funnel in $d=10$.  The last column uses the high-reference score-RMSE proxy; the ground-truth floor is a finite-sample reference floor.}
\label{tab:funnel-d10-results}
\end{table}

LFGI is the best learned estimator on the reported funnel metrics.  Its Sliced KS is close to the ground-truth floor, $0.0614\pm0.0053$ for LFGI versus $0.0528\pm0.0082$ for the floor, and its MMD and NLL are lower than those of the other learned estimators.  For this non-analytic funnel target, the score-RMSE value should be read as a finite-reference proxy for noisy-score error rather than as an independent closed-form exact score; the distributional conclusions rest on Sliced KS, MMD, NLL, and the qualitative funnel-coordinate panel.  The qualitative panel in \cref{fig:funnel-d10-heatmaps} shows the corresponding sample-shape pattern.  Tweedie under-resolves the funnel geometry, the spatially uniform gates improve some bulk diagnostics but do not match the spatially varying LFGI score proxy, Scalar Blend partially opens the funnel but produces a less faithful sample cloud, Matrix Blend does not turn its full matrix class into a better finite-reference sampler, and LFGI most closely matches the neck-to-mouth structure of the reference.  The important point for this paper is not that the funnel satisfies the same clean PSD GMM assumptions as the first two experiments.  It does not.  Rather, the result shows that the LFGI gate can remain useful on a genuinely non-Gaussian, non-mixture target with state-dependent anisotropy.

\subsection{Darcy-flow density evaluation for inverse problems}
\label{subsec:density-experiments}

We evaluate normalized-density estimation on a nonlinear elliptic inverse problem for Darcy flow.  On the unit square $D=[0,1]^2$, the pressure $p$ solves
\[
    -\nabla\!\cdot\{k(x;\alpha)\nabla p(x;\alpha)\}=1,
    \qquad x\in D,
    \qquad p|_{\partial D}=0.
\]
The unknown is the log-permeability field
\[
    m(x;\alpha)=\sum_{j=1}^{32}\alpha_j\phi_j(x),
    \qquad
    k(x;\alpha)=\exp(m(x;\alpha)),
    \qquad
    \alpha\sim\N(0,I_{32}),
\]
where the basis is obtained from the leading modes of an exponential covariance kernel on a $32\times32$ grid.  The observation map reads the pressure at $120$ interior sensors and adds Gaussian noise with standard deviation $10^{-3}$,
\[
    y_{\rm obs}=S p(\cdot;\alpha_\star)+\eta,
    \qquad
    \eta\sim\N(0,\sigma^2 I),
    \qquad
    \sigma=10^{-3}.
\]
Up to an additive constant, the posterior energy is therefore
\[
    E(\alpha)=-\log\tilde p(\alpha)
    =
    \frac12\|\alpha\|^2
    +
    \frac{1}{2\sigma^2}\|S p(\cdot;\alpha)-y_{\rm obs}\|^2 .
\]

The posterior score for the Darcy-flow benchmark in \cref{subsec:density-experiments} is available by differentiating the JAX forward solve, as recorded in \cref{app:darcy-gn-precision-details}.
The LFGI--GN density row uses the positive-semidefinite prior-plus-Gauss--Newton precision proxy from \cref{sec:gn-gates-body}; throughout the inverse-problem density examples, LFGI--GN denotes this Gauss--Newton gate construction plugged into the frozen finite-reference probability-flow surrogate of \cref{def:frozen-pf-surrogate}.  The posterior energy and score still correspond to the nonlinear target above.  The MAP--Laplace baseline is the Gaussian density model $q_{\rm Lap}=\N(x_{\rm MAP},H_E(x_{\rm MAP})^{-1})$, where $x_{\rm MAP}$ minimizes the posterior energy $E=-\log\tilde p$ and $H_E$ is the Hessian of $E$ at that point.  Thus the Darcy-flow benchmark in \cref{subsec:density-experiments} tests whether the finite-reference probability-flow density from \cref{subsec:pf-density-frozen-field} tracks the nonlinear posterior energy on held-out pilot samples, and compares that density with this local Gaussian baseline.

The Darcy-flow density-evaluation setup in \cref{subsec:density-experiments} constructs MALA posterior banks and splits posterior samples into score-signal, gate, and density-evaluation roles as in \cref{app:pilot-split-setup}, freezes the score estimator, and integrates the probability-flow density equation in \eqref{eq:pf-change-of-variables-density}.  Each seeded run uses $N_{\rm signal}=1000$, $N_{\rm gate}=1000$, and $N_{\rm eval}=1000$, with independent score, gate, and MALA-EVAL density-evaluation banks.  Under the independent score/gate coupling, the source MALA bank therefore contains $2000=1000+1000$ retained posterior samples, which are sliced into disjoint score-signal and gate banks; a separate MALA-EVAL run contributes the $1000$ held-out held-out density-evaluation samples.  Each retained posterior sample is produced by an adjoint-gradient MALA chain with $600$ transition steps, $200$ burn-in steps, step size $5\times10^{-5}$, and no additional thinning.  The probability-flow path uses $64$ time steps over $t\in[10^{-2.5},5]$.  Because the held-out density-evaluation samples include heavy energy tails, the plotted and tabled affine diagnostics use the central $3$--$97\%$ energy band as the fixed robust posterior-bulk convention in Appendix \cref{app:density-metrics}, Def.~\ref{def:central-density}.

\begin{table}[H]
\centering
\small
\begin{tabular}{@{}lrrrrr@{}}
\toprule
Method & Spearman $\uparrow$ & Central slope $\to 1$ & Central $R^2$ $\uparrow$ & Central RMSE $\downarrow$ & NLL $\downarrow$ \\
\midrule
Tweedie & $0.598\pm0.016$ & $1.72\pm0.03$ & $0.377\pm0.027$ & $40.0\pm0.9$ & $136.3\pm3.4$ \\
Uniform Scalar Blend & $0.579\pm0.024$ & $7.3\pm12.6$ & $0.298\pm0.148$ & $607\pm1263$ & $873\pm1650$ \\
Scalar Blend & $0.19\pm0.27$ & $(-1.6\pm12.7){\times}10^9$ & $0.004\pm0.003$ & $(4.3\pm4.6){\times}10^{12}$ & $(1.1\pm1.2){\times}10^{13}$ \\
Uniform Matrix Blend & $0.48\pm0.26$ & $6.0\pm11.0$ & $0.294\pm0.249$ & $1005\pm2155$ & $431\pm685$ \\
Matrix Blend & $0.660\pm0.023$ & $2.02\pm0.04$ & $0.432\pm0.038$ & $42.1\pm2.0$ & $127.6\pm2.1$ \\
LFGI--GN (ours) & $\mathbf{0.949\pm0.015}$ & $\mathbf{0.86\pm0.02}$ & $\mathbf{0.86\pm0.06}$ & $\mathbf{6.3\pm1.6}$ & $\mathbf{28.8\pm1.2}$ \\
MAP--Laplace & $0.862\pm0.004$ & $2.45\pm0.11$ & $0.626\pm0.009$ & $34.3\pm1.9$ & $198\pm62$ \\
\bottomrule
\end{tabular}
\caption{Darcy-flow density-evaluation results, averaged over $5$ seeded runs with $n_{\rm eval}=1000$ held-out MALA-EVAL samples per run.  Metrics are robust posterior-bulk diagnostics on the central $3$--$97\%$ energy band; definitions are in Appendix \cref{app:density-metrics}.  Entries are mean $\pm$ one standard deviation.}
\label{tab:darcy-density-results}
\end{table}

\Cref{tab:darcy-density-results} reports posterior-bulk diagnostics for the Darcy-flow density benchmark.  In the central $3$--$97\%$ posterior-energy band (Appendix \cref{app:density-metrics}, Def.~\ref{def:central-density}), LFGI--GN has central affine $R^2=0.86\pm0.06$ and central RMSE $6.3\pm1.6$, compared with $0.626\pm0.009$ and $34.3\pm1.9$ for MAP--Laplace.  LFGI--GN also gives the best rank correlation and the lowest pointwise NLL (Def.~\ref{def:pf-nll}), while Matrix Blend improves some rank and NLL diagnostics relative to Tweedie but remains much less centrally calibrated than LFGI--GN.  The Uniform Scalar Blend and Uniform Matrix Blend density rows are less stable on this inverse-problem posterior, consistent with a target whose useful curvature information is spatially varying rather than spatially uniform.  The posterior-bulk diagnostics therefore show a finite-reference LFGI advantage for central energy calibration and normalized-density NLL.  The corresponding density-energy grid is reported in Appendix \cref{app:darcy-density-grid}.

\subsection{Known-normalization calibration on analytic inverse problems}
\label{subsec:known-z-calibration}

The nonlinear Darcy-flow posterior is not analytically normalized, so absolute evidence error is measured on separate analytic inverse-problem calibration targets.  The first target is a linear-Gaussian positive control with a closed-form Gaussian posterior and closed-form evidence.  The second target uses a four-component Gaussian-mixture likelihood under the same Gaussian prior and a shared forward geometry, giving an exact non-Gaussian Gaussian-mixture posterior and a closed-form evidence.  The third target is again an analytic Gaussian-mixture inverse problem, but the component likelihoods use different right-singular frames, so posterior curvature is stiff and mode-dependent even though the evidence remains closed form.  The Gaussian target checks the implementation of the probability-flow log-density and known-evidence metric; the two mixture targets separate shared-geometry multimodality from misaligned anisotropic curvature.  MAP--Laplace is exact only in the Gaussian positive control.  The learned estimator labels are those in \cref{tab:comparison-estimators}, with LFGI--GN using the Gauss--Newton gate convention from \cref{sec:gn-gates-body}.

\begin{table}[H]
\centering
\footnotesize
\setlength{\tabcolsep}{3.2pt}
\renewcommand{\arraystretch}{1.12}
\begin{tabular}{@{}lcccc@{}}
\toprule
Method &
\begin{tabular}[c]{@{}c@{}}$\log\qpf$\\bias $\to 0$\end{tabular} &
\begin{tabular}[c]{@{}c@{}}$\log\qpf$\\RMSE $\downarrow$\end{tabular} &
\begin{tabular}[c]{@{}c@{}}$|\widehat{\log Z}-\log Z|$\\$\downarrow$\end{tabular} &
\begin{tabular}[c]{@{}c@{}}Correction\\ESS/$n$ $\uparrow$\end{tabular} \\
\midrule
Tweedie & $-18.4 \pm 0.10$ & $24.5 \pm 0.21$ & $18.4 \pm 0.10$ & $(6.09 \pm 2.38){\times}10^{-4}$ \\
Uniform Scalar Blend & $-1.91 \pm 0.04$ & $3.51 \pm 0.06$ & $1.91 \pm 0.04$ & $0.00107 \pm 5.68{\times}10^{-4}$ \\
Scalar Blend & $-5.01 \pm 1.4$ & $7.68 \pm 1.0$ & $5.01 \pm 1.4$ & $(5.13 \pm 0.19){\times}10^{-4}$ \\
Uniform Matrix Blend & $-0.0840 \pm 0.0091$ & $0.458 \pm 0.020$ & $0.0840 \pm 0.0091$ & $0.867 \pm 0.0047$ \\
Matrix Blend & $-1.77 \pm 0.13$ & $5.45 \pm 0.60$ & $1.77 \pm 0.13$ & $(5.00 \pm 0.003){\times}10^{-4}$ \\
LFGI--GN (ours) & $-0.0798 \pm 0.0094$ & $0.367 \pm 0.021$ & $0.0798 \pm 0.0094$ & $0.934 \pm 0.0037$ \\
MAP--Laplace & $\mathbf{0 \pm 0}$ & $\mathbf{0 \pm 0}$ & $\mathbf{0 \pm 0}$ & $\mathbf{1 \pm 0}$ \\
\bottomrule
\end{tabular}
\caption{Known-normalization calibration on the linear-Gaussian analytic inverse problem, averaged over $5$ seeded runs.  Entries are mean $\pm$ one run-to-run standard deviation.  The log-density and correction-weight diagnostics are defined in Appendix \cref{app:density-metrics}, Def.~\ref{def:evidence-diagnostics}; ESS is also reported as ESS/$n$.  Bold values mark the best method in each metric column, using absolute value for bias.}
\label{tab:known-z-gaussian-calibration}
\end{table}

\begin{table}[H]
\centering
\footnotesize
\setlength{\tabcolsep}{3.2pt}
\renewcommand{\arraystretch}{1.12}
\begin{tabular}{@{}lcccc@{}}
\toprule
Method &
\begin{tabular}[c]{@{}c@{}}$\log\qpf$\\bias $\to 0$\end{tabular} &
\begin{tabular}[c]{@{}c@{}}$\log\qpf$\\RMSE $\downarrow$\end{tabular} &
\begin{tabular}[c]{@{}c@{}}$|\widehat{\log Z}-\log Z|$\\$\downarrow$\end{tabular} &
\begin{tabular}[c]{@{}c@{}}Correction\\ESS/$n$ $\uparrow$\end{tabular} \\
\midrule
Tweedie & $-29.6 \pm 0.41$ & $38.2 \pm 0.93$ & $29.6 \pm 0.41$ & $(5.01 \pm 0.014){\times}10^{-4}$ \\
Uniform Scalar Blend & $-3.69 \pm 0.13$ & $5.99 \pm 0.28$ & $3.69 \pm 0.13$ & $(6.05 \pm 2.22){\times}10^{-4}$ \\
Scalar Blend & $-5.92 \pm 2.8$ & $11.0 \pm 2.0$ & $5.92 \pm 2.8$ & $(5.00 \pm 0.002){\times}10^{-4}$ \\
Uniform Matrix Blend & $-0.0948 \pm 0.0052$ & $0.499 \pm 0.016$ & $0.0948 \pm 0.0052$ & $0.847 \pm 0.0059$ \\
Matrix Blend & $-5.27 \pm 0.28$ & $12.7 \pm 1.3$ & $5.27 \pm 0.28$ & $(5.00 \pm 0.007){\times}10^{-4}$ \\
LFGI--GN (ours) & $\mathbf{-0.0916 \pm 0.0073}$ & $\mathbf{0.367 \pm 0.007}$ & $\mathbf{0.0916 \pm 0.0073}$ & $\mathbf{0.934 \pm 0.0016}$ \\
MAP--Laplace & $-672 \pm 14$ & $865 \pm 11$ & $672 \pm 14$ & $(6.04 \pm 2.22){\times}10^{-4}$ \\
\bottomrule
\end{tabular}
\caption{Same known-normalization diagnostics and reporting convention as \cref{tab:known-z-gaussian-calibration}, now for the shared-geometry four-component analytic inverse problem.  The exact posterior is a Gaussian mixture with closed-form evidence and shared component curvature; MAP--Laplace is only a single local Gaussian approximation.}
\label{tab:known-z-mixture-calibration}
\end{table}

\begin{table}[H]
\centering
\footnotesize
\setlength{\tabcolsep}{3.2pt}
\renewcommand{\arraystretch}{1.12}
\begin{tabular}{@{}lcccc@{}}
\toprule
Method &
\begin{tabular}[c]{@{}c@{}}$\log\qpf$\\bias $\to 0$\end{tabular} &
\begin{tabular}[c]{@{}c@{}}$\log\qpf$\\RMSE $\downarrow$\end{tabular} &
\begin{tabular}[c]{@{}c@{}}$|\widehat{\log Z}-\log Z|$\\$\downarrow$\end{tabular} &
\begin{tabular}[c]{@{}c@{}}Correction\\ESS/$n$ $\uparrow$\end{tabular} \\
\midrule
Tweedie & $-29.9 \pm 0.62$ & $37.7 \pm 0.82$ & $29.9 \pm 0.62$ & $(5.00 \pm 0){\times}10^{-4}$ \\
Uniform Scalar Blend & $-3.77 \pm 0.15$ & $6.10 \pm 0.24$ & $3.77 \pm 0.15$ & $(5.35 \pm 0.79){\times}10^{-4}$ \\
Scalar Blend & $-5.87 \pm 2.7$ & $10.5 \pm 2.1$ & $5.87 \pm 2.7$ & $(6.07 \pm 1.89){\times}10^{-4}$ \\
Uniform Matrix Blend & $-2.37 \pm 0.06$ & $4.23 \pm 0.11$ & $2.37 \pm 0.06$ & $(7.16 \pm 2.45){\times}10^{-4}$ \\
Matrix Blend & $-5.66 \pm 0.42$ & $13.9 \pm 1.8$ & $5.66 \pm 0.42$ & $(5.00 \pm 0.003){\times}10^{-4}$ \\
LFGI--GN (ours) & $\mathbf{-0.102 \pm 0.007}$ & $\mathbf{0.439 \pm 0.035}$ & $\mathbf{0.102 \pm 0.007}$ & $\mathbf{0.237 \pm 0.302}$ \\
MAP--Laplace & $-415 \pm 110$ & $527 \pm 134$ & $415 \pm 110$ & $(5.00 \pm 0){\times}10^{-4}$ \\
\bottomrule
\end{tabular}
\caption{Same known-normalization diagnostics and reporting convention as \cref{tab:known-z-gaussian-calibration}, now for the misaligned-curvature analytic inverse problem.  The exact posterior is a Gaussian mixture with closed-form evidence, but component likelihoods have different right-singular frames, making the stiff directions mode-dependent.}
\label{tab:known-z-misaligned-mixture-calibration}
\end{table}

In \cref{tab:known-z-gaussian-calibration}, MAP--Laplace is exact up to roundoff because the posterior Hessian is constant, while LFGI--GN has absolute log-evidence error $0.0798\pm0.0094$ and correction ESS/$n$ $0.934\pm0.0037$ across five seeded runs.  Uniform Matrix Blend is also accurate on this Gaussian target with spatially uniform curvature, with absolute log-evidence error $0.0840\pm0.0091$ and correction ESS/$n$ $0.867\pm0.0047$.  In \cref{tab:known-z-mixture-calibration}, the shared-geometry mixture remains non-Gaussian and multimodal, but the dominant curvature is still shared across the mixture; accordingly, Uniform Matrix Blend stays close to LFGI--GN, while scalar, Tweedie, Matrix Blend, and MAP--Laplace rows remain much less calibrated.  In \cref{tab:known-z-misaligned-mixture-calibration}, the component curvature frames differ across modes.  There Uniform Matrix Blend improves over the scalar and Tweedie rows but its absolute log-evidence error remains $2.37\pm0.06$, while LFGI--GN reduces the error to $0.102\pm0.007$.  The three known-normalization controls therefore separate implementation correctness, shared non-Gaussian curvature, and mode-dependent anisotropic curvature.


\section{Discussion}
\label{sec:discussion}

\subsection{What this paper establishes}

We combine the Tweedie and target-score identities when the OU-marginal score is estimated from weighted reference samples.  Since their disagreement signal is a conditionally zero-mean control variate, any measurable matrix gate leaves the conditional mean equal to $s_t(y)$ and only changes conditional risk; see \eqref{eq:tweedie-identity}, \eqref{eq:tsi-identity}, \eqref{eq:delta-zero-mean}, and \eqref{eq:measurable-gate-target}.  The risk-minimizing population gate is characterized by a normal equation, targeted directly by the centered primal Matrix Blend in \cref{eq:centered-primal-gate}, and equivalently by the Laplace--Fisher Gate Identity in \cref{thm:lfgi}.  At finite reference count, centered primal regression estimates $-\Cebd\Cdd^{-1}$ from weighted covariance moments, while LFGI estimates $H(y,t)$ and applies $\Psiop(H)$.  \Cref{prop:risk-weighted-lfgi-perturbation,thm:hessian-reference-condition} together with \eqref{eq:lfgi-risk-capture-condition} bound the excess-risk term in \cref{eq:excess-risk} by the Hessian-average error $\epsilon_H$ and the pole factor $\Lambdapole$.

\subsection{When LFGI helps}

LFGI is not a universal Hessian plug-in rule.
It helps under the conditions stated in \cref{prop:risk-weighted-lfgi-perturbation,ass:hessian-concentration,thm:hessian-reference-condition}: the shifted Hessian average $A=\alphat^2I_d+\gammat H$ must stay away from poles on the disagreement-risk directions, and the relative Hessian fluctuation $\epsilon_H$ must be small enough for \cref{eq:lfgi-risk-capture-condition}.  The constant-Hessian Gaussian case motivates matrix gates over scalar gates but is a tie with centered primal regression because of the pointwise cancellation in \cref{prop:gaussian-centered-cancellation}.  The non-Gaussian examples therefore test a more specific claim: scalar interpolation is too crude, $\widehat C_{r_\star\delta}\widehat C_{\delta\delta}^{-1}\Cdd^{1/2}$ can be large for centered primal regression, and the LFGI error $\epsilon_H$ remains controlled.  The GMMs test multimodality and locally rotated singular subspaces; the funnel checks a non-GMM target with state-dependent anisotropy; the Darcy-flow inverse problem tests whether the fitted probability-flow score estimator $\widehat s(x,t)$ supports normalized density evaluation in a PDE-constrained posterior.  Curl is reported only as the auxiliary structural diagnostic defined in Appendix \cref{app:metrics}, Def.~\ref{def:curl}, not as a risk objective optimized by LFGI.

The Neal-funnel pole audit in \cref{app:funnel-pole-diagnostics} separates two conditions that can be conflated.  Raw indefinite curvature does not by itself violate \cref{prop:risk-weighted-lfgi-perturbation,thm:hessian-reference-condition}; the relevant matrix is the shifted conditional average $A(y,t)$.  In the audited funnel run, no nonpositive shifted matrices or active poles were observed, even though the strict finite-reference risk-capture inequality failed on a non-negligible fraction of middle-time queries.  The PSD and pole-separated assumptions give sufficient conditions for the finite-reference bound, but they are not necessary conditions for every useful run.  LFGI can remain stable beyond the PSD setting when OU conditioning and the disagreement covariance avoid the shifted-pole directions.  When the relative Hessian error or the perturbation-to-gain ratio is large, the observed performance should be reported as empirical stability rather than as a consequence of the finite-reference certificate.

\subsection{Limitations}

The method assumes access to the target score and to the observed information or a useful Hessian/Gauss--Newton proxy.  This is natural for explicit-energy Bayesian inverse problems and controlled scientific targets, but not for sample-only generative modeling without additional score or Hessian learning.  It also assumes a finite set of reference samples, denoted earlier by $\bank_N$.  In the density-evaluation application the bank is supplied by a MALA pilot, and the contribution is to add normalized pointwise-density diagnostics, evidence checks, and ESS overlap certificates that the pilot alone lacks; producing the bank from scratch by adaptive transport is outside the present paper.

If the pilot bank misses a posterior mode, the fixed probability-flow score estimator is not evaluated there.  A missing mode may not be detected by local ESS on the region covered by the pilot/reference bank or held-out MALA evaluation samples, and affine errors in $-\log\qpf$ can propagate into evidence estimates unless diagnosed by calibration checks.  The LFGI gate can also become unstable near poles or when the conditional Hessian distribution fluctuates strongly in amplified directions; this is the failure mode described by the perturbation and concentration conditions in \cref{prop:risk-weighted-lfgi-perturbation,ass:hessian-concentration,thm:hessian-reference-condition}.

Finally, full Hessians and full matrix gates are \(d^2\)-scale objects.  The validation suite in \cref{sec:validation-plan} consists of moderate-dimensional controlled validations plus the Darcy-flow inverse problem in \cref{subsec:density-experiments}, and scalable structured gates remain a separate implementation problem.  A sharper concentration theory for self-normalized OU reference banks, beyond the effective-sample-size hypothesis used here, is also left open.

\subsection{Future directions}
Several extensions are natural: adaptive density-ratio-corrected transport using the probability-flow density $\qpf$ induced by the LFGI score estimator for correction weights; structured Fisher--Gauss--Newton gates such as diagonal, block, low-rank, diagonal-plus-low-rank, and matrix-free forms; sharper self-normalized OU concentration theory for tied or partially shared MALA banks; and optional distillation of the validated fixed finite-reference score estimator into a faster score or density surrogate.

\appendix


\section{Validation Targets and Fixed-Budget Setup}
\label{app:validation-targets}

This appendix records the target definitions and reproducibility details used in \cref{sec:validation-plan}.  The validation suite contains one misaligned singular-subspace GMM in $d=8$, the same misaligned family in $d=24$, an explicitly non-Gaussian Neal funnel in $d=10$, one Darcy-flow inverse problem for density evaluation, and three analytic inverse-problem targets for known-normalization calibration.

\subsection{Misaligned singular-subspace Gaussian mixture}
\label{app:target-misaligned-subspace-gmm}

The GMM validation family is
\[
    p_0(x)=\frac1K\sum_{k=1}^K \N(x;m_k,\Sigma_k),
\]
where the component means lie on a fixed-radius configuration and each covariance has a low-dimensional active subspace plus a stiff normal complement:
\[
    \Sigma_k
    =
    U_k \Lambda_\parallel U_k^\top
    +
    \sigma_\perp^2 (I_d-U_kU_k^\top),
    \qquad
    U_k^\top U_k=I_r.
\]
The subspaces $U_k$ are deliberately misaligned across mixture components.  Thus the target is non-Gaussian and locally anisotropic, but each component precision is PSD.  The displayed runs use
\[
    (K,d,r,\sigma_\perp,\mathrm{radius})=(8,8,3,0.035,3)
\]
for the low-dimensional GMM and
\[
    (K,d,r,\sigma_\perp,\mathrm{radius})=(8,24,4,0.035,4.5)
\]
for the scaling check.

This target makes the centered-primal term $\widehat C_{r_\star\delta}\widehat C_{\delta\delta}^{-1}\Cdd^{1/2}$ and the LFGI Hessian error $\epsilon_H$ behave differently.  Direct primal gate estimation must infer singular score-moment geometry from finite weighted samples, while LFGI estimates a PSD Hessian average and applies $\Psiop$; in this target the possible inverse poles are separated in the directions that carry non-negligible disagreement covariance $\Cdd$.

\subsection{\texorpdfstring{Neal funnel in $d=10$}{Neal funnel in d=10}}
\label{app:target-funnel-d10}

The funnel validation target is the standard Neal funnel hierarchy, used here with $d=10$ and $\sigma_f^2=6$:
\[
    X_1\sim \N(0,\sigma_f^2),
    \qquad
    X_{2:d}\mid X_1=x_1\sim \N(0,\exp(x_1)I_{d-1}).
\]
Equivalently, up to an additive normalizing constant,
\[
    \log p_0(x)
    =
    -\frac{x_1^2}{2\sigma_f^2}
    -\frac{d-1}{2}x_1
    -\frac12\exp(-x_1)\|x_{2:d}\|^2.
\]
Reference samples are generated exactly from the hierarchy, rather than by fitting or sampling a Gaussian mixture.  The score and observed information are therefore available in closed form and are also obtained by autodiff in the benchmark script.

Writing $u=x_{2:d}$ and $a_f=\exp(-x_1)$, the clean score is
\[
    s_0(x)
    =
    \left(
        -\frac{x_1}{\sigma_f^2}-\frac{d-1}{2}+\frac12 a_f\|u\|^2,
        -a_f u
    \right),
\]
and the observed information $H_0(x)=-\nabla^2\log p_0(x)$ has block form
\[
    H_0(x)
    =
    \begin{bmatrix}
        \sigma_f^{-2}+\frac12 a_f\|u\|^2 & -a_f u^\top \\
        -a_f u & a_f I_{d-1}
    \end{bmatrix}.
\]
The Schur complement of the lower block is $\sigma_f^{-2}-\frac12 a_f\|u\|^2$, so the raw Hessian geometry is not constrained to be the benign PSD component geometry of the GMM experiments.  The funnel has an indefinite Schur complement, so it tests the same LFGI estimator on a non-Gaussian, non-GMM target with state-dependent anisotropy and a stiff neck.

\subsection{Fixed-budget setup}
\label{app:validation-run-setup}

The displayed fixed-budget comparison tables use the estimator labels fixed in \cref{tab:comparison-estimators}: Tweedie, Uniform Scalar Blend, Scalar Blend, Uniform Matrix Blend, Matrix Blend, and LFGI.  LFGI uses the full raw Hessian precision and does not use Hessian clipping or PSD projection.

The synthetic validation setup uses a single OU corruption and weighting convention.  All synthetic validation runs use the OU process in \cref{subsec:ou-process}, with
\[
    \alphat=e^{-t},
    \qquad
    \gammat=1-e^{-2t},
    \qquad
    X_t\mid X_0=x\sim\N(\alphat x,\gammat I_d).
\]
For a reference bank $\{x_i\}$, the SNIS weights at query $(y,t)$ are
\[
    \widetilde w_i(y,t)
    =
    \frac{\exp\{-\|y-\alphat x_i\|^2/(2\gammat)\}}
    {\sum_j\exp\{-\|y-\alphat x_j\|^2/(2\gammat)\}},
\]
computed in log space.  The score bank constructs the Tweedie and TSI signals, while an independent gate bank constructs the spatially uniform schedules, the spatially varying scalar coefficient, the Matrix Blend gate, or the LFGI gate.  The fixed-budget runs use
\[
    N_{\rm score}=N_{\rm gate}=250,
    \qquad
    N_{\rm gen}=N_{\rm test}=12000,
    \qquad
    N_{\rm large}=512.
\]
Here $N_{\rm large}$ is the held-out SNIS reference bank used only for the score-RMSE proxy when the exact OU-marginal score is not available.

The reverse-sampling discretization is shared across learned estimators.  All learned estimators are plugged into the same Heun predictor--corrector reverse SDE sampler.  The run uses double precision, the linear decreasing grid
\[
    t_k=3.0-k(3.0-0.01)/300,
    \qquad
    k=0,\ldots,300,
\]
initialization $Y_{t_0}\sim\N(0,I_d)$, and final Tweedie denoising at $t_{300}=0.01$. With $h_k=t_k-t_{k+1}$, $z_k\sim\N(0,I_d)$, and $f(y,t)=y+2\widehat s(y,t)$, one step is
\[
\begin{aligned}
    \widetilde Y_{k+1}
    &=Y_k+h_k f(Y_k,t_k)+\sqrt{2h_k}\,z_k,\\
    Y_{k+1}
    &=Y_k+\frac{h_k}{2}
      \{f(Y_k,t_k)+f(\widetilde Y_{k+1},t_{k+1})\}
      +\sqrt{2h_k}\,z_k.
\end{aligned}
\]
The same Gaussian increment is used in the predictor and corrector.  After the last stochastic step, the returned sample is
\[
    \widehat X_0
    =
    \frac{Y_{t_{300}}+\gammat\widehat s(Y_{t_{300}},t_{300})}{\alphat}
    \quad\text{at }t=t_{300}.
\]

The blend-estimator regularization is fixed across runs.  The scalar blend uses no tuned ridge parameter.  Its weighted least-squares coefficient is the scalar gate in \cref{tab:comparison-estimators}, with denominator floor $10^{-20}$ and clipping to $[0,1]$.  The Uniform Scalar Blend and Uniform Matrix Blend estimate the target-score second moment from the gate bank,
\[
    \widehat I_\pi^{\rm g}=\frac{1}{N_{\rm gate}}\sum_i s_0(x_i^{\rm gate})s_0(x_i^{\rm gate})^\top,
\]
and use the estimators in \cref{eq:s-unif-sc,eq:s-unif-mat} with $I_\pi$ replaced by $\widehat I_\pi^{\rm g}$.
The Matrix Blend gate is estimated from the independent gate bank as
\[
    \widehat G
    =
    -\widehat{\operatorname{Cov}}(b,\delta)
    \left(\widehat{\operatorname{Cov}}(\delta,\delta)+\rho I_d\right)^{-1},
    \qquad
    \delta=c-b,
\]
where $b_i=-(y-\alphat x_i)/\gammat$, $c_i=s_0(x_i)/\alphat$, and
\[
    \rho=10^{-8}+10^{-2}\,\frac{\operatorname{tr}(\widehat{\operatorname{Cov}}(\delta,\delta))}{d}.
\]

No symmetrization is applied to the fitted matrix gate; the numerical gate-entry clip is set to $10^6$.  LFGI uses the independent gate-bank Hessian average $\widehat H(y,t)=\sum_i\widetilde w_i^{\rm gate}(y,t)H_0(x_i^{\rm gate})$ and applies
\[
    \widehat G_{\rm LFGI}(y,t)
    =
    \alphat^2(\alphat^2I_d+\gammat\widehat H(y,t))^{-1}
\]
through an eigensolve of the symmetrized averaged Hessian; this is the raw gate, not a PSD-projected gate.

The repetitions and metric evaluation are fixed as follows.  The synthetic fixed-budget aggregate tables average five independent runs with base seed $42$, using seeds $42,43,\ldots,46$.  Reference samples for both GMM targets are exact mixture draws, and reference samples for the funnel are exact hierarchical draws.  Sliced KS uses at most $512$ generated and $512$ reference samples per run, with $1000$ random projections.  MMD uses the multiscale RBF grid in \cref{def:mmd}.  KDE NLL uses a Gaussian KDE fit to at most $5000$ generated samples with bandwidth $\max\{n^{-1/(d+4)},0.05\}$ and is evaluated on the independent test bank.  The ``GT floor'' rows in the main tables are independent reference-sample floors for the sample-quality metrics, not learned score estimators.  The separate pole audit for the funnel is reported in \cref{app:funnel-pole-diagnostics}; it checks the shifted conditional matrices in \cref{eq:Lambda-pole,eq:lfgi-risk-capture-condition}, not the raw Hessian eigenvalues alone.


\section{Additional qualitative panels for the $d=24$ GMM scaling check}
\label{app:d24-qualitative-panels}

The $d=24$ misaligned GMM table provides a compact scaling check.  The qualitative PCA panel gives the corresponding low-dimensional visualization.  In this run, LFGI reduces score RMSE from $10.157\pm1.231$ for Tweedie, $9.823\pm0.961$ for Uniform Scalar Blend, $11.580\pm0.971$ for Scalar Blend, $9.650\pm0.508$ for Uniform Matrix Blend, and $11.755\pm0.574$ for Matrix Blend to $4.470\pm0.199$, while also improving the reported Sliced KS, MMD, and NLL sample-quality diagnostics.  The target remains inside the intended PSD/pole-separated regime, so the figure is interpreted as a scaling check consistent with the residual-leakage analysis of \cref{subsec:bimodal-primal-failure-scale} rather than as a separate application result.

\begin{figure}[H]
\centering
\includegraphics[width=.99\textwidth]{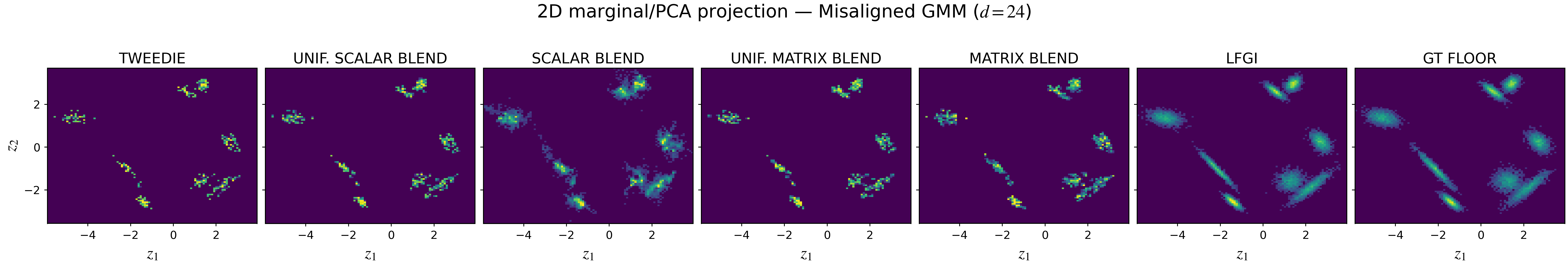}
\caption{Same projection diagnostic as \cref{fig:misaligned-gmm-d8-heatmaps}, now for the $d=24$, rank-$4$ target.  Panels follow the method titles.}
\label{fig:misaligned-gmm-d24-heatmaps}
\end{figure}


\section{Density-Evaluation Inverse-Problem Targets and Pilot Construction, Bank Splits, and Density-Evaluation Setup}
\label{app:density-targets}

This appendix records the PDE, parameterization, observation model, pilot construction, and bank split for the Darcy-flow density-evaluation benchmark.

\subsection{Posterior-pilot OU weights}
\label{app:darcy-likelihood-informed-weights}

The inverse-problem construction uses the posterior density in coefficient space as the target density,
\[
    \tilde p_0(\alpha)=p_{\rm prior}(\alpha)L(\alpha),
    \qquad
    L(\alpha)=p(y_{\rm obs}\mid \alpha).
\]

The posterior target score is
\[
    s^{\rm post}_0(\alpha)=\nabla_\alpha\log p_{\rm prior}(\alpha)+\nabla_\alpha\log L(\alpha).
\]
In the reported Darcy-flow density run, the reference samples are already posterior pilot samples.  For a noisy OU query point $z$ at time $t$, the finite-reference OU weights are therefore the posterior-bank SNIS weights
\[
    w_i(z,t)
    =
    \frac{K_t^{\rm OU}(z\mid \alpha_i)}
    {\sum_{j=1}^{N_{\rm ref}}K_t^{\rm OU}(z\mid \alpha_j)} .
\]

The likelihood enters through the posterior pilot distribution and through the stored posterior score and energy, not through an additional likelihood tilt of the OU weights.  These weights define posterior Tweedie, TSI, scalar-blend, and LFGI score estimates by applying the same finite-reference identities to the posterior target.  The density-evaluation benchmark holds the resulting score estimator fixed and evaluates the normalized probability-flow density on held-out posterior samples.

\subsection{Darcy-flow target}
\label{app:darcy-target-details}

The unknown is a log-permeability field on $D=[0,1]^2$.  The pressure satisfies
\[
    -\nabla\!\cdot\{k(x;\alpha)\nabla p(x;\alpha)\}=1,
    \qquad x\in D,
    \qquad p|_{\partial D}=0.
\]
The permeability is parameterized as $k(x;\alpha)=\exp(m(x;\alpha))$, where
\[
    m(x;\alpha)=\sum_{j=1}^{32}\alpha_j\phi_j(x),
    \qquad
    \alpha\sim\N(0,I_{32}).
\]
The basis functions are the leading modes of an exponential covariance kernel with correlation length $\ell=0.1$ and prior scale $1.0$, discretized on a $32\times32$ grid and truncated to $32$ latent coefficients.  The forward solve uses a five-point finite-difference discretization on the interior grid, with harmonic face averages for the permeability.

\subsection{Observation model and held-out geometry}
\label{app:darcy-observation-model}

The likelihood observes the pressure at $120$ randomly selected interior grid locations.  A disjoint set of $30$ interior locations is retained for held-out predictive diagnostics. With observation operator $S$ and synthetic truth $\alpha_\star$, the observed training data are
\[
    y_{\rm obs}=S p(\cdot;\alpha_\star)+\eta,
    \qquad
    \eta\sim\N(0,\sigma^2 I),
    \qquad
    \sigma=10^{-3}.
\]

The posterior energy used in the Darcy-flow density benchmark of \cref{subsec:density-experiments} is
\[
    E(\alpha)
    =
    \frac12\|\alpha\|^2
    +
    \frac{1}{2\sigma^2}\|S p(\cdot;\alpha)-y_{\rm obs}\|^2
    \quad
    \text{up to an additive constant.}
\]
This target is nonlinear because the pressure solve depends on $\exp(m)$, while the small observation noise and sparse interior measurements produce a posterior with strong anisotropic curvature in coefficient space.

\subsection{Posterior differential information and Gauss--Newton precision}
\label{app:darcy-gn-precision-details}

The target score at $t=0$ is available in coefficient coordinates,
\[
    \nabla_\alpha\log\tilde p(\alpha)
    =
    -\alpha+\nabla_\alpha\log L(\alpha),
\]
where the likelihood gradient is computed by differentiating the JAX forward solve.  For LFGI in this inverse problem, we use the prior-plus-Gauss--Newton observed-information proxy
\[
    P^{\GN}(\alpha)
    =
    I_d+\frac{1}{\sigma^2}J(\alpha)^\top J(\alpha),
    \qquad
    J(\alpha)=\nabla_\alpha\{S p(\cdot;\alpha)\}.
\]
Because $P^{\GN}(\alpha)\succeq I_d$, the shifted matrix $\alphat^2I_d+\gammat P^{\GN}(\alpha)$ is separated from zero before averaging, so the LFGI resolvent has no negative-curvature poles in this proxy.  It drops residual second-derivative terms from the exact likelihood Hessian, while the posterior score and energy still use the nonlinear likelihood itself.  Accordingly, the density table labels the method ``LFGI--GN'' to indicate that the estimator is LFGI with a Gauss--Newton precision proxy used to construct the gate.

\subsection{Pilot-bank and density-evaluation setup}
\label{app:darcy-pilot-setup}

\Cref{tab:darcy-density-results} uses adjoint-gradient MALA posterior banks.  Under independent score/gate coupling, the source bank contains $2000=1000+1000$ retained posterior samples: $1000$ are assigned to the score-signal bank and $1000$ to the gate bank.  A separate MALA-EVAL run supplies $1000$ held-out held-out density-evaluation samples.  Each retained posterior sample is produced by a MALA chain with $600$ transition steps, $200$ burn-in steps, step size $5\times10^{-5}$, and no additional thinning.  The density benchmark constructs three roles from posterior pilot samples:
\begin{enumerate}[leftmargin=*]
    \item a score-signal bank for finite-reference Tweedie/TSI score signals;
    \item a gate bank for LFGI--GN, Uniform Scalar/Matrix Blend schedules, matrix-blend, or Scalar Blend gate estimation;
    \item a density-evaluation bank, denoted MALA-EVAL, on which $\log\qpf$ and the true unnormalized posterior energy are both evaluated.
\end{enumerate}
Each seeded run uses $N_{\rm signal}=N_{\rm gate}=N_{\rm eval}=1000$ and integrates the probability-flow density equation over $t\in[10^{-2.5},5]$ using $64$ time steps.  Reported density comparisons use independent score, gate, and evaluation banks, with central affine diagnostics computed on the $3$--$97\%$ energy band.  The PF sensitivity diagnostic varies only $N_{\rm PF}$ and $t_{\min}$ for LFGI--GN and Tweedie; it is a numerical sanity check rather than a broad hyperparameter sweep.

\subsection{Known-normalization calibration targets}
\label{app:known-z-calibration-target}

The known-normalization controls in \cref{subsec:known-z-calibration} use the same probability-flow density-evaluation harness as the Darcy-flow density-evaluation benchmark, but replace the PDE likelihood by analytic inverse problems with closed-form posterior normalization.  All three runs use a standard Gaussian prior in $d=8$, Gaussian-noise scale $\sigma=0.35$, and independent score/gate/evaluation banks with
\[
    N_{\rm signal}=N_{\rm gate}=N_{\rm eval}=2000,
\]
probability-flow integration over $t\in[10^{-2.5},5]$ with $64$ time steps, and the density diagnostics in \cref{app:density-metrics}.  The linear-Gaussian and shared-geometry mixture controls use an $8\times 8$ deterministic linear forward map $A$ with singular values geometrically spaced from $4$ to $4/60$.  The misaligned-curvature mixture uses component-dependent forward maps with the same singular spectrum but different right-singular frames.  In all analytic targets, $L(x)$ omits Gaussian observation normalizing constants; the analytic $Z$ and the tabulated log-evidence errors use the same convention.

The first analytic calibration target is the linear-Gaussian positive control:
\[
    x\sim \N(0,I_d),
    \qquad
    L(x)=\exp\left\{-\frac{1}{2\sigma^2}\|Ax-y\|^2\right\}.
\]
With
\[
    \Lambda=I_d+\sigma^{-2}A^\top A,
    \qquad
    h=\sigma^{-2}A^\top y,
    \qquad
    m=\Lambda^{-1}h,
\]
the posterior is $\N(m,\Lambda^{-1})$ and the benchmark evidence is
\[
    \log Z
    =
    -\frac{1}{2\sigma^2}\|y\|^2
    +\frac12 h^\top \Lambda^{-1}h
    -\frac12\log\det\Lambda .
\]
The source and evaluation banks are exact posterior draws.  MAP--Laplace is exact in the linear-Gaussian control because the posterior is Gaussian and the negative-log posterior Hessian is constant.

The second analytic calibration target is a four-component mixture.  It keeps the same prior and forward operator, but replaces the likelihood by a four-component ambiguous inverse problem,
\[
    L(x)
    =
    \sum_{k=1}^4 \pi_k
    \exp\left\{-\frac{1}{2\sigma^2}\|Ax-y_k\|^2\right\},
    \qquad
    \pi_k=\frac14 .
\]
The four observations are generated from separated latent centers in the first two right-singular directions of $A$: the signs are $(+,+),(+,-),(-,+),(-,-)$ and the component-separation parameter is $2.0$.  For each component, define
\[
    \Lambda_k=I_d+\sigma^{-2}A^\top A,
    \qquad
    h_k=\sigma^{-2}A^\top y_k,
    \qquad
    m_k=\Lambda_k^{-1}h_k,
\]
and
\[
    \log Z_k
    =
    -\frac{1}{2\sigma^2}\|y_k\|^2
    +\frac12 h_k^\top \Lambda_k^{-1}h_k
    -\frac12\log\det\Lambda_k .
\]
Then
\[
    Z=\sum_{k=1}^4\pi_k Z_k,
    \qquad
    p(x\mid y) = \sum_{k=1}^4 \tau_k \N(x;m_k,\Lambda_k^{-1}),
    \qquad
    \tau_k=\frac{\pi_k Z_k}{\sum_{\ell=1}^4\pi_\ell Z_\ell}.
\]
The source and evaluation banks are exact draws from the posterior mixture.  The MAP--Laplace row in \cref{tab:known-z-mixture-calibration} is a single local Gaussian approximation, so it is not exact for the mixture target.

The third analytic calibration target introduces misaligned curvature.  It keeps the standard Gaussian prior and finite Gaussian-mixture evidence, but gives each component its own forward map,
\[
    L(x)
    =
    \sum_{k=1}^4 \pi_k
    \exp\left\{-\frac{1}{2\sigma^2}\|A_kx-y_k\|^2\right\},
    \qquad
    \pi_k=\frac14 .
\]
Write a singular-value decomposition $A_k=U_k S V_k^\top$.  The matrices share the same singular values in $S$, but the right-singular frames $V_k$ differ across components.  For each component,
\[
    \Lambda_k=I_d+\sigma^{-2}A_k^\top A_k,
    \qquad
    h_k=\sigma^{-2}A_k^\top y_k,
    \qquad
    m_k=\Lambda_k^{-1}h_k,
\]
and
\[
    \log Z_k
    =
    -\frac{1}{2\sigma^2}\|y_k\|^2
    +\frac12 h_k^\top \Lambda_k^{-1}h_k
    -\frac12\log\det\Lambda_k .
\]
The evidence and posterior mixture are again
\[
    Z=\sum_{k=1}^4\pi_k Z_k,
    \qquad
    p(x\mid y)=\sum_{k=1}^4\tau_k\N(x;m_k,\Lambda_k^{-1}),
    \qquad
    \tau_k=\frac{\pi_k Z_k}{\sum_\ell \pi_\ell Z_\ell}.
\]
Unlike the shared-geometry target, the component precision matrices $\Lambda_k$ have different stiff eigenspaces.  This makes the population gate mode-dependent while preserving exact posterior sampling and exact evidence for \cref{tab:known-z-misaligned-mixture-calibration}.

\subsection{Pilot bank split}
\label{app:pilot-split-setup}

The density/evidence diagnostics use held-out density-evaluation points.  Tied score/gate banks are useful for cost studies.  Reported $\log\qpf$ diagnostics use an independent MALA-EVAL bank or a leave-one-out analogue as a guard against in-sample calibration artifacts.


\section{Auxiliary Darcy-Flow Density-Energy Grid}
\label{app:darcy-density-grid}

\Cref{fig:darcy-density-energy-grid} gives the held-out density-energy and affine-residual panels corresponding to \cref{tab:darcy-density-results}.  The table in the main text is the primary comparison; this figure records the row-by-row calibration geometry on the same central $3$--$97\%$ posterior-energy band used for the affine diagnostics.

\clearpage
\begin{figure}[p]
\centering
\includegraphics[height=0.92\textheight]{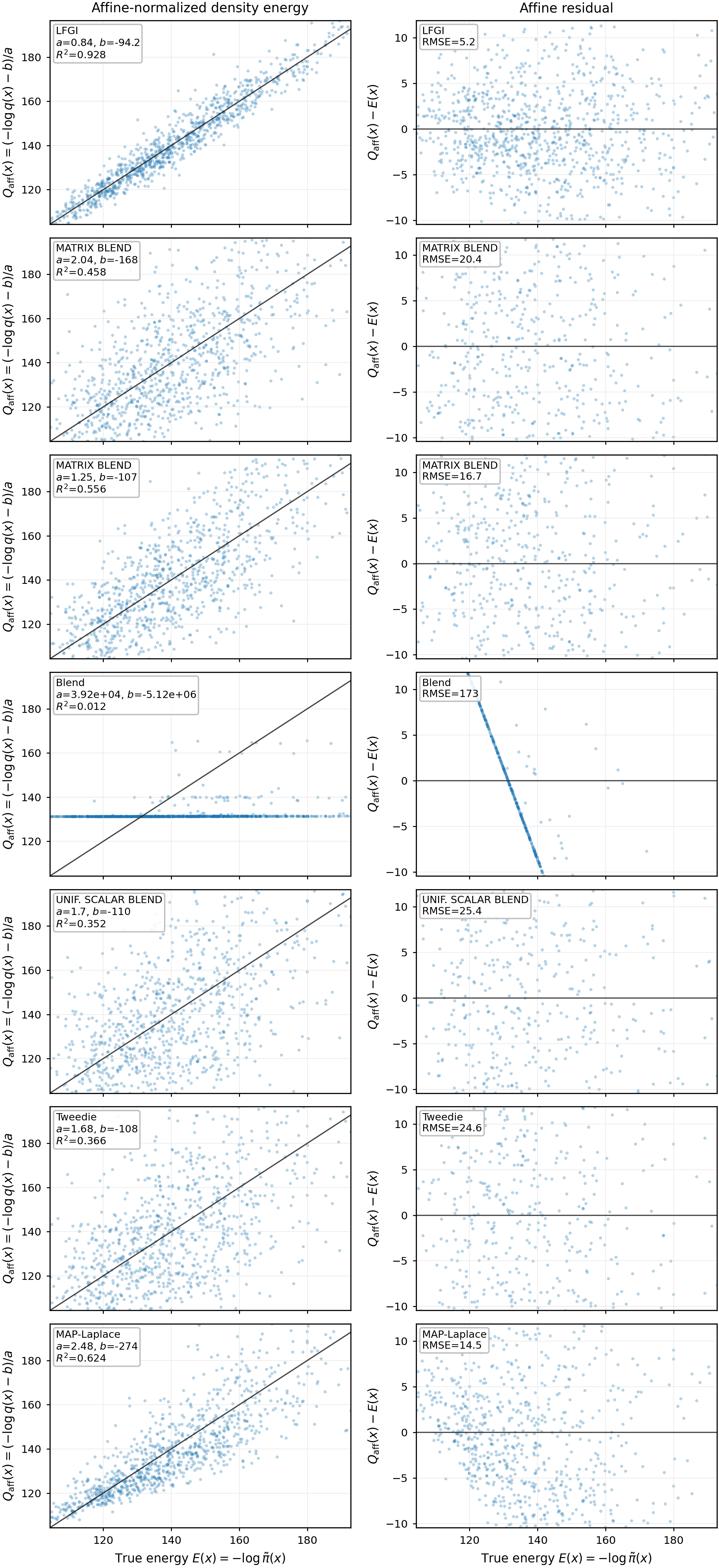}
\caption{Darcy-flow density-evaluation diagnostic on held-out MALA-EVAL samples.  Left: true posterior energy against affine-normalized probability-flow energy on the central $3$--$97\%$ band.  Right: affine residuals.  LFGI--GN is tightest; MAP--Laplace captures part of the trend; the uniform and local blend rows remain more diffuse.}
\label{fig:darcy-density-energy-grid}
\end{figure}
\clearpage

\section{Metric Definitions and Evaluation Conventions}
\label{app:metrics}

This appendix gives the metric definitions used in \cref{sec:validation-plan}.  The notation below matches the implementation.  Let $X^{\rm gen}=\{x_i^{\rm gen}\}_{i=1}^n$ denote generated samples from a method and let $X^{\rm ref}=\{x_j^{\rm ref}\}_{j=1}^m$ denote held-out reference samples from the target.  When a finite sample cap is used, both sets are first subsampled to the cap used by the corresponding script.  Reported means and standard deviations in the tables are computed over independent repetitions of the whole benchmark run, not only over bootstrap resamples of a fixed run.  Runtime is wall-clock time for estimator construction and reverse sampling under the same benchmark script.

\subsection{Sample-discrepancy metrics}
\label{app:metric-sample-discrepancies}

\begin{definition}[Maximum mean discrepancy]
\label[definition]{def:mmd}
For a positive definite kernel $k$, the biased squared maximum mean discrepancy (MMD) estimator is the empirical diagnostic \citep{gretton2012kernel}
\[
    \widehat{\operatorname{MMD}}_k^2(X^{\rm gen},X^{\rm ref})
    =
    \frac1{n^2}\sum_{i,i'} k(x_i^{\rm gen},x_{i'}^{\rm gen})
    +
    \frac1{m^2}\sum_{j,j'} k(x_j^{\rm ref},x_{j'}^{\rm ref})
    -
    \frac{2}{nm}\sum_{i,j} k(x_i^{\rm gen},x_j^{\rm ref}).
\]
We use RBF kernels $k_\sigma(x,y)=\exp\{-\|x-y\|^2/(2\sigma^2)\}$ and report the average of this quantity over a multiscale bandwidth grid.  In the fixed-budget dashboard script the grid is $\sigma\in\{0.5,1,2,5,10\}$ after the samples are represented in the benchmark coordinates.  In the validation-sweep script the grid is $\{0.5\widehat\sigma,\widehat\sigma,2\widehat\sigma,4\widehat\sigma\}$, where $\widehat\sigma$ is the median cross-distance between the two evaluated sample sets.
\end{definition}

\begin{definition}[Kernel Stein discrepancy]
\label[definition]{def:ksd}
For a target score $s_0(x)=\nabla\log p_0(x)$ and an RBF kernel $k_h(x,y)=\exp\{-\|x-y\|^2/(2h^2)\}$, the Stein kernel used in the validation metrics is the standard kernel Stein discrepancy (KSD) kernel \citep{liu2016ksd,chwialkowski2016kernelgof}:
\[
\begin{aligned}
    u_{s_0}(x,y)
    &=
    k_h(x,y)\,s_0(x)^\top s_0(y)
    +s_0(x)^\top \nabla_y k_h(x,y)
    +s_0(y)^\top \nabla_x k_h(x,y) \\
    &\quad
    +\operatorname{tr}\{\nabla_x\nabla_y k_h(x,y)\}.
\end{aligned}
\]
Equivalently, for the RBF kernel this is evaluated as
\[
    u_{s_0}(x,y)
    =
    k_h(x,y)
    \left[
        s_0(x)^\top s_0(y)
        -\frac{(s_0(x)-s_0(y))^\top(x-y)}{h^2}
        +\frac{d}{h^2}
        -\frac{\|x-y\|^2}{h^4}
    \right].
\]
The reported KSD is
\[
    \widehat{\operatorname{KSD}}(X^{\rm gen},p_0)
    =
    \left(\frac1{n^2}\sum_{i,j}u_{s_0}(x_i,x_j)\right)^{1/2},
\]
with $h$ set by the median pairwise distance unless a fixed bandwidth is specified.  For ground-truth floor rows, the reported KSD is the same finite-sample diagnostic applied to independent reference samples; it should be read as a reference-sample floor rather than as a learned-estimator score diagnostic.
\end{definition}

\begin{definition}[Sliced Kolmogorov--Smirnov distance]
\label[definition]{def:sliced-ks}
For a direction $v\in\mathbb S^{d-1}$, let the one-dimensional Kolmogorov--Smirnov statistic \citep{kolmogorov1933sulla,smirnov1948table} be
\[
    \operatorname{KS}_v(X^{\rm gen},X^{\rm ref})
    =
    \sup_{a\in\mathbb R}
    \left|
        \widehat F_{\{v^\top x_i^{\rm gen}\}}(a)
        -
        \widehat F_{\{v^\top x_j^{\rm ref}\}}(a)
    \right|.
\]
The sliced KS metric is the average of $\operatorname{KS}_v$ over randomly sampled unit directions.  The fixed-budget tables use this as the robust sample-quality discrepancy for the GMM and funnel targets because it is stable in moderate dimension and sensitive to one-dimensional marginal mismatch.
\end{definition}

\begin{definition}[Entropic Wasserstein-2]
\label[definition]{def:entropic-w2}
For overlap with the diffusion-path sequential Monte Carlo (DPSMC) benchmark suite, the script also implements an entropy-regularized Wasserstein diagnostic in the Sinkhorn form of \citet{cuturi2013sinkhorn}. With uniform empirical weights and squared Euclidean cost $C_{ij}=\|x_i^{\rm gen}-x_j^{\rm ref}\|^2$, define
\[
    \widehat W_{2,\varepsilon}^2(X^{\rm gen},X^{\rm ref})
    =
    \min_{\Pi\in\Gamma(a,b)}
    \sum_{i,j}\Pi_{ij} C_{ij}
    +
    \varepsilon\sum_{i,j}\Pi_{ij}(\log \Pi_{ij}-1),
\]
where $a$ and $b$ are uniform weights on the two finite sample sets.  The reported value is $\widehat W_{2,\varepsilon}=\sqrt{\widehat W_{2,\varepsilon}^2}$ with $\varepsilon=0.05$ in the synthetic validation targets.
\end{definition}

\begin{definition}[KDE negative log likelihood (NLL)]
\label[definition]{def:nll}\label[definition]{def:kde-nll}
For fixed-budget sample-quality evaluation, KDE NLL is computed by fitting a Gaussian-kernel density estimator $\widehat p_{\rm KDE}$ \citep{silverman1986density} to generated samples and evaluating
\[
    \operatorname{NLL}_{\rm KDE}(X^{\rm gen};X^{\rm ref})
    :=
    -\frac1{|X^{\rm ref}|}\sum_{x\in X^{\rm ref}}\log \widehat p_{\rm KDE}(x).
\]
The KDE bandwidth is the rule $\max\{n^{-1/(d+4)},0.05\}$ used in the dashboard script.  The diagnostic is reported as NLL in the fixed-budget sampling tables; lower values indicate that generated samples assign higher KDE density to independent reference samples.
\end{definition}

\begin{definition}[KDE effective sample size]
\label[definition]{def:kde-ess}
The KDE ESS diagnostic fits the same type of KDE as in \cref{def:kde-nll} to generated samples, forms normalized weights
\[
    \omega_i
    \propto
    \exp\{\log p_0(x_i)-\log \widehat p_{\rm KDE}(x_i)\},
    \qquad
    \sum_i\omega_i=1,
\]
and reports the normalized effective sample size diagnostic \citep{kong1994sequential,martino2017effective}
\[
    \operatorname{ESS}_{\rm KDE}(X^{\rm gen})
    =
    \frac{1}{n\sum_{i=1}^n\omega_i^2}.
\]
This is a diagnostic for sample collapse and target-density agreement, not an importance-sampling correctness certificate.
\end{definition}

\subsection{Score-field and curl diagnostics}
\label{app:metric-score-curl}

\begin{definition}[Score RMSE]
\label[definition]{def:score-rmse}
For a finite time grid $\mathcal T$, the ideal noisy-score RMSE is
\[
    \operatorname{RMSE}_{\rm score}(\widehat s)
    =
    \left(
    \frac1{|\mathcal T|}
    \sum_{t\in\mathcal T}
    \E_{Y_t\sim p_t}
    \bigl[\|\widehat s(Y_t,t)-s_t(Y_t)\|^2\bigr]
    \right)^{1/2}.
\]
In analytic GMM validation sweeps, $s_t$ is the exact OU-marginal score.  When an exact OU-marginal score is not available in closed form or is not evaluated directly, $s_t$ is replaced by a high-reference SNIS--Tweedie proxy computed from a large held-out reference bank.  This proxy is intended as a consistent large-reference approximation to the OU-marginal score; it is not treated as an additional estimator baseline, and conclusions for the non-analytic benchmarks are also supported by sample-quality metrics that do not use this proxy.  In all cases the Monte Carlo estimate uses noisy states generated by drawing $X_0$ from held-out reference samples and then applying the OU corruption at the displayed time grid.
\end{definition}

\begin{definition}[Curl diagnostic]
\label[definition]{def:curl}
For a differentiable estimated score field $\widehat s(\cdot,t)$, define the skew Jacobian
\[
    \operatorname{skew}\nabla\widehat s(y,t)
    =
    \frac12\left(\nabla_y\widehat s(y,t)-\nabla_y\widehat s(y,t)^\top\right).
\]
The absolute curl diagnostic is the trajectory-averaged RMS quantity
\[
    \operatorname{curl}_{\rm abs}
    =
    \left(
    \frac1N\sum_{(y,t)\in\mathcal P}
    \|\operatorname{skew}\nabla\widehat s(y,t)\|_F^2
    \right)^{1/2},
\]
where $\mathcal P$ is a subsampled set of points along estimator-generated reverse trajectories.  The relative curl reported in the tables is
\[
    \operatorname{curl}_{\rm rel}
    =
    \frac{\operatorname{curl}_{\rm abs}}
    {\left(N^{-1}\sum_{(y,t)\in\mathcal P}\|\nabla_y\widehat s(y,t)\|_F^2\right)^{1/2}}.
\]
The implementation estimates these Frobenius norms without materializing Jacobians, using centered finite differences and independent Gaussian bilinear probes.  Thus curl is an implementation-level structural diagnostic for approximate integrability of the learned score field, not a theorem-level risk metric or an objective optimized by LFGI.
\end{definition}

\subsection{Gate-capture diagnostics}
\label{app:metric-gate-capture}

\begin{definition}[Gate-capture error]
\label[definition]{def:gate-capture}
The gate sweep compares an empirical gate $\widehat G$ with the population gate $\Gstar$; in plots where $\Gstar$ is not analytic, it is approximated using a large held-out reference bank.  Let
\[
    \Cdd(y,t)=\E\bigl[(c-b)(c-b)^\top\mid Y_t=y\bigr]
\]
be the conditional disagreement covariance.  The risk-weighted gate error is
\[
    \operatorname{Err}_{\rm risk}(\widehat G,\Gstar;y,t)
    =
    \operatorname{tr}\left
    \{(\widehat G-\Gstar)\Cdd(y,t)(\widehat G-\Gstar)^\top\right\}.
\]
This is the excess-risk geometry induced by \cref{prop:residual-risk}, specifically \cref{eq:excess-risk}.  The displayed risk-weighted gate-capture curves in Appendix \cref{app:gate-capture-breakdown} normalize the mean of this quantity by $\operatorname{tr}\Cdd$.  The relative Frobenius error is the auxiliary matrix-scale diagnostic
\[
    \operatorname{Err}_{F}(\widehat G,\Gstar)
    =
    \frac{\|\widehat G-\Gstar\|_F}{\|\Gstar\|_F}.
\]
The large-reference approximations to $\Gstar$ are computed using a held-out reference bank, while the plotted finite-reference gates use the stated gate-bank size.
\end{definition}

\section{Auxiliary Computations and Worked Examples}
\label{app:auxiliary-computations-worked-examples}

The auxiliary computations below give the model-specific formulas and finite-reference comparisons used by the main text.  The ordering follows the order in which the corresponding objects appear in the main text.

\subsection{Why spatially varying matrix gates are needed: misaligned bimodal stiffness}
\label{subsec:misaligned-bimodal-gmm}

The residual mechanism from \cref{subsec:non-gaussian-residual} becomes relevant in targets where the risk-minimizing matrix gate itself varies with position.  Scalar Tweedie--TSI coefficients and spatially uniform matrix schedules can represent fixed attenuation patterns, but they cannot match a gate whose eigenspaces rotate with $y$.  A simple example is a two-component Gaussian mixture whose component precisions have the same singular spectrum but different eigenspaces.

\begin{definition}[Bimodal misaligned-stiffness GMM]
\label[definition]{def:bimodal-misaligned-gmm}
Let
\begin{equation*}
    p_0(x)
    =
    \pi_+\N(x;m_+,P_+^{-1})
    +
    \pi_-\N(x;m_-,P_-^{-1}),
    \qquad
    \pi_++\pi_-=1,
\end{equation*}
with
\begin{equation*}
    P_\pm=U_\pm\Lambda U_\pm^\top,
    \qquad
    \Lambda=\diag(\lambda_1,\ldots,\lambda_d),
\end{equation*}
where \(\Lambda\) is ill-conditioned and \(P_+\) and \(P_-\) are not simultaneously diagonalizable.  Equivalently, the stiff directions of the two components are misaligned.
\end{definition}

For this target, the OU posterior has a closed-form mixture representation.  Define
\begin{equation*}
    \Sigma^Y_k(t):=\gammat I_d+\alphat^2P_k^{-1},
    \qquad k\in\{+,-\},
\end{equation*}
component responsibilities
\begin{equation*}
    \tau_k(y,t)
    :=
    \frac{\pi_k\N(y;\alphat m_k,\Sigma^Y_k(t))}
         {\sum_{\ell\in\{+,-\}}\pi_\ell\N(y;\alphat m_\ell,\Sigma^Y_\ell(t))},
\end{equation*}
and posterior component covariance and mean
\begin{equation*}
\Sigma^{X\mid Y}_k(t)
    :=
    \left(P_k+\frac{\alphat^2}{\gammat}I_d\right)^{-1}
    =
    \gammat(\alphat^2I_d+\gammat P_k)^{-1},
    \qquad
    \mu_k(y,t)
    :=
    \Sigma^{X\mid Y}_k(t)\left(P_km_k+\frac{\alphat}{\gammat}y\right).
\end{equation*}
Then
\begin{equation*}
    p_0(x\mid Y_t=y)
    =
    \sum_{k\in\{+,-\}}\tau_k(y,t)\N(x;\mu_k(y,t),\Sigma^{X\mid Y}_k(t)).
\end{equation*}
The time-\(t\) marginal score is also explicit.  If
\begin{equation*}
    s^{Y}_{t,k}(y):=-\Sigma^Y_k(t)^{-1}(y-\alphat m_k),
\end{equation*}
then
\begin{equation*}
s_t(y)=\sum_{k\in\{+,-\}}\tau_k(y,t)s^{Y}_{t,k}(y).
\end{equation*}

To isolate the obstruction for a spatially uniform matrix gate, augment the mixture with its latent component label \(K\).  The label-resolved score contribution \(s_K(x)/\alphat\), where \(s_k(x)=-P_k(x-m_k)\), has conditional mean \(s_t(y)\).  The implemented estimator instead uses the marginal contribution \(s_0(x)/\alphat=\E[s_K(x)/\alphat\mid X=x]\).  The two gates are close in component-local regions: the marginal gate replaces \(P_\tau(y,t)\) by \(P_\tau(y,t)-C_{\rm ov}(y,t)\), where
\begin{equation}
\label{eq:gmm-marginal-overlap-correction}
    C_{\rm ov}(y,t)
    :=
    \E\!\left[
        \omega_+(X_0)\omega_-(X_0)
        \{s_+(X_0)-s_-(X_0)\}\{s_+(X_0)-s_-(X_0)\}^\top
        \mid Y_t=y
    \right]
\end{equation}
and \(\omega_k(x)=\Pr(K=k\mid X_0=x)\).  Thus the correction is controlled by responsibility overlap; \cref{app:proof-bimodal-marginal-overlap} gives the verification.

\begin{proposition}[Closed-form label-resolved spatially varying gate for the misaligned bimodal GMM]
\label[proposition]{prop:bimodal-local-gate}
Let
\begin{equation*}
    P_\tau(y,t):=\sum_{k\in\{+,-\}}\tau_k(y,t)P_k,
    \qquad
    A_k(t):=\alphat^2I_d+\gammat P_k,
    \qquad
    B_k(t):=-\frac{1}{\alphat\gammat}A_k(t).
\end{equation*}
Conditioned on \(K=k\) and \(Y_t=y\), write \(X_0=\mu_k(y,t)+u_k\), with \(u_k\sim\N(0,\Sigma^{X\mid Y}_k(t))\).  For the label-resolved score contribution,
\begin{equation*}
b(X_0;y,t)-s_t(y)
    =
    a_k(y,t)+\frac{\alphat}{\gammat}u_k,
    \qquad
    \delta(X_0;y,t)=B_k(t)u_k,
\end{equation*}
where
\begin{equation*}
    a_k(y,t):=s^{Y}_{t,k}(y)-s_t(y).
\end{equation*}
Consequently,
\begin{equation*}
    \Cdd(y,t)
    =
    \frac{1}{\gammat}I_d+\frac{1}{\alphat^2}P_\tau(y,t),
    \qquad
    \Cebd(y,t)=-\frac{1}{\gammat}I_d,
\end{equation*}
and the population gate for the label-resolved OU-marginal score identity using the latent-component target-score signal is

\begin{equation}
\label{eq:bimodal-local-gate}
    G_{\rm lab}(y,t)
    =
    \alphat^2
    \left(\alphat^2I_d+\gammat P_\tau(y,t)\right)^{-1}.
\end{equation}
\end{proposition}

Proof deferred to \cref{app:proof-prop-bimodal-local-gate}.

In \cref{eq:bimodal-local-gate}, the responsibility-weighted precision \(P_\tau(y,t)\) changes with \(y\).  Near component \(k\), \(G_{\rm lab}(y,t)\) is close to
\[
    F_k(t):=\alphat^2(\alphat^2I_d+\gammat P_k)^{-1}.
\]
If \(P_+\) and \(P_-\) are not simultaneously diagonalizable, then \(F_+(t)\) and \(F_-(t)\) have different eigenspaces whenever the spectral filter separates stiff and soft eigenvalues.  Any fixed spatially uniform matrix gate \(G(t)\) must miss at least one of them, since
\[
    \max_{k\in\{+,-\}}\|G(t)-F_k(t)\|_{\rm op}
    \ge \frac12\|F_+(t)-F_-(t)\|_{\rm op}.
\]
When \(C_{\rm ov}\) in \cref{eq:gmm-marginal-overlap-correction} is small, \(G_{\rm marg}\) is close to the corresponding \(F_k\).  The matrix separation above then gives excess score-estimation risk through \cref{eq:excess-risk}.  A scalar blend fails for the simpler reason already visible in \cref{cor:scalar-failure}: it cannot apply different attenuation along stiff and soft directions.  The centered primal Matrix Blend estimator is spatially varying in $(y,t)$ and targets the same normal-equation gate as LFGI; its difficulty is finite-reference estimation, not population expressivity.

\subsection{Closed-form residual scale in the bimodal stiffness example}
\label{subsec:bimodal-primal-failure-scale}

In the bimodal model, centered primal regression is difficult when the empirical residual--disagreement cross moment $\widehat C_{r_\star\delta}$ in \cref{eq:centered-regression-residual-risk} is large.  Continue using the component-conditioned score notation of \cref{prop:bimodal-local-gate}, and write
\[
    L_k(y,t)
    :=
    \frac{\alphat}{\gammat}I_d
    -
    \frac{1}{\alphat\gammat}G_{\rm lab}(y,t)A_k(t).
\]
Then the label-resolved gate residual decomposes as
\[
    r_\star(X_0;y,t)=a_K(y,t)+L_K(y,t)u_K .
\]
The term $a_K$ is a between-component score residual; the term $L_Ku_K$ is a stiffness-mismatch residual.  Both vanish in the single-Gaussian case, but they are generically nonzero in responsibility-overlap regions when the component stiffness frames are misaligned.  \Cref{eq:centered-regression-residual-risk} depends on this residual through the empirical cross moment $\widehat C_{r_\star\delta}$.

Let
\[
    Q_k(y,t):=B_k(t)^\top \Cdd(y,t)^{-1}B_k(t),
    \qquad
    R_k(y,t):=L_k(y,t)^\top L_k(y,t).
\]
The closed-form residual-leakage scale is
\begin{proposition}[Residual-leakage scale in the bimodal stiffness model]
\label[proposition]{prop:bimodal-residual-leakage-scale}
Under the label-resolved bimodal model of \cref{prop:bimodal-local-gate},
\begin{equation}
\label{eq:bimodal-centered-primal-noise-scale}
\begin{aligned}
V_{\rm mix}^{\rm cen}(y,t)
    &:={\mathbb E}\left[
        \|r_\star(X_0;y,t)\delta(X_0;y,t)^\top \Cdd(y,t)^{-1/2}\|_F^2
        \mid Y_t=y
    \right]
    \\
    &=
    \sum_{k\in\{+,-\}}\tau_k
    \Bigl\{
        \|a_k\|^2\tr(Q_k\Sigma^{X\mid Y}_k)
        +
        \tr(R_k\Sigma^{X\mid Y}_k)\tr(Q_k\Sigma^{X\mid Y}_k)
        +
        2\tr(R_k\Sigma^{X\mid Y}_kQ_k\Sigma^{X\mid Y}_k)
    \Bigr\}.
\end{aligned}
\end{equation}
\end{proposition}

Consequently the iid centered-primal perturbation has the order-of-magnitude variance scale
\[
    \E\left[
        \|\widehat C_{r_\star\delta}\Cdd^{-1/2}\|_F^2
        \mid Y_t=y
    \right]
    \asymp
    \frac{V_{\rm mix}^{\rm cen}(y,t)}{N},
\]
up to centering and self-normalization corrections.  The excess risk then multiplies the displayed scale $V_{\rm mix}^{\rm cen}(y,t)/N$ by the empirical inverse-amplification factor in \cref{eq:centered-residual-amplification-bound}.  The misaligned GMM exposes the amplification of a nonzero $\widehat C_{r_\star\delta}$ by the empirical inverse of $\widehat C_{\delta\delta}$.  Centered primal regression must estimate a residual-coupled covariance quotient, and the residual is largest exactly when responsibility overlap and stiffness-frame mismatch are active.  The proof of \cref{prop:bimodal-residual-leakage-scale} is deferred to \cref{app:proof-bimodal-residual-leakage}.

\subsection{Constant-Hessian Gaussian tie}
\label{subsec:constant-hessian-gaussian-tie}

The closed-form Gaussian formulas used in this subsection are collected in \cref{app:gaussian-formulas}.

In the singular Gaussian example, the LFGI Hessian-average error cancels exactly because the target observed information is constant.  The centered primal regression error also cancels exactly because the optimal linear relation holds pointwise; see \cref{prop:gaussian-centered-cancellation}.  Thus the Gaussian example is not a failure example against the centered primal estimator, and both matrix estimators recover the constant-Hessian gate exactly.

Let
\[
    p_0^\varepsilon=\N(m,P_\varepsilon^{-1}),
    \qquad
    P_\varepsilon\succ0,
    \qquad
    H_0(x)=P_\varepsilon
    \quad
    \text{for all }x.
\]
Then every nonempty finite-reference LFGI gate bank has
\begin{equation}
\label{eq:constant-hessian-lfgi-recovery}
    \Hhat_N(y,t)=P_\varepsilon,
    \qquad
    \Ghat_N(y,t)
    =
    \Gstareps(t)
    =
    \alphat^2
    \left(
        \alphat^2I_d+\gammat P_\varepsilon
    \right)^{-1}
\end{equation}
for every $y,t$ and every condition number $\kappa(P_\varepsilon)$.  Moreover, the centered primal regression gate satisfies $\widehat C_{\eb\delta}=-\Gstareps(t)\widehat C_{\delta\delta}$ for every finite weighted sample by the Gaussian cancellation in \cref{prop:gaussian-centered-cancellation}.  Hence it recovers the same gate on the empirical disagreement span, and recovers the full gate whenever $\widehat C_{\delta\delta}$ has full rank.

The constant-Hessian Gaussian case isolates what the exact cancellation does and does not show.  In that case, LFGI matches the population gate from any nonempty Hessian bank, while centered primal regression solves it from the exact pointwise linear relation once the empirical disagreement span identifies the relevant directions.  The LFGI advantage over centered primal regression is therefore not a Gaussian phenomenon.  It appears only when non-Gaussian residuals break the cancellation in \cref{prop:gaussian-centered-cancellation}.

\subsection{Residual-covariance inversion versus Hessian averaging}
\label{subsec:residual-coupling-regime}

The finite-reference comparison is between two error mechanisms: the centered residual-covariance quotient and the relative Hessian-average error.  LFGI and centered primal regression target the same population gate, but they differ in the finite-reference quantity that must be estimated.  Centered primal regression estimates a quotient of empirical moments,
\[
    \widehat G_{\rm cen}
    =
    -\widehat C_{\eb\delta}(\widehat C_{\delta\delta}+\lambdaridge I_d)^{-1},
\]
whereas LFGI estimates the conditional Hessian average \(H\) and then applies the deterministic map $\Psiop$
\[
    \Ghat_N
    =
    \alphat^2(\alphat^2 I_d+\gammat\Hhat_N)^{-1}.
\]
Thus the comparison is between the centered-primal error term $\widehat C_{r_\star\delta}\widehat C_{\delta\delta}^{-1}\Cdd^{1/2}$ in \cref{eq:centered-regression-residual-risk} and the LFGI Hessian-average error $\epsilon_H$ in \cref{eq:relative-hessian-concentration}.

For centered primal regression, the exact perturbation identity in \cref{eq:centered-regression-residual-risk} gives
\begin{equation*}
    \mathcal R(\widehat G_{\rm cen})-\mathcal R(\Gstar)
    =
    \nor{\widehat C_{r_\star\delta}(\widehat C_{\delta\delta}+\lambdaridge I_d)^{-1}\Cdd^{1/2}}_F^2
    \quad
    \text{up to the explicit ridge/projection term.}
\end{equation*}
The key finite-reference object is therefore the interaction between \(\widehat C_{r_\star\delta}\) and \(\widehat C_{\delta\delta}\).  A low-rank or ill-conditioned disagreement covariance is harmless in the centered Gaussian case because \(\widehat C_{r_\star\delta}=0\) exactly.  It becomes a failure mechanism in stiff non-Gaussian targets, where residual leakage \(\widehat C_{r_\star\delta}\neq0\) is multiplied by the empirical inverse of the disagreement covariance.

For LFGI, \cref{prop:risk-weighted-lfgi-perturbation,thm:hessian-reference-condition} give a sufficient upper bound on $\mathcal R(\widehat G_N)-\mathcal R(\Gstar)$
\begin{equation*}
    \mathcal R(\Ghat_N)-\mathcal R(\Gstar)
    \lesssim
    \alphat^4\Lambdapole(y,t)
    \left(
        \frac{v_A^2\log(2d/\delta)}{N_{\eff}(y,t)}
        +
        \text{lower-order envelope terms}
    \right).
\end{equation*}
In PSD or pole-separated regimes, the trace factor in \cref{thm:hessian-reference-condition} shows that raw stiff curvature is cancelled inside the resolvent trace factor.  The LFGI quantity to control is then the relative Hessian fluctuation $\epsilon_H=\gammat\nor{A^{-1/2}(\widehat H_N-H)A^{-1/2}}_{\op}$, not inversion of a noisy residual-coupled covariance matrix.

The same contrast appears in the bimodal misaligned-stiffness calculation, which gives closed-form LFGI-side quantities for the target family used to expose the centered-primal residual.  In the component-dominant regime, the true mixture observed information satisfies
\begin{equation}
\label{eq:gmm-observed-information-identity}
    H_0(x)
    =
    \sum_j \omega_j(x)P_j
    -
    \operatorname{Cov}_{j\sim \omega(x)}\!\bigl(P_j(x-m_j)\bigr),
    \qquad
    \omega_j(x):=\Pr(K=j\mid X=x).
\end{equation}
Indeed, with component scores \(s_j(x)=-P_j(x-m_j)\), the mixture score is
\(s_0(x)=\sum_j\omega_j(x)s_j(x)\), and the responsibilities satisfy
\(\nabla\omega_j(x)=\omega_j(x)\{s_j(x)-s_0(x)\}\).  Differentiating the responsibility-weighted score gives
\[
    \nabla s_0(x)
    =
    \sum_j\omega_j(x)(-P_j)
    +
    \sum_j\omega_j(x)(s_j(x)-s_0(x))s_j(x)^\top,
\]
so \(-\nabla s_0(x)\) is exactly \cref{eq:gmm-observed-information-identity}.  Thus the observed information reduces to the active component precision up to the displayed responsibility-covariance correction.  When clean responsibilities are sharp under the OU posterior or when the correction remains PSD/pole-separated, \(H(y,t)\) is close to \(P_\tau(y,t)\) and LFGI gives the same spatially varying gate \cref{eq:bimodal-local-gate} up to a controlled Hessian perturbation.  In the label-resolved idealization \(H_K=P_K\), the relative Hessian fluctuation entering \cref{ass:hessian-concentration} is simply
\begin{equation*}
    Z_k(y,t)
    =
    \gammat
    A(y,t)^{-1/2}
    \bigl(P_k-P_\tau(y,t)\bigr)
    A(y,t)^{-1/2},
    \qquad
    A=\alphat^2I_d+\gammat P_\tau.
\end{equation*}
Hence one may take
\begin{equation*}
    R_A(y,t)\le \max_k\|Z_k(y,t)\|_{\op},
    \qquad
    v_A^2(y,t)\le \sum_k\tau_k(y,t)\|Z_k(y,t)\|_{\op}^2.
\end{equation*}
No between-component score residual \(a_k\), no stiffness-mismatch residual \(L_ku_k\), and no inverse empirical disagreement covariance appear in this LFGI concentration scale.  This is the desired friendly case: local curvature and local stiffness are aligned well enough that the Hessian average concentrates, while centered primal regression still pays for residual leakage amplified by an empirical covariance inverse.

A useful sufficient-condition guide is therefore
\begin{equation}
\label{eq:relative-advantage-regime}
    \nor{\widehat C_{r_\star\delta} \Cdd^{-1/2}}_F^2
    \nor{\Cdd^{1/2}(\widehat C_{\delta\delta}+\lambdaridge I_d)^{-1}\Cdd^{1/2}}_{\op}^2
    \gg
    \alphat^4\Lambdapole(y,t)
    \frac{v_A^2}{N_{\eff}(y,t)},
\end{equation}
up to logarithmic, ridge, and envelope factors.  We do not use \cref{eq:relative-advantage-regime} as a calibrated predictor of empirical error.  Its role is to identify the two distinct finite-reference error sources: centered primal regression must control residual leakage through an empirical covariance inverse, whereas LFGI must control Hessian-average perturbation.  The right-hand side is a conservative sufficient upper bound and can be loose in the risk-weighted directions that matter for gate capture.  The validation study in \cref{subsec:validation-score-gate-sweeps} therefore checks the estimator-level consequence directly through score RMSE (see Appendix \cref{app:metrics}, Def.~\ref{def:score-rmse}) and the auxiliary gate-capture breakdown in Appendix \cref{app:gate-capture-breakdown}, not through a fitted internal diagnostic ratio.

\section{Auxiliary Gate-Capture Breakdown}
\label{app:gate-capture-breakdown}

The gate-estimation sweep uses the $d=8$ misaligned singular-subspace GMM from \cref{subsec:exp-misaligned-gmm-d8}.  The target, query distribution, and score-bank construction are held fixed while the gate-bank size varies over $N_g\in\{2^6,2^7,\ldots,2^{12}\}$.  The comparison is restricted to Matrix Blend and LFGI, the two full-matrix estimators of the same population gate $\Gstar(y,t)$.  Matrix Blend estimates the centered primal covariance quotient in \cref{eq:centered-primal-gate,eq:centered-covariances}; LFGI estimates the OU-conditional observed-information average and applies the resolvent map $\Psiop$.

\Cref{fig:validation-gate-capture-appendix} reports both gate-capture diagnostics.  The top row reports relative Frobenius gate error, which measures ordinary matrix-scale recovery of $\Gstar$.  The bottom row reports the risk-weighted error from Def.~\ref{def:gate-capture}, normalized by $\tr(\Cdd)$.  The risk-weighted diagnostic is the one tied directly to excess conditional score risk through \cref{eq:excess-risk}.

\begin{figure}[H]
\centering
\includegraphics[width=0.315\textwidth]{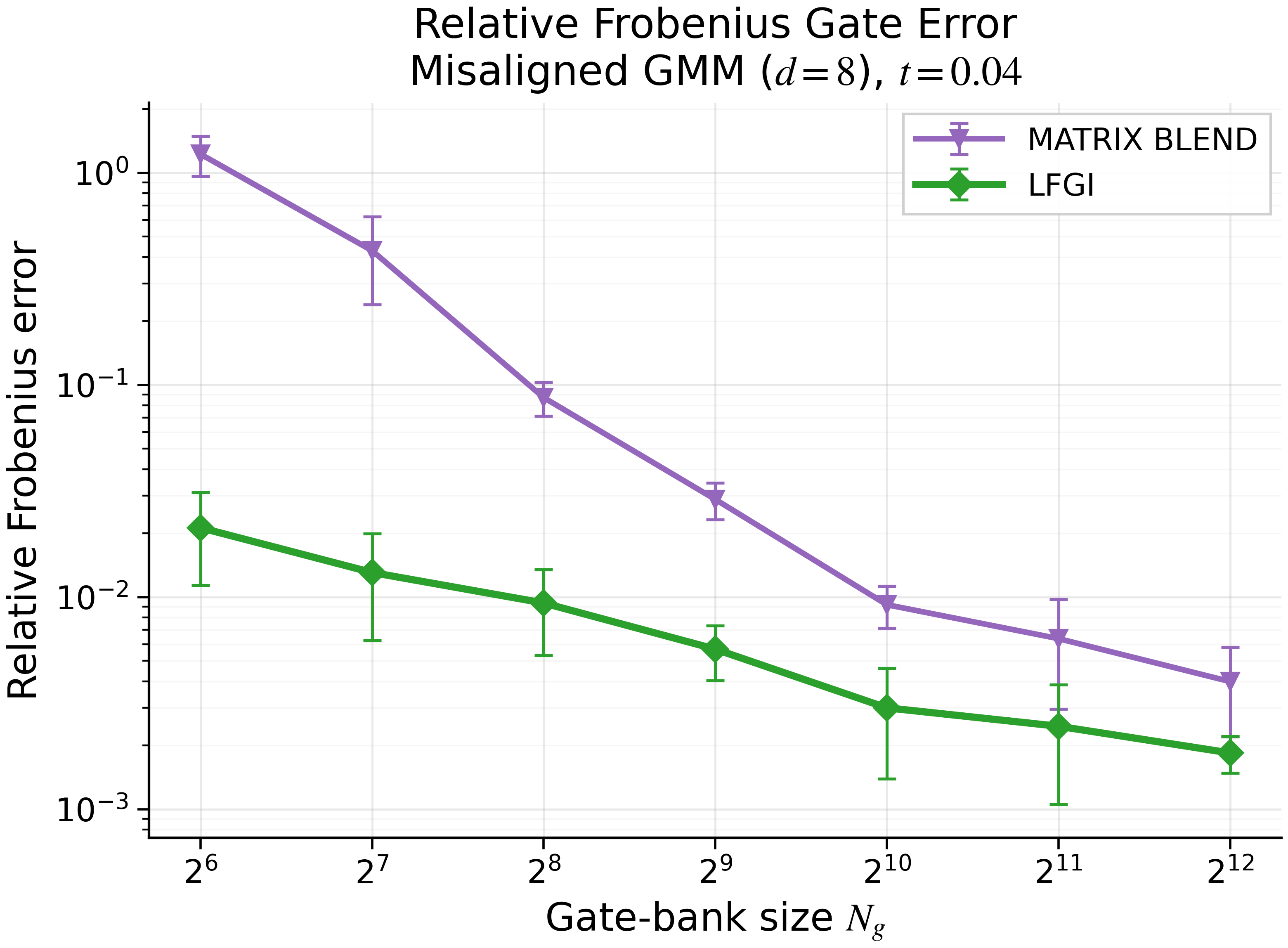}\hfill
\includegraphics[width=0.315\textwidth]{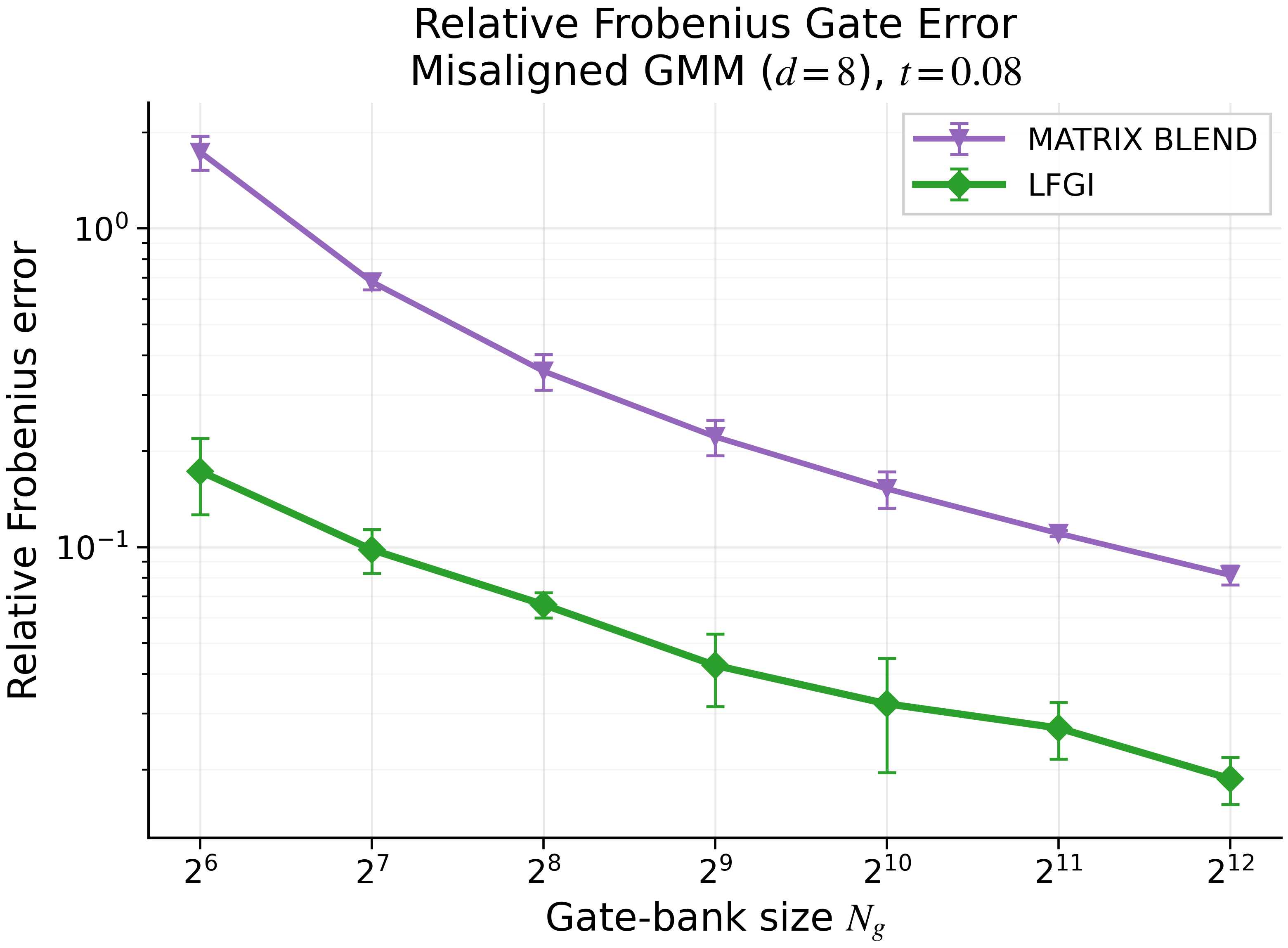}\hfill
\includegraphics[width=0.315\textwidth]{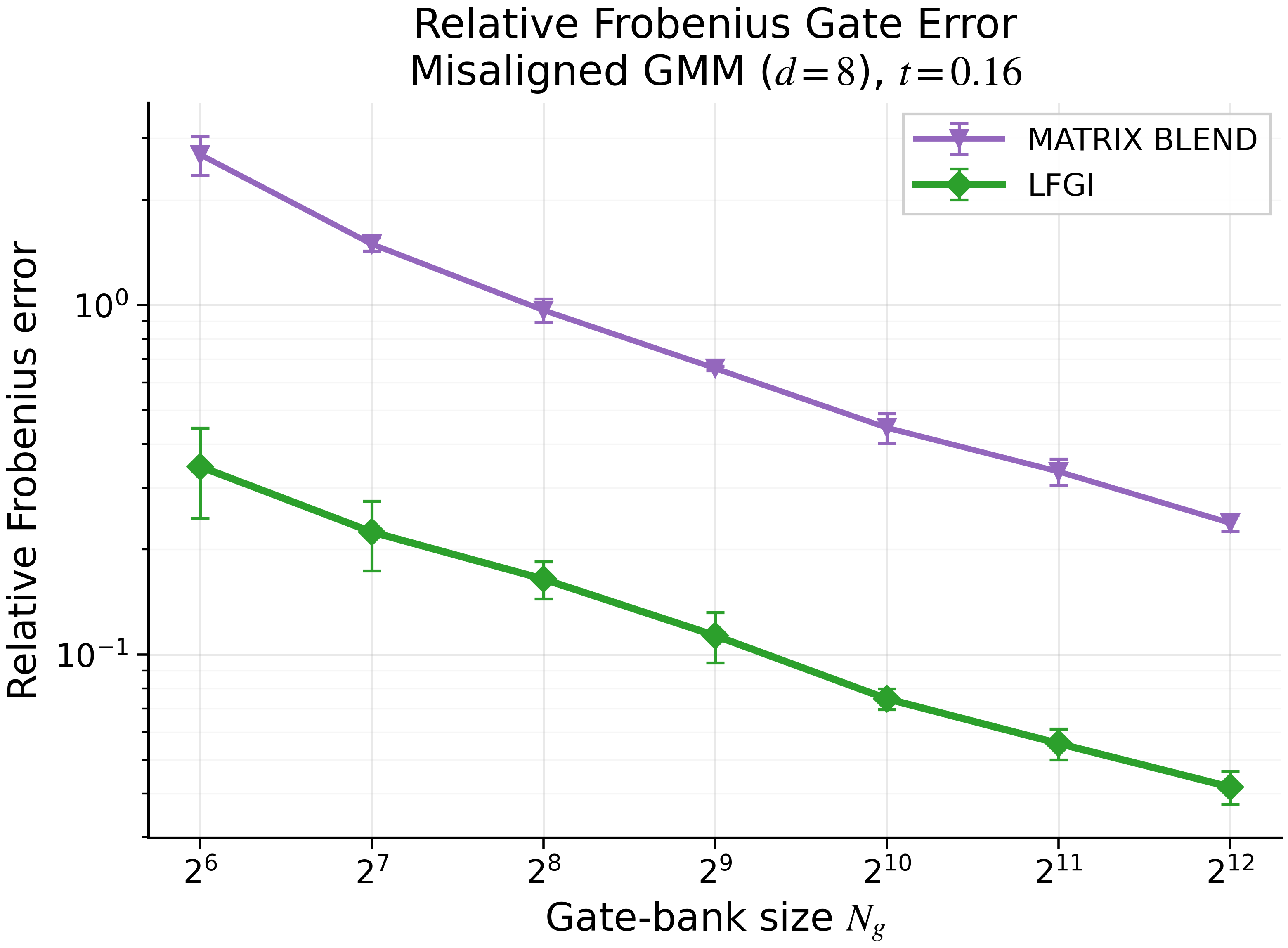}
\par\vspace{0.75em}
\includegraphics[width=0.315\textwidth]{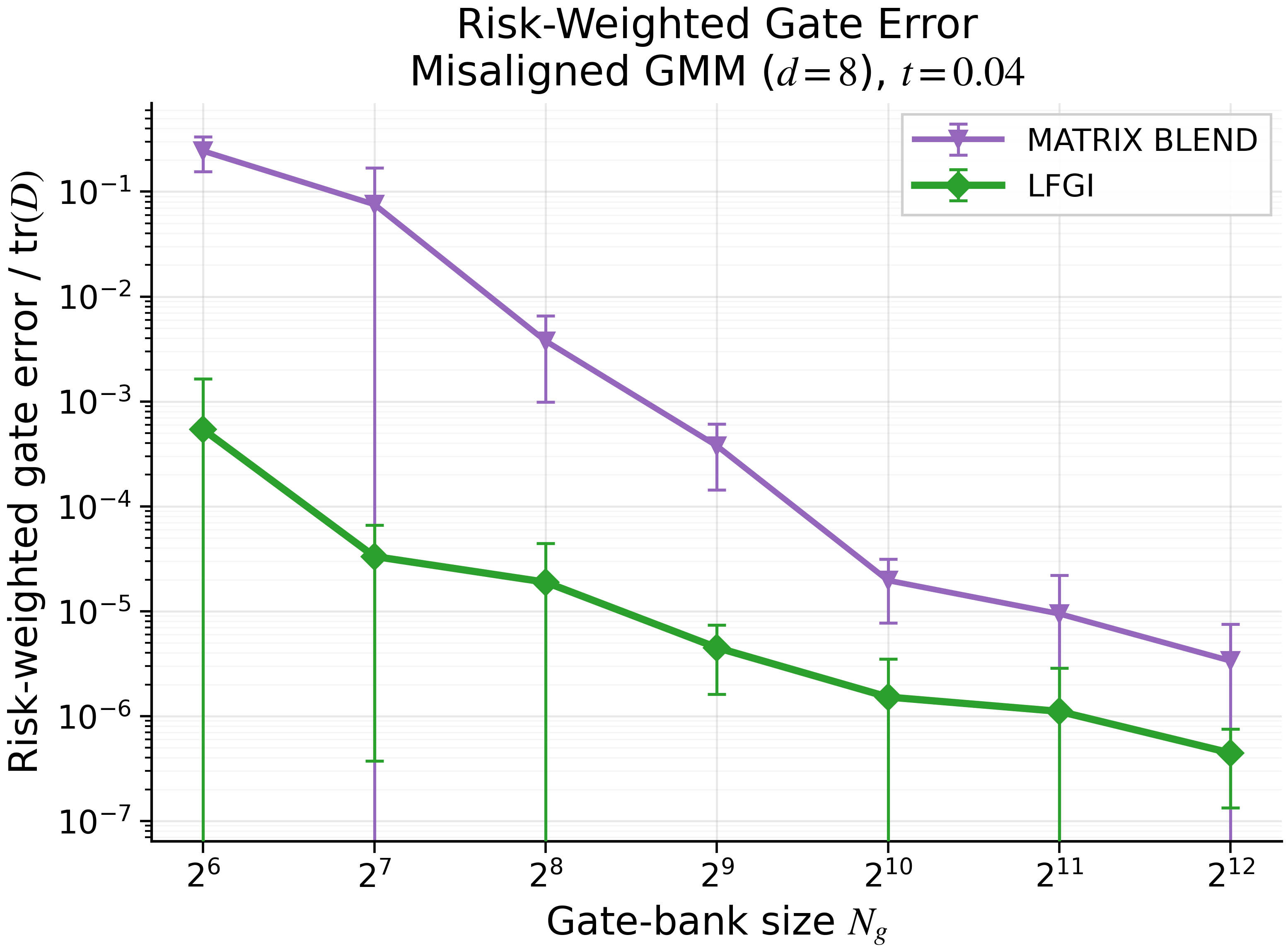}\hfill
\includegraphics[width=0.315\textwidth]{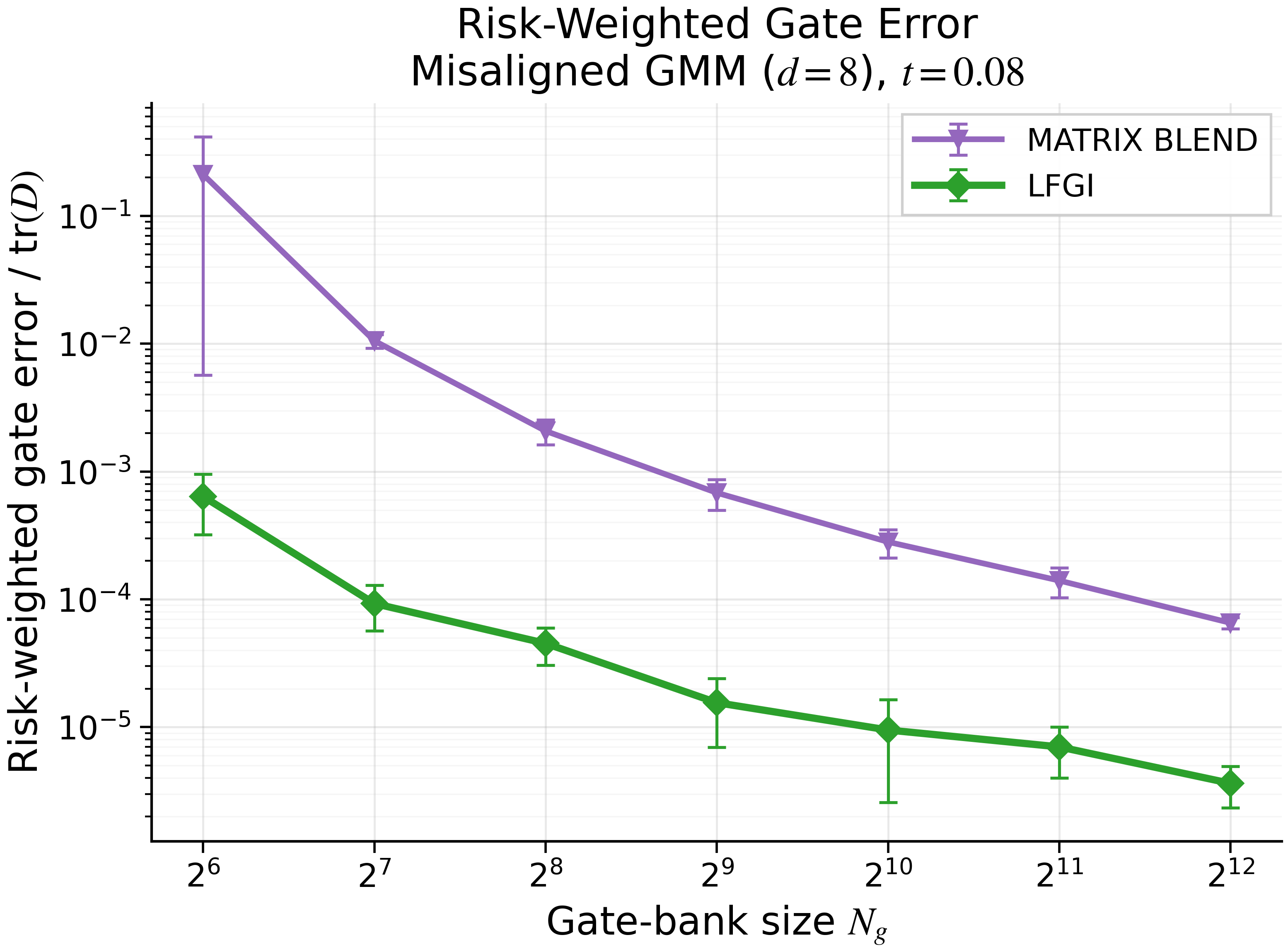}\hfill
\includegraphics[width=0.315\textwidth]{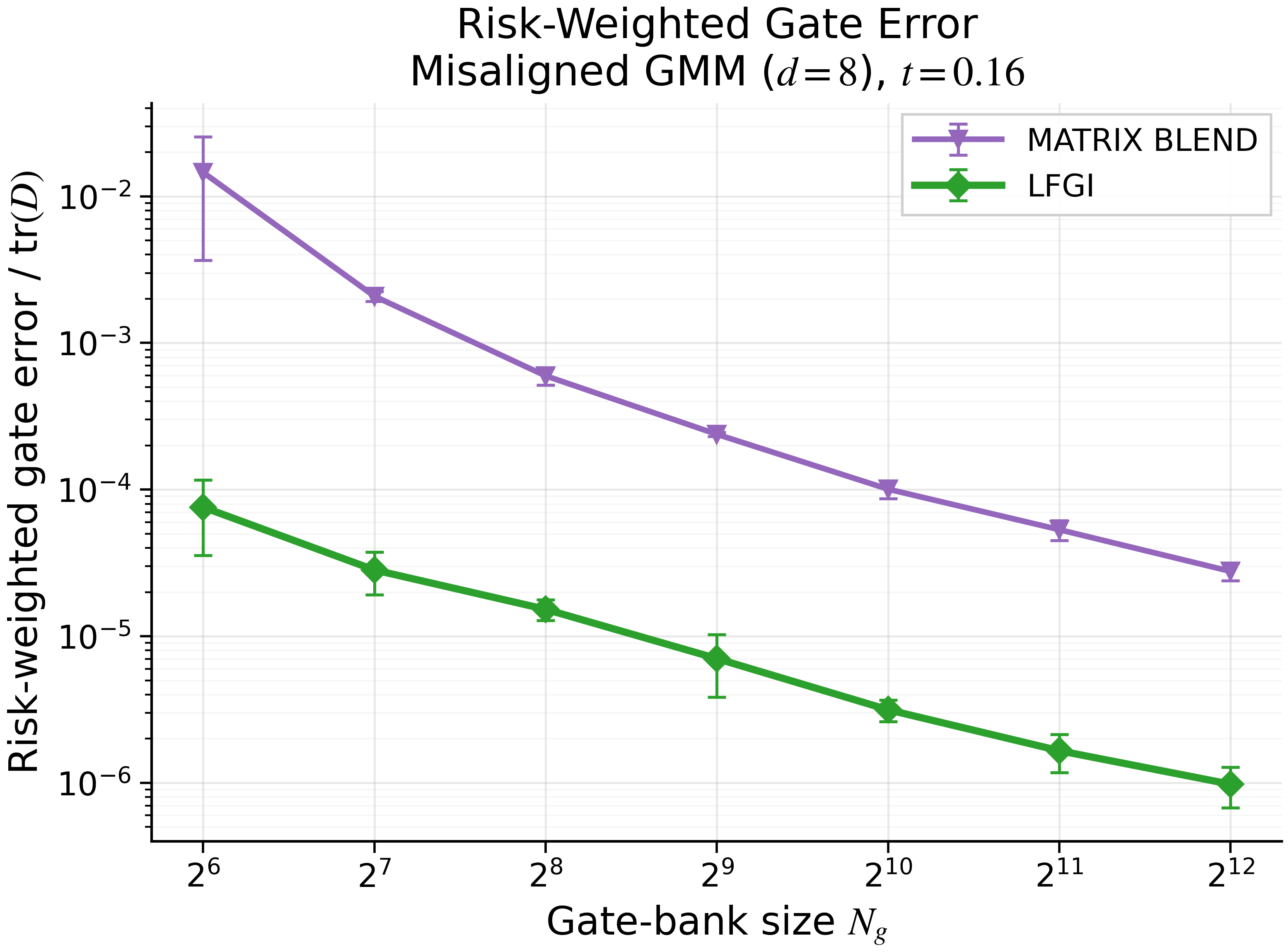}
\caption{Auxiliary gate-capture sweep on the $d=8$ misaligned singular-subspace GMM.  Columns use $t=0.04,0.08,0.16$; rows show relative Frobenius gate error and risk-weighted gate error.  LFGI has smaller error than Matrix Blend across the displayed gate-bank sizes in both geometries.}
\label{fig:validation-gate-capture-appendix}
\end{figure}

\section{Auxiliary Pole Diagnostics}
\label{app:pole-diagnostics}

\subsection{Neal-funnel shifted-pole audit}
\label{app:funnel-pole-diagnostics}

The pole diagnostic checks the finite-reference quantities that appear in \cref{prop:risk-weighted-lfgi-perturbation,thm:hessian-reference-condition}.  For each audited pair $(y,t)$, a large held-out reference bank estimates
\[
    H(y,t)=\E[H_0(X_0)\mid Y_t=y],
    \qquad
    \Cdd(y,t)=\E[\delta\delta^\top\mid Y_t=y],
\]
and the shifted matrix $A(y,t)=\alphat^2I_d+\gammat H(y,t)$ is diagonalized.  The table reports whether $A(y,t)$ is nonpositive, whether a nonpositive shifted eigenvalue occurs in a direction carrying non-negligible $\Cdd$ mass, the lower tail of $\lambda_{\min}(A)$, and the upper tail of the relative Hessian error
\[
    \epsilon_H
    =
    \gammat
    \|A^{-1/2}(\widehat H_N-H)A^{-1/2}\|_{\op}.
\]
It also reports the upper tail of
\[
    \operatorname{CR}(y,t)
    =
    \frac{
    \alphat^4
    \left(\epsilon_H/(1-\epsilon_H)\right)^2
    \Lambdapole(y,t)
    }{
    \Gainstar(y,t)
    },
\]
where \cref{eq:lfgi-risk-capture-condition} requires $\operatorname{CR}(y,t)\le\eta$ when the small-perturbation condition holds.  The strict pass rate in \cref{tab:funnel-pole-audit-summary} uses $\eta=0.5$ and $\epsilon_H\le1/2$.

\begin{table}[H]
\centering
\small
\begin{tabular}{@{}rrrrrrr@{}}
\toprule
$t$ & Pass rate & Nonpos. $A$ rate & Active-pole rate & $q_{0.05}\lambda_{\min}(A)$ & $q_{0.90}\epsilon_H$ & $q_{0.90}\operatorname{CR}$ \\
\midrule
$0.01$ & $0.982$ & $0.000$ & $0.000$ & $0.940$ & $0.181$ & $0.052$ \\
$0.03$ & $0.953$ & $0.000$ & $0.000$ & $0.871$ & $0.317$ & $0.256$ \\
$0.10$ & $0.693$ & $0.000$ & $0.000$ & $0.712$ & $0.594$ & $2.465$ \\
$0.30$ & $0.510$ & $0.000$ & $0.000$ & $0.448$ & $0.767$ & $11.332$ \\
$1.00$ & $0.678$ & $0.000$ & $0.000$ & $0.120$ & $0.916$ & $127.016$ \\
$3.00$ & $0.756$ & $0.000$ & $0.000$ & $4.583$ & $0.528$ & $0.453$ \\
\bottomrule
\end{tabular}
\caption{Auxiliary pole audit for Neal's funnel, using $512$ audited OU query points per displayed time.  Columns report nonpositive shifted-Hessian events, active-pole events, and the perturbation-to-gain ratio from \cref{eq:lfgi-risk-capture-condition}.  No shifted poles are observed, although the sufficient condition is not uniformly certified.}
\label{tab:funnel-pole-audit-summary}
\end{table}

\section{Density, Evidence, and Calibration Metrics}
\label{app:density-metrics}

For evaluation points $x_i$ with true unnormalized energy $E_i=-\log\tildep(x_i)$ and probability-flow estimated energy $\Epf_i=-\log\qpf(x_i)$, the density diagnostics compare the scalar arrays $\{E_i\}$ and $\{\Epf_i\}$ on a held-out density-evaluation bank.  The Darcy-flow comparison table reports the diagnostics defined below: Spearman rank correlation, central affine fit slope, central affine $R^2$, central slope-normalized RMSE, and pointwise NLL, with the affine columns computed on the robust central band.

\begin{definition}[Robust central density diagnostics]
\label[definition]{def:central-density}
Given percentiles $q_{\rm lo}<q_{\rm hi}$, define the central index set
\[
    \mathcal I_{\rm cen}
    =
    \{i: Q_{q_{\rm lo}}(E)\le E_i\le Q_{q_{\rm hi}}(E)\},
\]
where $Q_q(E)$ is the empirical $q$th percentile of the true energy values on the evaluation bank.  Central Pearson and central Spearman are the corresponding correlations restricted to $\mathcal I_{\rm cen}$.  The central affine fit is the least-squares fit
\[
    \Epf_i \approx a E_i+b,
    \qquad i\in\mathcal I_{\rm cen},
\]
with slope $a$, intercept $b$, and central affine $R^2$.  The affine-normalized surrogate energy plotted in the density grids is
\[
    Q_{\rm aff}(x_i)=\frac{\Epf_i-b}{a},
\]
and the central residual is $Q_{\rm aff}(x_i)-E_i$.  The central affine RMSE is the RMSE of the fitted $\Epf_i\approx aE_i+b$ relation on $\mathcal I_{\rm cen}$, while the central slope-normalized RMSE is the RMSE of $Q_{\rm aff}(x_i)-E_i$ on $\mathcal I_{\rm cen}$.  The Darcy-flow table uses the $3$--$97\%$ central band, matching the displayed density-energy grid.  Spearman reports rank ordering over the full evaluation bank unless explicitly marked as central.
\end{definition}

\begin{definition}[Pointwise probability-flow NLL]
\label[definition]{def:pf-nll}
For a normalized probability-flow surrogate $\qpf$ evaluated on held-out posterior points $x_i$, the pointwise negative log likelihood is
\[
    \operatorname{NLL}_{\qpf}
    =
    -\frac1n\sum_{i=1}^n \log \qpf(x_i)
    =
    \frac1n\sum_{i=1}^n \Epf_i .
\]
This column checks the normalized density value assigned by the probability-flow surrogate on the held-out density-evaluation bank; it is distinct from the KDE NLL in \cref{def:kde-nll}.
\end{definition}

\begin{definition}[Evidence and correction-weight diagnostics]
\label[definition]{def:evidence-diagnostics}
The evidence diagnostics are pointwise $\widehat{\logZ}_{\rm pt}$ on held-out posterior points, reciprocal-IS $\widehat{Z}_{\rm rIS}^{-1}$ on pilot points, optional forward-IS/bridge estimates from samples generated by $\qpf$, and the associated directional ESS values.  If $Y\sim\qpf$, then
\[
    Z=\E_{\qpf}\left[\frac{\tildep(Y)}{\qpf(Y)}\right],
\]
whereas if $X\sim p_0$, then
\[
    \frac1Z=\E_{p_0}\left[\frac{\qpf(X)}{\tildep(X)}\right].
\]
The realized correction ESS is computed from the normalized weights proportional to $\tildep(Y_j)/\qpf(Y_j)$ for generated samples $Y_j\sim\qpf$.  In known-$Z$ experiments, the reported evidence error is the absolute difference $|\widehat{\logZ}-\log Z|$ before any affine calibration.
\end{definition}

Heavy-tailed inverse-problem energy banks can make global Pearson and global affine $R^2$ misleading.  A single extreme-energy particle can dominate the covariance and make a local MAP--Laplace proxy appear nearly perfectly correlated with the target, even when its posterior-bulk calibration is poor.  The central diagnostics in \cref{def:central-density} separate tail leverage from typical posterior calibration.  The displayed density-calibration figures use robust axes, affine-normalized residuals, and a common LFGI--GN axis reference across rows.


\section{Proofs for OU identities}
\label{app:ou-proofs}

\subsection{Proof of \cref{lem:posterior-score-disagreement}}
\label{app:proof-lem-posterior-score-disagreement}

From \cref{eq:rho-ou},
\[
    \log\rhoOU(x)
    =
    \log p_0(x)
    -
    \frac{\|y-\alphat x\|^2}{2\gammat}
    -
    \log p_t(y)
    -
    \frac d2\log(2\pi\gammat).
\]
The last two terms are independent of $x$.  Differentiating with respect to $x$ gives
\[
    \nabla_x\log\rhoOU(x)
    =
    s_0(x)+\frac{\alphat(y-\alphat x)}{\gammat}
    =
    s_0(x)-\alphat b(x;y,t)
    =
    \alphat(c(x;t)-b(x;y,t))
    =
    \alphat \delta(x;y,t).
\]
This proves the claim.

\section{Proofs for the matrix-gate risk}
\label{app:risk-proofs}

\subsection{Proof of \cref{prop:normal-equation}}
\label{app:proof-prop-normal-equation}

At fixed $(y,t)$, suppress the arguments and write all expectations with respect to $\rhoOU=p_0(\cdot\mid Y_t=y)$.  Expanding the quadratic risk gives
\[
\begin{aligned}
    \mathcal R(G)
    &=
    \E_{\rhoOU}\|\eb+G\delta\|^2  \\
    &=
    \E_{\rhoOU}\|\eb\|^2
    +2\E_{\rhoOU}[\eb^\top G\delta]
    +\E_{\rhoOU}[\delta^\top G^\top G\delta]  \\
    &=
    \E_{\rhoOU}\|\eb\|^2
    +2\tr(G\Cebd^\top)
    +\tr(G\Cdd G^\top),
\end{aligned}
\]
where $\Cdd=\E_{\rhoOU}[\delta\delta^\top]$ and $\Cebd=\E_{\rhoOU}[\eb\delta^\top]$.  The Euclidean gradient with respect to $G$ is
\[
    \nabla_G\mathcal R(G)=2(G\Cdd+\Cebd).
\]
Thus every unconstrained minimizer satisfies $G\Cdd+\Cebd=0$.  If $\Cdd$ is nonsingular, the solution is unique and equals $G=-\Cebd \Cdd^{-1}$.

\subsection{Proof of \cref{prop:residual-risk}}
\label{app:proof-prop-residual-risk}

Let $\Gstar$ solve $\Gstar \Cdd+\Cebd=0$ and set $\Delta=G-\Gstar$.  Then
\[
    \eb+G\delta=(\eb+\Gstar\delta)+\Delta\delta.
\]
Expanding the risk difference gives
\[
\begin{aligned}
    \mathcal R(G)-\mathcal R(\Gstar)
    &=
    2\E_{\rhoOU}[(\eb+\Gstar\delta)^\top\Delta\delta]
    +
    \E_{\rhoOU}[\delta^\top\Delta^\top\Delta\delta].
\end{aligned}
\]
The cross term vanishes because
\[
    \E_{\rhoOU}[(\eb+\Gstar\delta)\delta^\top]
    =
    \Cebd+\Gstar \Cdd=0.
\]
Therefore
\[
    \mathcal R(G)-\mathcal R(\Gstar)
    =
    \tr(\Delta \Cdd\Delta^\top),
\]
which is \cref{eq:excess-risk}.  The residual identity follows from
\[
    \OmegaG(G)=G\Cdd+\Cebd=(G-\Gstar)\Cdd=\Delta \Cdd.
\]
If $\Cdd$ is nonsingular, then
\[
    \tr(\OmegaG(G)\Cdd^{-1}\OmegaG(G)^\top)
    =
    \tr(\Delta \Cdd\Cdd^{-1}\Cdd\Delta^\top)
    =
    \tr(\Delta \Cdd\Delta^\top).
\]

\subsection{Proof of \cref{prop:gaussian-gate}}
\label{app:proof-prop-gaussian-gate}

Let $\mu_{y,t}=\E[X_0\mid Y_t=y]$ and write $x=\mu_{y,t}+u$.  Gaussian conditioning gives
\[
    X_0\mid Y_t=y\sim
    \N\!\bigl(\mu_{y,t},\,\gammat(\alphat^2I_d+\gammat P)^{-1}\bigr).
\]
Since $b(x;y,t)=(\alphat x-y)/\gammat$ is affine in $x$, centering under the conditional law gives
\[
    b(x;y,t)-s_t(y)=\frac{\alphat}{\gammat}u.
\]
The posterior-score relation also gives, by direct substitution into $c-b$,
\[
    \delta(x;y,t)
    =
    -\frac{1}{\alphat\gammat}(\alphat^2I_d+\gammat P)u.
\]
For
\[
    G_t=\alphat^2(\alphat^2I_d+\gammat P)^{-1},
\]
these two displays imply
\[
    b(x;y,t)-s_t(y)=-G_t\delta(x;y,t)
    \qquad\text{for every }x.
\]
Thus the local risk residual $\eb+G_t\delta$ vanishes pointwise, so $G_t$ is the risk-minimizing gate.  Since $P\succ0$, the conditional covariance of $\delta$ is nonsingular, and the minimizer is unique by \cref{prop:normal-equation}.  If $Pu_j=\lambda_j u_j$, spectral calculus gives
\[
    G_tu_j
    =
    \frac{\alphat^2}{\alphat^2+\gammat\lambda_j}u_j
    =
    \frac{1}{1+\lambdat\lambda_j}u_j,
\]
which proves the stated formula.

\subsection{Proof of \cref{cor:scalar-failure}}
\label{app:proof-cor-scalar-failure}

If a scalar gate $gI_d$ equaled the Gaussian optimal gate, then for every eigenvector $u_j$ carrying nonzero disagreement variance one would have
\[
    gu_j=\Gstar u_j=\psi_t(\lambda_j)u_j.
\]
Thus $g=\psi_t(\lambda_j)$ on every such eigendirection.  If two active eigenvalues $\lambda_i\neq\lambda_j$ have $\psi_t(\lambda_i)\neq\psi_t(\lambda_j)$, no single scalar $g$ can satisfy both equations.  Since $\psi_t(\lambda)=\alphat^2/(\alphat^2+\gammat\lambda)$ is strictly decreasing for $t>0$, unequal eigenvalues give unequal attenuation factors.

\subsection{Proof of \cref{prop:gaussian-centered-cancellation}}
\label{app:proof-prop-gaussian-centered-cancellation}

For the Gaussian target, the OU posterior $X_0\mid Y_t=y$ is Gaussian with covariance
\[
    \gammat(\alphat^2 I_d+\gammat P)^{-1}
\]
and mean $\mu_{y,t}$.  Write $x=\mu_{y,t}+u$.  Direct substitution into the definitions of $b$ and $c$ gives
\[
    b(x;y,t)-s_t(y)=\frac{\alphat}{\gammat}u,
    \qquad
    \delta(x;y,t)= -\frac{1}{\alphat\gammat}(\alphat^2 I_d+\gammat P)u.
\]
Multiplying the second display by
\[
    \Gstar(t)=\alphat^2(\alphat^2I_d+\gammat P)^{-1}
\]
gives
\[
    b(x;y,t)-s_t(y)=-\Gstar(t)\delta(x;y,t)
    \qquad\text{for every }x.
\]
Thus $r_\star\equiv0$.  Subtracting weighted empirical means gives
\[
    b_i-\bar b=-\Gstar(t)(\delta_i-\bar\delta),
\]
and therefore $\widehat C_{\eb\delta}=-\Gstar\widehat C_{\delta\delta}$ exactly.  It follows that $\widehat G_{\rm cen}=\Gstar$ on the empirical disagreement span, and as a full matrix whenever $\widehat C_{\delta\delta}$ has full rank.

\subsection{Proof of \cref{prop:bimodal-local-gate}}
\label{app:proof-prop-bimodal-local-gate}

The posterior mixture formula follows by multiplying the Gaussian OU likelihood by each Gaussian component of the prior. Within component $k$, the conditional covariance and mean are
\[
    \Sigma^{X\mid Y}_k(t)
    =
    \left(P_k+\frac{\alphat^2}{\gammat}I_d\right)^{-1}
    =
    \gammat(\alphat^2I_d+\gammat P_k)^{-1},
    \qquad
    \mu_k(y,t)
    =
    \Sigma^{X\mid Y}_k(t)\left(P_km_k+\frac{\alphat}{\gammat}y\right).
\]
The component marginal score is $s^Y_{t,k}$, and differentiating the two-component Gaussian mixture marginal gives
\[
    s_t(y)=\sum_{k\in\{+,-\}}\tau_k(y,t)s^Y_{t,k}(y).
\]
Condition on $K=k$ and $Y_t=y$, and write $X_0=\mu_k(y,t)+u_k$ with
$u_k\sim\mathcal N(0,\Sigma^{X\mid Y}_k(t))$.  Direct substitution into the component-score signal gives
\[
    b(X_0;y,t)-s_t(y)
    =
    a_k(y,t)+\frac{\alphat}{\gammat}u_k,
    \qquad
    \delta(X_0;y,t)=B_k(t)u_k,
\]
where $a_k(y,t)=s^Y_{t,k}(y)-s_t(y)$ and
$B_k(t)=-(\alphat\gammat)^{-1}(\alphat^2I_d+\gammat P_k)$.
Since $\E[u_k\mid K=k,Y_t=y]=0$, the between-component term $a_k$ contributes no cross moment with $\delta$.  Therefore
\[
    \Cebd
    =
    \sum_k \tau_k\frac{\alphat}{\gammat}\Sigma^{X\mid Y}_k B_k^\top
    =
    -\frac{1}{\gammat}I_d.
\]
Similarly,
\[
    \Cdd
    =
    \sum_k \tau_k B_k\Sigma^{X\mid Y}_k B_k^\top
    =
    \sum_k \tau_k\frac{1}{\alphat^2\gammat}(\alphat^2I_d+\gammat P_k)
    =
    \frac{1}{\gammat}I_d+\frac{1}{\alphat^2}P_\tau.
\]
Substituting these formulas into the normal-equation gate gives
\[
    G_{\rm lab}(y,t)
    =
    \alphat^2(\alphat^2I_d+\gammat P_\tau(y,t))^{-1},
\]
which is \cref{eq:bimodal-local-gate}.

\subsection{Verification of the marginal-score overlap correction}
\label{app:proof-bimodal-marginal-overlap}

For the same two-component mixture, let
\[
    \omega_k(x):=\Pr(K=k\mid X_0=x)
    =
    \frac{\pi_k\N(x;m_k,P_k^{-1})}
         {\sum_{\ell\in\{+,-\}}\pi_\ell\N(x;m_\ell,P_\ell^{-1})}.
\]
The marginal target score satisfies
\[
    s_0(x)=\nabla_x\log p_0(x)=\sum_{k\in\{+,-\}}\omega_k(x)s_k(x).
\]
Differentiating this expression gives the observed information identity
\[
    -\nabla_x^2\log p_0(x)
    =
    \sum_{k\in\{+,-\}}\omega_k(x)P_k
    -
    \omega_+(x)\omega_-(x)
    \{s_+(x)-s_-(x)\}\{s_+(x)-s_-(x)\}^\top .
\]
Because \(K\to X_0\to Y_t\) is a Markov chain,
\(\E[\omega_k(X_0)\mid Y_t=y]=\Pr(K=k\mid Y_t=y)=\tau_k(y,t)\).  Therefore
\[
    H_{\rm marg}(y,t)
    :=
    \E[-\nabla_x^2\log p_0(X_0)\mid Y_t=y]
    =
    P_\tau(y,t)-C_{\rm ov}(y,t),
\]
with \(C_{\rm ov}\) defined in \cref{eq:gmm-marginal-overlap-correction}.  Hence the implemented marginal-score gate is
\[
    G_{\rm marg}(y,t)
    =
    \alphat^2
    \left(\alphat^2I_d+\gammat\{P_\tau(y,t)-C_{\rm ov}(y,t)\}\right)^{-1}.
\]
Let \(A_{\rm lab}:=\alphat^2I_d+\gammat P_\tau(y,t)\).  If
\[
    \epsilon_{\rm ov}(y,t)
    :=
    \gammat\|A_{\rm lab}^{-1/2}C_{\rm ov}(y,t)A_{\rm lab}^{-1/2}\|_{\rm op}<1,
\]
then the Neumann-series bound gives
\[
    \|A_{\rm lab}^{1/2}(G_{\rm marg}-G_{\rm lab})A_{\rm lab}^{1/2}\|_{\rm op}
    \le
    \alphat^2\frac{\epsilon_{\rm ov}}{1-\epsilon_{\rm ov}}.
\]
The parameter \(\epsilon_{\rm ov}\) is small in separated component-local regions.  Indeed, if the log-odds
\(\ell(x):=\log\{\omega_+(x)/\omega_-(x)\}\) satisfies \(|\ell(x)|\ge L\) on a set \(E_y\), then
\(\omega_+(x)\omega_-(x)\le e^{-L}\) on \(E_y\), and
\[
    \|C_{\rm ov}(y,t)\|_{\rm op}
    \le
    e^{-L}\,\E[\|s_+(X_0)-s_-(X_0)\|^2\mid Y_t=y]
    +
    \frac14\E[\|s_+(X_0)-s_-(X_0)\|^2\mathbf 1_{E_y^c}\mid Y_t=y].
\]
Thus increasing the mixture separation makes the marginal-score correction exponentially small on the component-local regions, up to the displayed Gaussian tail term.

\subsection{Proof of \cref{prop:bimodal-residual-leakage-scale}}
\label{app:proof-bimodal-residual-leakage}
\begin{proof}
Condition on $K=k$ and $Y_t=y$.  Then $u_k\sim\N(0,\Sigma^{X\mid Y}_k)$, $r_\star=a_k+L_ku_k$, and $\delta=B_ku_k$.  Hence
\[
    \|r_\star\delta^\top\Cdd^{-1/2}\|_F^2
    =
    \|a_k+L_ku_k\|^2\,u_k^\top Q_ku_k .
\]
The cross term $2a_k^\top L_ku_k\,u_k^\top Q_ku_k$ has zero $\E_{k,y}$-expectation because it is an odd centered-Gaussian moment.  Here $\E_{k,y}$ denotes expectation over $u_k\sim\N(0,\Sigma^{X\mid Y}_k(t))$.  For $R_k=L_k^\top L_k$, Isserlis' identity \citep{isserlis1918formula} gives\[
    \E_{k,y}\{(u_k^\top R_ku_k)(u_k^\top Q_ku_k)\}
    =
    \tr(R_k\Sigma^{X\mid Y}_k)\tr(Q_k\Sigma^{X\mid Y}_k)
    +2\tr(R_k\Sigma^{X\mid Y}_kQ_k\Sigma^{X\mid Y}_k).
\]
The deterministic part contributes $\|a_k\|^2\E_{k,y}[u_k^\top Q_ku_k]=\|a_k\|^2\tr(Q_k\Sigma^{X\mid Y}_k)$.  Averaging the resulting conditional expression over $K$ with weights $\tau_k(y,t)$ proves \cref{eq:bimodal-centered-primal-noise-scale}.
\end{proof}

\subsection{Proof of \cref{thm:direct-plugin-ill-conditioning}}
\label{app:proof-thm-direct-plugin-ill-conditioning}

  The empirical and population normal equations are
\[
    \widehat G_{\rm cen}\widehat C_{\delta\delta}+\widehat C_{\eb\delta}=0,
    \qquad
    \Gstar \Cdd+\Cebd=0.
\]
With $\widehat C_{\delta\delta}=\Cdd+\Ddd$ and $\widehat C_{\eb\delta}=\Cebd+\Debd$,
\[
    (\widehat G_{\rm cen}-\Gstar)\widehat C_{\delta\delta}
    =
    -(\Debd+\Gstar\Ddd).
\]
Multiplication by $\widehat C_{\delta\delta}^{-1}\Cdd^{1/2}$ and \cref{eq:excess-risk} give
\[
    \mathcal R(\widehat G_{\rm cen})-\mathcal R(\Gstar)
    =
    \| (\Debd+\Gstar\Ddd)\widehat C_{\delta\delta}^{-1}\Cdd^{1/2}\|_F^2.
\]
If $r_{\delta\delta}=\|\Cdd^{-1/2}\Ddd \Cdd^{-1/2}\|_{\op}<1$, write
\[
    \widehat C_{\delta\delta}=\Cdd^{1/2}(I+B)\Cdd^{1/2},
    \qquad
    B=\Cdd^{-1/2}\Ddd \Cdd^{-1/2}.
\]
Then $\|(I+B)^{-1}\|_{\op}\le (1-r_{\delta\delta})^{-1}$, yielding \cref{eq:direct-moment-risk-bound}.

\section{Proof of the Laplace--Fisher Gate Identity}
\label{app:lfgi-proofs}

\subsection{Proof of \cref{lem:M-fisher}}
\label{app:proof-lem-m-fisher}

By \cref{lem:posterior-score-disagreement}, $q=\alphat\delta$.  Therefore
\[
    \Cdd
    =
    \E[\delta\delta^\top\mid Y_t=y]
    =
    \frac1{\alphat^2}
    \E_{\rhoOU}[q(X)q(X)^\top]
    =
    \frac1{\alphat^2}\mathcal I(\rhoOU).
\]
For the Bartlett identity, write componentwise
\[
    \mathcal I_{ij}(\rhoOU)
    =
    \int (\partial_i\log\rhoOU)(\partial_j\log\rhoOU)\rhoOU\,\dd x
    =
    \int (\partial_i\rhoOU)(\partial_j\log\rhoOU)\,\dd x.
\]
Integration by parts gives
\[
    \mathcal I_{ij}(\rhoOU)
    =
    -\int \rhoOU\,\partial_{ij}^2\log\rhoOU\,\dd x,
\]
with vanishing boundary term by \cref{ass:regularity}.  This proves \cref{eq:bartlett}.

\subsection{Proof of \cref{lem:Cebd-universal}}
\label{app:proof-lem-cebd-universal}

All expectations below are under $\rhoOU=p_0(\cdot\mid Y_t=y)$.  Since $\E[\delta\mid Y_t=y]=0$,
\[
    \Cebd
    =
    \E[(b-s_t)\delta^\top\mid Y_t=y]
    =
    \E[b\delta^\top\mid Y_t=y].
\]
Using $\delta=\alphat^{-1}\nabla_x\log\rhoOU$, the $(i,j)$ entry is
\[
\begin{aligned}
    (\Cebd)_{ij}
    &=
    \frac1{\alphat}
    \int b_i(x;y,t)\,\partial_{x_j}\log\rhoOU(x)\,\rhoOU(x)\,\dd x \\
    &=
    \frac1{\alphat}
    \int b_i(x;y,t)\,\partial_{x_j}\rhoOU(x)\,\dd x.
\end{aligned}
\]
Integrating by parts and using the boundary decay from \cref{ass:regularity},
\[
    (\Cebd)_{ij}
    =
    -\frac1{\alphat}
    \int \rhoOU(x)\,\partial_{x_j} b_i(x;y,t)\,\dd x.
\]
Because $b_i(x;y,t)=(\alphat x_i-y_i)/\gammat$,
\[
    \partial_{x_j} b_i(x;y,t)
    =
    \frac{\alphat}{\gammat}\mathbf 1_{\{i=j\}}.
\]
Thus
\[
    (\Cebd)_{ij}
    =
    -\frac1{\alphat}
    \frac{\alphat}{\gammat}\mathbf 1_{\{i=j\}}
    \int \rhoOU(x)\,\dd x
    =
    -\frac1{\gammat}\mathbf 1_{\{i=j\}}.
\]
Hence $\Cebd=-\gammat^{-1}I_d$.

\subsection{Proof of \cref{lem:fisher-H}}
\label{app:proof-lem-fisher-h}

From the posterior expression,
\[
    \log\rhoOU(x)
    =
    \log p_0(x)
    -
    \frac{\|y-\alphat x\|^2}{2\gammat}
    +
    \text{terms independent of }x.
\]
Differentiating twice with respect to $x$ gives
\[
    -\nabla_x^2\log\rhoOU(x)
    =
    -\nabla_x^2\log p_0(x)
    +
    \frac{\alphat^2}{\gammat}I_d
    =
    H_0(x)+\frac{\alphat^2}{\gammat}I_d.
\]
Taking expectation under $\rhoOU$ and applying \cref{lem:M-fisher}'s Bartlett identity gives
\[
    \mathcal I(\rhoOU)
    =
    \E[H_0(X_0)\mid Y_t=y]
    +
    \frac{\alphat^2}{\gammat}I_d
    =
    H(y,t)+\frac{\alphat^2}{\gammat}I_d.
\]
The expression for $\Cdd$ follows from \cref{lem:M-fisher}.

\subsection{Proof of \cref{thm:lfgi}}
\label{app:proof-thm-lfgi}

By \cref{prop:normal-equation}, the population gate satisfies
\[
    \Gstar=-\Cebd \Cdd^{-1}
\]
whenever the inverse exists.  By \cref{lem:Cebd-universal}, $\Cebd=-\gammat^{-1}I_d$.  By \cref{lem:fisher-H},
\[
    \Cdd
    =
    \frac{1}{\gammat\alphat^2}
    (\alphat^2I_d+\gammat H).
\]
Therefore
\[
    -\Cebd \Cdd^{-1}
    =
    \frac1{\gammat}I_d
    \left[\frac{1}{\gammat\alphat^2}(\alphat^2I_d+\gammat H)\right]^{-1}
    =
    \alphat^2(\alphat^2I_d+\gammat H)^{-1}.
\]
This is exactly \cref{eq:lfgi}, and \cref{eq:Psi} is the same identity written as the resolvent map $\Psiop$.

\section{Finite-reference and concentration proofs}
\label{app:finite-reference-proofs}

\subsection{Proof of \cref{thm:snis-lfgi-consistency}}
\label{app:proof-thm-snis-lfgi-consistency}

Let $f(x)=H_0(x)$ and define $W(x)=\KOU(y\mid x)$.  The self-normalized estimator is
\[
    \Hhat_N
    =
    \frac{N^{-1}\sum_{i=1}^N W(X_i)f(X_i)}{N^{-1}\sum_{i=1}^N W(X_i)}.
\]
By the strong law of large numbers and the moment assumption,
\[
    N^{-1}\sum_{i=1}^N W(X_i)f(X_i)
    \to
    \E_{p_0}[W(X)f(X)]
    \quad\text{a.s.},
\]
and
\[
    N^{-1}\sum_{i=1}^N W(X_i)
    \to
    \E_{p_0}[W(X)]=p_t(y)>0
    \quad\text{a.s.}
\]
Thus
\[
    \Hhat_N
    \to
    \frac{\E_{p_0}[\KOU(y\mid X)H_0(X)]}{p_t(y)}
    =
    \E[H_0(X_0)\mid Y_t=y]
    =
    H(y,t)
    \quad\text{a.s.}
\]
Since matrix inversion is continuous on the open set of nonsingular matrices and $A=\alphat^2I_d+\gammat H$ is nonsingular, $\AhatN=\alphat^2I_d+\gammat\Hhat_N$ is nonsingular for all sufficiently large $N$ almost surely and
\[
    \Ghat_N=\alphat^2\AhatN^{-1}
    \to
    \alphat^2A^{-1}=\Gstar.
\]

\subsection{Proof of \cref{prop:risk-weighted-lfgi-perturbation}}
\label{app:proof-prop-risk-weighted-lfgi-perturbation}

Set
\[
    B=\gammat A^{-1/2}\Delta_HA^{-1/2},
    \qquad \nor{B}_{\op}=\epsilon_H<1.
\]
Then
\[
    \AhatN
    =A+\gammat\Delta_H
    =A^{1/2}(I+B)A^{1/2}.
\]
Since $B$ is symmetric and $\lambda_{\min}(I+B)\ge 1-\epsilon_H>0$, we have $\AhatN\succ0$ and
\[
    \AhatN^{-1}=A^{-1/2}(I+B)^{-1}A^{-1/2}.
\]
The resolvent identity gives
\[
    \Ghat_N-\Gstar
    =\alphat^2(\AhatN^{-1}-A^{-1})
    =-\alphat^2\gammat\AhatN^{-1}\Delta_HA^{-1}.
\]
The sign is irrelevant for the risk norm.  Using the previous display,
\[
    \gammat\AhatN^{-1}\Delta_HA^{-1}\Cdd^{1/2}
    =
    A^{-1/2}(I+B)^{-1}BA^{-1/2}\Cdd^{1/2}.
\]
Moreover
\[
    \nor{(I+B)^{-1}B}_{\op}
    \le
    \nor{(I+B)^{-1}}_{\op}\nor{B}_{\op}
    \le
    \frac{\epsilon_H}{1-\epsilon_H}.
\]
By \cref{eq:excess-risk},
\[
\begin{aligned}
    \mathcal R(\Ghat_N;y,t)-\mathcal R(\Gstar;y,t)
    & =
    \nor{(\Ghat_N-\Gstar)\Cdd^{1/2}}_F^2 \\
    & \le
    \alphat^4
    \nor{A^{-1/2}}_{\op}^2
    \nor{(I+B)^{-1}B}_{\op}^2
    \nor{A^{-1/2}\Cdd^{1/2}}_F^2 \\
    & \le
    \alphat^4
    \left(\frac{\epsilon_H}{1-\epsilon_H}\right)^2
    \nor{A^{-1}}_{\op}\,\tr(A^{-1}\Cdd).
\end{aligned}
\]
The final inequality uses
$\nor{A^{-1/2}\Cdd^{1/2}}_F^2=\tr(A^{-1}\Cdd)$.
If $\epsilon_H\le1/2$, then $(1-\epsilon_H)^{-2}\le4$, which gives the stated simplified bound.

\subsection{Iid conditional matrix Bernstein calculation underlying \cref{ass:hessian-concentration}}
\label{app:proof-hessian-concentration-bernstein}

This subsection records the rigorous concentration statement that justifies the iid conditional-sample version of \cref{ass:hessian-concentration}.  Let $X_1,\ldots,X_N$ be iid from $p_0(\cdot\mid Y_t=y)$ and define
\[
    \Xi_{H,i}
    =
    \gammat A^{-1/2}(H_0(X_i)-H)A^{-1/2}.
\]
Assume $\Xi_{H,i}$ are self-adjoint, mean zero, and satisfy
\[
    \|\Xi_{H,i}\|_{\op}\le R_A
    \quad\text{a.s.},
    \qquad
    \left\|\E_{\rhoOU}[\Xi_{H,i}^2]\right\|_{\op}\le v_A^2.
\]

The matrix Bernstein inequality gives, with probability at least $1-\delta$,
\[
    \left\|\frac1N\sum_{i=1}^N \Xi_{H,i}\right\|_{\op}
    \le
    C\left(
        \sqrt{\frac{v_A^2\log(2d/\delta)}{N}}
        +
        \frac{R_A\log(2d/\delta)}{N}
    \right)
\]
for a universal constant $C$.  Since the left-hand side is exactly
\[
    \gammat\left\|A^{-1/2}(\Hhat_N-H)A^{-1/2}\right\|_{\op},
\]
this proves \cref{eq:relative-hessian-concentration} in the iid conditional-sample setting.  For the self-normalized OU reference-bank estimator, \cref{ass:hessian-concentration} is retained as an explicit effective-sample-size hypothesis rather than derived here.

\subsection{Proof of \cref{thm:hessian-reference-condition}}
\label{app:proof-thm-hessian-reference-condition}

This theorem is a deterministic consequence of \cref{ass:hessian-concentration} and \cref{prop:risk-weighted-lfgi-perturbation}.  Under $\epsilon_H\le1/2$,
\[
    \mathcal R(\Ghat_N)-\mathcal R(\Gstar)
    \le
    4\alphat^4\epsilon_H^2\Lambdapole(y,t).
\]
It is sufficient for the right-hand side to be at most $\eta\Gainstar(y,t)$.  Equivalently,
\[
    \epsilon_H
    \le
    \sqrt{\frac{\eta\Gainstar(y,t)}{4\alphat^4\Lambdapole(y,t)}}.
\]
Combining this threshold with the small-perturbation requirement $\epsilon_H\le1/2$ gives the right-hand side of \cref{eq:hessian-reference-sufficient-full}.  If the Bernstein linear term is ignored and only the square-root variance term is kept, solving the resulting inequality for $N_{\eff}$ gives \cref{eq:hessian-reference-sufficient}.


\section{Proofs for normalized-density evaluation}
\label{app:density-proofs}

\subsection{Proof of \cref{prop:closed-form-lfgi-divergence-density}}
\label{app:proof-closed-form-lfgi-divergence-density}
\begin{proof}
Write the SNIS weights as $w_i$ and define centered noisy-score signals $\Delta b_i=b_i-\widehat b$ and centered precisions $\Delta P_i=P_i-\widehat H$.  For OU weights, the derivative of the normalized weights is
\[
    \partial_{y_a}w_i=w_i(\Delta b_i)_a.
\]

Consequently
\[
    \partial_y\widehat b=J_b=\widehat C_{bb}-\gammat^{-1}I_d,
    \qquad
    \partial_y\widehat c=\widehat C_{cb},
    \qquad
    \partial_y\widehat\delta=J_{\widehat\delta}=\widehat C_{cb}-J_b .
\]
Moreover, for each coordinate $a$,
\[
    (T_a)_{uv}:=\partial_{y_a}\widehat H_{uv}
    =\sum_i w_i(\Delta b_i)_a(\Delta P_i)_{uv}.
\]
Let $A=\alphat^2I_d+\gammat\widehat H$ and $\widehat G=\alphat^2A^{-1}$.  The resolvent derivative is
\[
    \partial_{y_a}\widehat G
    =-\alphat^2A^{-1}(\gammat T_a)A^{-1}
    =-\frac{\gammat}{\alphat^2}\widehat G T_a\widehat G .
\]
Taking the divergence of $\widehat s=\widehat b+\widehat G\widehat\delta$ gives
\[
    \nabla\!\cdot\widehat s
    =\tr(J_b)+\sum_a e_a^\top\widehat G(\partial_{y_a}\widehat\delta)
      +\sum_a e_a^\top(\partial_{y_a}\widehat G)\widehat\delta .
\]
The middle term is $\langle\widehat G,J_{\widehat\delta}^\top\rangle_F$.  Substituting the resolvent derivative into the last term yields
\[
    -\frac{\gammat}{\alphat^2}
    \sum_{a,u,v}\widehat G_{au}T_{auv}(\widehat G\widehat\delta)_v,
\]
which proves the formula.
\end{proof}

\subsection{Derivation of correction and evidence identities}
\label{app:derivation-correction-evidence-identities}
For any proposal density $q$ whose support covers the posterior support and any integrable test function $f$,
\[
    \E_{p_0}f
    =\frac{\int f(x)\tildep(x)\,dx}{Z}
    =\frac{\E_q[f(X)\tildep(X)/q(X)]}{\E_q[\tildep(X)/q(X)]}.
\]
Taking $q=\qpf$ gives the correction weights used for surrogate-generated samples.  The forward evidence identity follows from
\[
    \E_{\qpf}\left[\frac{\tildep(Y)}{\qpf(Y)}\right]
    =\int \tildep(y)\,dy
    =Z,
\]
and the reciprocal identity for posterior-pilot samples follows from
\[
    \E_{p_0}\left[\frac{\qpf(X)}{\tildep(X)}\right]
    =\int \frac{\qpf(x)}{Zp_0(x)}p_0(x)\,dx
    =\frac1Z.
\]
The pointwise diagnostic is obtained by taking logarithms of the ideal identity $p_0=\tildep/Z$: when $\qpf=p_0$, $\log\tildep(x)-\log\qpf(x)=\log Z$ pointwise.  Away from that ideal case the same quantity is a calibration diagnostic rather than an exact evidence identity.

\subsection{Proof of \cref{lem:gn-psd-conditioning,prop:proxy-hessian-gate-density}}
\label{app:proof-gn-proxy-density}
\begin{proof}
For the PSD lemma, convexity of the Loewner cone gives $\widehat H_{\GN}=\sum_i\omega_iP_i^{\GN}\succeq\lambda_0I_d$.  Hence $\widehat A_{\GN}=\alphat^2I_d+\gammat\widehat H_{\GN}\succeq(\alphat^2+\gammat\lambda_0)I_d$.  Spectral calculus applied to $\widehat G_{\GN}=\alphat^2\widehat A_{\GN}^{-1}$ gives the eigenvalue bound.

For the proxy perturbation proposition, write
\[
    \widetilde A=A+\gammat\Delta
    =A^{1/2}(I+E)A^{1/2},
    \qquad
    E=\gammat A^{-1/2}\Delta A^{-1/2}.
\]
If $\|E\|_{\op}=\varepsilon_{\rm prox}<1$, then $(I+E)^{-1}$ exists and
\[
    \|(I+E)^{-1}-I\|_{\op}\le \frac{\varepsilon_{\rm prox}}{1-\varepsilon_{\rm prox}}.
\]
Therefore
\[
    \|\widetilde G-\Gstar\|_{\op}
    =\alphat^2\|A^{-1/2}[(I+E)^{-1}-I]A^{-1/2}\|_{\op}
    \le
    \alphat^2\|A^{-1/2}\|_{\op}^2
    \frac{\varepsilon_{\rm prox}}{1-\varepsilon_{\rm prox}}.
\]
Substituting this operator perturbation into the local quadratic excess-risk identity gives the final claim.
\end{proof}


\section{Closed-form Gaussian formulas}
\label{app:gaussian-formulas}

Let $p_0=\N(m,P^{-1})$ with $P\succ0$.  Write $\Sigma=P^{-1}$.  Under the OU forward process,
\[
    Y_t=\alphat X_0+\sqrt{\gammat}\xi,
    \qquad
    \xi\sim\N(0,I_d),
\]
the OU marginal is
\[
    Y_t\sim\N(\alphat m,C_t),
    \qquad
    C_t=\alphat^2P^{-1}+\gammat I_d.
\]
Hence the exact OU-marginal score is
\[
    s_t(y)
    =
    -C_t^{-1}(y-\alphat m)
    =
    -(\alphat^2I_d+\gammat P)^{-1}P(y-\alphat m).
\]

The OU posterior $X_0\mid Y_t=y$ is Gaussian with precision, covariance, and mean
\[
    Q_t=P+\frac{\alphat^2}{\gammat}I_d,
    \qquad
    \Sigma_{0\mid t}=Q_t^{-1},
    \qquad
    m_{0\mid t}(y)
    =
    Q_t^{-1}\left(Pm+\frac{\alphat}{\gammat}y\right).
\]
The posterior score is
\[
    \nabla_x\log p_0(x\mid Y_t=y)
    =
    -Q_t(x-m_{0\mid t}(y)).
\]

The target score and observed information are
\[
    s_0(x)=-P(x-m),
    \qquad
    H_0(x)=P.
\]
Thus the Tweedie, TSI, and disagreement signals are
\[
    b(x;y,t)=\frac{\alphat x-y}{\gammat},
    \qquad
    c(x;t)=-\frac1{\alphat}P(x-m),
\]
\[
    \delta(x;y,t)=c(x;t)-b(x;y,t)
    =
    \frac1{\alphat}\nabla_x\log p_0(x\mid Y_t=y)
    =
    -\frac1{\alphat}Q_t(x-m_{0\mid t}(y)).
\]
Since the posterior Fisher information is $Q_t$, the normal-equation moments are
\[
    \Cdd(t)=\E[\delta\delta^\top\mid Y_t=y]
    =
    \frac1{\alphat^2}Q_t
    =
    \frac{1}{\gammat\alphat^2}(\alphat^2I_d+\gammat P),
    \qquad
    \Cebd(t)=-\frac1{\gammat}I_d.
\]
Finally, the exact LFGI gate is
\[
    \Gstar(t)
    =
    -\Cebd(t)\Cdd(t)^{-1}
    =
    \alphat^2(\alphat^2I_d+\gammat P)^{-1}.
\]
If $Pu_j=\lambda_j u_j$, then
\[
    \Gstar(t)u_j
    =
    \frac{\alphat^2}{\alphat^2+\gammat\lambda_j}u_j,
\]
which is the spectral attenuation formula used above.


\bibliographystyle{plainnat}
\bibliography{refs}

@book{vershynin2018high,
  title={High-Dimensional Probability},
  author={Vershynin, Roman},
  year={2018},
  publisher={Cambridge University Press}
}

@book{wainwright2019high,
  title={High-Dimensional Statistics},
  author={Wainwright, Martin J.},
  year={2019},
  publisher={Cambridge University Press}
}

@book{liu2001combined,
  title={Monte Carlo Strategies in Scientific Computing},
  author={Liu, Jun S.},
  series={Springer Series in Statistics},
  year={2001},
  publisher={Springer},
  address={New York}
}

@article{kong1994sequential,
  title={Sequential imputations and Bayesian missing data problems},
  author={Kong, Augustine and Liu, Jun S. and Wong, Wing Hung},
  journal={Journal of the American Statistical Association},
  volume={89},
  number={425},
  pages={278--288},
  year={1994},
  publisher={Taylor \& Francis}
}

@book{vaart1998asymptotic,
  title={Asymptotic Statistics},
  author={van der Vaart, Aad W.},
  year={1998},
  publisher={Cambridge University Press}
}

@inproceedings{song2021score,
  title        = {Score-Based Generative Modeling through Stochastic Differential Equations},
  author       = {Song, Yang and Sohl-Dickstein, Jascha and Kingma, Diederik P. and Kumar, Abhishek and Ermon, Stefano and Poole, Ben},
  booktitle    = {International Conference on Learning Representations (ICLR)},
  year         = {2021},
  url          = {https://openreview.net/forum?id=PxTIG12RRHS},
  eprint       = {2011.13456},
  archivePrefix= {arXiv},
  primaryClass = {cs.LG}
}

@inproceedings{liu2016svgd,
  title        = {Stein Variational Gradient Descent: A General Purpose Bayesian Inference Algorithm},
  author       = {Liu, Qiang and Wang, Dilin},
  booktitle    = {Advances in Neural Information Processing Systems (NeurIPS)},
  year         = {2016},
  eprint       = {1608.04471},
  archivePrefix= {arXiv},
  primaryClass = {stat.ML},
  url          = {https://arxiv.org/abs/1608.04471}
}

@book{owen2013mc,
  title        = {Monte Carlo: Theory, Methods and Examples},
  author       = {Owen, Art B.},
  year         = {2013},
  publisher    = {Stanford University},
  url          = {https://artowen.su.domains/mc/},
  note         = {Online book}
}

@book{robert2004montecarlo,
  title        = {Monte Carlo Statistical Methods},
  author       = {Robert, Christian P. and Casella, George},
  edition      = {2},
  year         = {2004},
  publisher    = {Springer},
  address      = {New York},
  isbn         = {978-0387212395}
}

@article{hyvarinen2005score,
  title={Estimation of Non-Normalized Statistical Models by Score Matching},
  author={Hyv{\"a}rinen, Aapo},
  journal={Journal of Machine Learning Research},
  volume={6},
  pages={695--709},
  year={2005}
}

@article{sriperumbudur2017infinite,
  title={Density Estimation in Infinite Dimensional Exponential Families},
  author={Sriperumbudur, Bharath K and Fukumizu, Kenji and Gretton, Arthur and Hyv{\"a}rinen, Aapo and Kumar, Revant},
  journal={Journal of Machine Learning Research},
  volume={18},
  number={57},
  pages={1--59},
  year={2017}
}

@inproceedings{arbel2018kcef,
  title={Kernel Conditional Exponential Family},
  author={Arbel, Michael and Gretton, Arthur},
  booktitle={Proceedings of the 21st International Conference on Artificial Intelligence and Statistics (AISTATS)},
  series={Proceedings of Machine Learning Research},
  volume={84},
  pages={1337--1346},
  year={2018},
  publisher={PMLR}
}

@inproceedings{wenliang2019deepkef,
  title={Learning Deep Kernels for Exponential Family Densities},
  author={Wenliang, Li K. and Sutherland, Danica J. and Strathmann, Heiko and Gretton, Arthur},
  booktitle={International Conference on Machine Learning (ICML)},
  series={Proceedings of Machine Learning Research},
  volume={97},
  year={2019},
  pages={6737--6746},
  publisher={PMLR}
}

@inproceedings{zhou2020nonparametricscore,
  title={Nonparametric Score Estimators},
  author={Zhou, Yuhao and Shi, Jiaxin and Zhu, Jun},
  booktitle={International Conference on Machine Learning (ICML)},
  series={Proceedings of Machine Learning Research},
  volume={119},
  pages={11513--11523},
  year={2020},
  publisher={PMLR}
}

@inproceedings{liu2016ksd,
  title={A Kernelized Stein Discrepancy for Goodness-of-fit Tests},
  author={Liu, Qiang and Lee, Jason and Jordan, Michael},
  booktitle={Proceedings of the 33rd International Conference on Machine Learning (ICML)},
  series={Proceedings of Machine Learning Research},
  volume={48},
  pages={276--284},
  year={2016},
  publisher={PMLR}
}

@inproceedings{chwialkowski2016kernelgof,
  title={A Kernel Test of Goodness of Fit},
  author={Chwialkowski, Krzysztof and Strathmann, Heiko and Gretton, Arthur},
  booktitle={Proceedings of the 33rd International Conference on Machine Learning (ICML)},
  series={Proceedings of Machine Learning Research},
  volume={48},
  pages={2606--2615},
  year={2016},
  publisher={PMLR}
}

@inproceedings{korba2021ksddescent,
  title={Kernel Stein Discrepancy Descent},
  author={Korba, Anna and Aubin-Frankowski, Pierre-Cyril and Majewski, Szymon and Ablin, Pierre},
  booktitle={International Conference on Machine Learning (ICML)},
  series={Proceedings of Machine Learning Research},
  volume={139},
  year={2021},
  publisher={PMLR}
}

@inproceedings{arbel2019mmdflow,
  title={Maximum Mean Discrepancy Gradient Flow},
  author={Arbel, Michael and Korba, Anna and Salim, Adil and Gretton, Arthur},
  booktitle={Advances in Neural Information Processing Systems (NeurIPS)},
  year={2019}
}

@inproceedings{ho2020denoising,
  title        = {Denoising Diffusion Probabilistic Models},
  author       = {Ho, Jonathan and Jain, Ajay and Abbeel, Pieter},
  booktitle    = {Advances in Neural Information Processing Systems (NeurIPS)},
  year         = {2020},
  url          = {https://proceedings.neurips.cc/paper/2020/file/4c5bcfec8584af0d967f1ab10179ca4b-Paper.pdf},
  eprint       = {2006.11239},
  archivePrefix= {arXiv},
  primaryClass = {cs.LG}
}

@article{efron2011tweedie,
  title        = {Tweedie’s Formula and Selection Bias},
  author       = {Efron, Bradley},
  journal      = {Journal of the American Statistical Association},
  volume       = {106},
  number       = {496},
  pages        = {1602--1614},
  year         = {2011},
  doi          = {10.1198/jasa.2011.tm11181}
}

@article{vincent2011connection,
  title        = {A Connection Between Score Matching and Denoising Autoencoders},
  author       = {Vincent, Pascal},
  journal      = {Neural Computation},
  volume       = {23},
  number       = {7},
  pages        = {1661--1674},
  year         = {2011},
  doi          = {10.1162/NECO_a_00142}
}

@article{gretton2012kernel,
  title={A Kernel Two-Sample Test},
  author={Gretton, Arthur and Borgwardt, Karsten and Rasch, Malte and Sch{\"o}lkopf, Bernhard and Smola, Alexander},
  journal={Journal of Machine Learning Research},
  volume={13},
  pages={723--773},
  year={2012}
}

@book{silverman1986density,
  title={Density Estimation for Statistics and Data Analysis},
  author={Silverman, Bernard W},
  year={1986},
  publisher={CRC press}
}

@article{debortoli2024tsm,
  title   = {Target Score Matching},
  author  = {De Bortoli, Valentin and Hutchinson, Michael and Wirnsberger, Peter and Doucet, Arnaud},
  journal = {arXiv preprint arXiv:2402.08667},
  year    = {2024},
  doi     = {10.48550/arXiv.2402.08667}
}

@article{phillips2024pdds,
  title   = {Particle Denoising Diffusion Sampler},
  author  = {Phillips, Angus and Dau, Hai-Dang and Hutchinson, Michael John and De Bortoli, Valentin and Deligiannidis, George and Doucet, Arnaud},
  journal = {arXiv preprint arXiv:2402.06320},
  year    = {2024},
  doi     = {10.48550/arXiv.2402.06320}
}

@article{akhound2024idem,
  title   = {Iterated Denoising Energy Matching for Sampling from Boltzmann Densities},
  author  = {Akhound-Sadegh, Tara and Rector-Brooks, Jarrid and Bose, Avishek Joey and Mittal, Sarthak and Lemos, Pablo and Liu, Cheng-Hao and Sendera, Marcin and Ravanbakhsh, Siamak and Gidel, Gauthier and Bengio, Yoshua and Malkin, Nikolay and Tong, Alexander},
  journal = {arXiv preprint arXiv:2402.06121},
  year    = {2024},
  doi     = {10.48550/arXiv.2402.06121}
}

@inproceedings{havens2025adjoint,
  title     = {Adjoint Sampling: Highly Scalable Diffusion Samplers via Adjoint Matching},
  author    = {Havens, Aaron J. and Miller, Benjamin Kurt and Yan, Bing and Domingo-Enrich, Carles and Sriram, Anuroop and Levine, Daniel S. and Wood, Brandon M. and Hu, Bin and Amos, Brandon and Karrer, Brian and Fu, Xiang and Liu, Guan-Horng and Chen, Ricky T. Q.},
  booktitle = {Proceedings of the 42nd International Conference on Machine Learning},
  series    = {Proceedings of Machine Learning Research},
  volume    = {267},
  pages     = {22204--22237},
  year      = {2025},
  publisher = {PMLR},
  url       = {https://proceedings.mlr.press/v267/havens25a.html},
  eprint    = {2504.11713},
  archivePrefix = {arXiv},
  primaryClass  = {cs.LG},
  doi       = {10.48550/arXiv.2504.11713}
}

@article{kahouli2025cvsm,
  title   = {Control Variate Score Matching for Diffusion Models},
  author  = {Kahouli, Khaled and Elie, Romuald and M{\"u}ller, Klaus-Robert and Berthet, Quentin and Unke, Oliver T. and Doucet, Arnaud},
  journal = {arXiv preprint arXiv:2512.20003},
  year    = {2025},
  doi     = {10.48550/arXiv.2512.20003}
}

@inproceedings{sohldickstein2015nonequilibrium,
  title     = {Deep Unsupervised Learning using Nonequilibrium Thermodynamics},
  author    = {Sohl-Dickstein, Jascha and Weiss, Eric and Maheswaranathan, Niru and Ganguli, Surya},
  booktitle = {Proceedings of the 32nd International Conference on Machine Learning},
  series    = {Proceedings of Machine Learning Research},
  volume    = {37},
  pages     = {2256--2265},
  year      = {2015},
  publisher = {PMLR},
  url       = {https://proceedings.mlr.press/v37/sohl-dickstein15.html}
}

@article{young2026dpsmc,
  title         = {Diffusion Path Samplers via Sequential Monte Carlo},
  author        = {Young, James Matthew and Cordero-Encinar, Paula and Reich, Sebastian and Duncan, Andrew and Akyildiz, {"O}. Deniz},
  journal       = {arXiv preprint arXiv:2601.21951},
  year          = {2026},
  eprint        = {2601.21951},
  archivePrefix = {arXiv},
  primaryClass  = {stat.ML},
  url           = {https://arxiv.org/abs/2601.21951}
}

@inproceedings{huang2024rdmc,
  title     = {Reverse Diffusion Monte Carlo},
  author    = {Huang, Xunpeng and Dong, Hanze and Hao, Yifan and Ma, Yian and Zhang, Tong},
  booktitle = {The Twelfth International Conference on Learning Representations},
  year      = {2024},
  url       = {https://openreview.net/forum?id=esc8PjUQ8e}
}

@inproceedings{grenioux2024slips,
  title     = {Stochastic Localization via Iterative Posterior Sampling},
  author    = {Grenioux, Louis and Noble, Maxence and Gabri{\'e}, Marylou and Durmus, Alain Oliviero},
  booktitle = {Proceedings of the 41st International Conference on Machine Learning},
  series    = {Proceedings of Machine Learning Research},
  year      = {2024},
  publisher = {PMLR}
}

@inproceedings{he2024zeroth,
  title     = {Zeroth-Order Sampling Methods for Non-Log-Concave Distributions: Alleviating Metastability by Denoising Diffusion},
  author    = {He, Ye and Rojas, Kevin and Tao, Molei},
  booktitle = {Advances in Neural Information Processing Systems},
  volume    = {37},
  pages     = {71122--71161},
  year      = {2024}
}

@inproceedings{wu2025rdsmc,
  title     = {Reverse Diffusion Sequential Monte Carlo Samplers},
  author    = {Wu, Luhuan and Han, Yi and Naesseth, Christian A. and Cunningham, John P.},
  booktitle = {Advances in Neural Information Processing Systems},
  year      = {2025},
  eprint        = {2508.05926},
  archivePrefix = {arXiv},
  primaryClass  = {cs.LG},
  url           = {https://arxiv.org/abs/2508.05926}
}

@article{delmoral2006smc,
  title   = {Sequential Monte Carlo Samplers},
  author  = {Del Moral, Pierre and Doucet, Arnaud and Jasra, Ajay},
  journal = {Journal of the Royal Statistical Society: Series B (Statistical Methodology)},
  volume  = {68},
  number  = {3},
  pages   = {411--436},
  year    = {2006},
  doi     = {10.1111/j.1467-9868.2006.00553.x}
}

@inproceedings{he2025reversekl,
  title     = {Training Neural Samplers with Reverse Diffusive {KL} Divergence},
  author    = {He, Jiajun and Chen, Wenlin and Zhang, Mingtian and Barber, David and Hern{\'a}ndez-Lobato, Jos{\'e} Miguel},
  booktitle = {Proceedings of The 28th International Conference on Artificial Intelligence and Statistics},
  series    = {Proceedings of Machine Learning Research},
  year      = {2025},
  publisher = {PMLR}
}

@inproceedings{cordero2025sampling,
  title     = {Sampling by Averaging: A Multiscale Approach to Score Estimation},
  author    = {Cordero-Encinar, Paula and Duncan, Andrew B. and Reich, Sebastian and Akyildiz, {"O}. Deniz},
  booktitle = {Advances in Neural Information Processing Systems},
  year      = {2025}
}

@inproceedings{ko2025latent,
  title     = {Latent Target Score Matching, with an Application to Simulation-Based Inference},
  author    = {Ko, Joohwan and Geffner, Tomas},
  booktitle = {Machine Learning and the Physical Sciences Workshop, NeurIPS},
  year      = {2025}
}

@article{girolami2011riemann,
  title   = {Riemann Manifold Langevin and Hamiltonian Monte Carlo Methods},
  author  = {Girolami, Mark and Calderhead, Ben},
  journal = {Journal of the Royal Statistical Society: Series B (Statistical Methodology)},
  volume  = {73},
  number  = {2},
  pages   = {123--214},
  year    = {2011},
  doi     = {10.1111/j.1467-9868.2010.00765.x}
}

@inproceedings{vargas2024transport,
  title     = {Transport Meets Variational Inference: Controlled Monte Carlo Diffusions},
  author    = {Vargas, Francisco and Padhy, Shreyas and Blessing, Denis and N{"u}sken, Nikolas},
  booktitle = {The Twelfth International Conference on Learning Representations},
  year      = {2024},
  url       = {https://openreview.net/forum?id=PP1rudnxiW}
}

@inproceedings{richter2024improved,
  title     = {Improved Sampling via Learned Diffusions},
  author    = {Richter, Lorenz and Berner, Julius},
  booktitle = {The Twelfth International Conference on Learning Representations},
  year      = {2024},
  url       = {https://openreview.net/forum?id=F2cS6SozN9}
}

@inproceedings{chen2025sequential,
  title     = {Sequential Controlled Langevin Diffusions},
  author    = {Chen, Junhua and Richter, Lorenz and Berner, Julius and Blessing, Denis and Neumann, Gerhard and Anandkumar, Anima},
  booktitle = {The Thirteenth International Conference on Learning Representations},
  year      = {2025}
}

@inproceedings{zhang2022pathintegral,
  title     = {Path Integral Sampler: A Stochastic Control Approach for Sampling},
  author    = {Zhang, Qinsheng and Chen, Yongxin},
  booktitle = {International Conference on Learning Representations},
  year      = {2022},
  url       = {https://openreview.net/forum?id=_uCb2ynRu7Y}
}

@inproceedings{noble2025learned,
  title     = {Learned Reference-Based Diffusion Sampler for Multi-Modal Distributions},
  author    = {Noble, Maxence and Grenioux, Louis and Gabri{\'e}, Marylou and Durmus, Alain Oliviero},
  booktitle = {The Thirteenth International Conference on Learning Representations},
  year      = {2025}
}

@article{anderson1982reverse,
  title   = {Reverse-Time Diffusion Equation Models},
  author  = {Anderson, Brian D. O.},
  journal = {Stochastic Processes and their Applications},
  volume  = {12},
  number  = {3},
  pages   = {313--326},
  year    = {1982},
  doi     = {10.1016/0304-4149(82)90051-5}
}

@article{stuart2010inverse,
  title   = {Inverse Problems: A Bayesian Perspective},
  author  = {Stuart, Andrew M.},
  journal = {Acta Numerica},
  volume  = {19},
  pages   = {451--559},
  year    = {2010},
  doi     = {10.1017/S0962492910000061}
}

@incollection{stein1972bound,
  title     = {A Bound for the Error in the Normal Approximation to the Distribution of a Sum of Dependent Random Variables},
  author    = {Stein, Charles},
  booktitle = {Proceedings of the Sixth Berkeley Symposium on Mathematical Statistics and Probability, Volume 2: Probability Theory},
  pages     = {583--602},
  year      = {1972},
  publisher = {University of California Press}
}

@article{tierney1986laplace,
  title   = {Accurate Approximations for Posterior Moments and Marginal Densities},
  author  = {Tierney, Luke and Kadane, Joseph B.},
  journal = {Journal of the American Statistical Association},
  volume  = {81},
  number  = {393},
  pages   = {82--86},
  year    = {1986},
  doi     = {10.1080/01621459.1986.10478240}
}

@article{tropp2012user,
  title   = {User-Friendly Tail Bounds for Sums of Random Matrices},
  author  = {Tropp, Joel A.},
  journal = {Foundations of Computational Mathematics},
  volume  = {12},
  number  = {4},
  pages   = {389--434},
  year    = {2012},
  doi     = {10.1007/s10208-011-9099-z}
}

@book{nocedal2006numerical,
  title     = {Numerical Optimization},
  author    = {Nocedal, Jorge and Wright, Stephen J.},
  edition   = {2},
  year      = {2006},
  publisher = {Springer},
  address   = {New York},
  doi       = {10.1007/978-0-387-40065-5}
}

@inproceedings{cuturi2013sinkhorn,
  author    = {Cuturi, Marco},
  title     = {Sinkhorn Distances: Lightspeed Computation of Optimal Transport},
  booktitle = {Advances in Neural Information Processing Systems},
  volume    = {26},
  year      = {2013}
}

@article{isserlis1918formula,
  author  = {Isserlis, Leon},
  title   = {On a Formula for the Product-Moment Coefficient of any Order of a Normal Frequency Distribution in any Number of Variables},
  journal = {Biometrika},
  volume  = {12},
  number  = {1/2},
  pages   = {134--139},
  year    = {1918}
}

@article{kolmogorov1933sulla,
  author  = {Kolmogorov, Andrey N.},
  title   = {Sulla determinazione empirica di una legge di distribuzione},
  journal = {Giornale dell'Istituto Italiano degli Attuari},
  volume  = {4},
  pages   = {83--91},
  year    = {1933}
}

@article{smirnov1948table,
  author  = {Smirnov, Nikolai V.},
  title   = {Table for Estimating the Goodness of Fit of Empirical Distributions},
  journal = {The Annals of Mathematical Statistics},
  volume  = {19},
  number  = {2},
  pages   = {279--281},
  year    = {1948}
}

@article{martino2017effective,
  author  = {Martino, Luca and Elvira, V{\'i}ctor and Louzada, Francisco},
  title   = {Effective Sample Size for Importance Sampling Based on Discrepancy Measures},
  journal = {Signal Processing},
  volume  = {131},
  pages   = {386--401},
  year    = {2017}
}

@article{Bui-ThanhGirolami14,
  author  = {Bui-Thanh, Tan and Girolami, Mark A.},
  title   = {Solving Large-Scale {PDE}-Constrained {B}ayesian Inverse Problems with {R}iemann Manifold {H}amiltonian {M}onte {C}arlo},
  journal = {Inverse Problems},
  volume  = {30},
  number  = {11},
  pages   = {114014},
  year    = {2014},
  doi     = {10.1088/0266-5611/30/11/114014}
}

@article{MartinWilcoxBursteddeEtAl12,
  author  = {Martin, James and Wilcox, Lucas C. and Burstedde, Carsten and Ghattas, Omar},
  title   = {A Stochastic {Newton MCMC} Method for Large-Scale Statistical Inverse Problems with Application to Seismic Inversion},
  journal = {SIAM Journal on Scientific Computing},
  volume  = {34},
  number  = {3},
  pages   = {A1460--A1487},
  year    = {2012},
  doi     = {10.1137/110845598}
}

@inproceedings{chen2018neuralode,
  author    = {Chen, Ricky T. Q. and Rubanova, Yulia and Bettencourt, Jesse and Duvenaud, David K.},
  title     = {Neural Ordinary Differential Equations},
  booktitle = {Advances in Neural Information Processing Systems},
  volume    = {31},
  year      = {2018},
  eprint    = {1806.07366},
  archivePrefix = {arXiv},
  primaryClass  = {cs.LG}
}

@inproceedings{grathwohl2019ffjord,
  author    = {Grathwohl, Will and Chen, Ricky T. Q. and Bettencourt, Jesse and Sutskever, Ilya and Duvenaud, David},
  title     = {{FFJORD}: Free-Form Continuous Dynamics for Scalable Reversible Generative Models},
  booktitle = {International Conference on Learning Representations},
  year      = {2019},
  eprint    = {1810.01367},
  archivePrefix = {arXiv},
  primaryClass  = {cs.LG}
}

@inproceedings{song2021maximum,
  author    = {Song, Yang and Durkan, Conor and Murray, Iain and Ermon, Stefano},
  title     = {Maximum Likelihood Training of Score-Based Diffusion Models},
  booktitle = {Advances in Neural Information Processing Systems},
  volume    = {34},
  pages     = {1415--1428},
  year      = {2021},
  eprint    = {2101.09258},
  archivePrefix = {arXiv},
  primaryClass  = {cs.LG}
}

@inproceedings{lu2022maximum,
  author    = {Lu, Cheng and Zheng, Kaiwen and Bao, Fan and Chen, Jianfei and Li, Chongxuan and Zhu, Jun},
  title     = {Maximum Likelihood Training for Score-Based Diffusion {ODE}s by High-Order Denoising Score Matching},
  booktitle = {Proceedings of the 39th International Conference on Machine Learning},
  series    = {Proceedings of Machine Learning Research},
  volume    = {162},
  pages     = {14429--14460},
  publisher = {PMLR},
  year      = {2022},
  eprint    = {2206.08265},
  archivePrefix = {arXiv},
  primaryClass  = {cs.LG}
}

@inproceedings{chen2023probability,
  author    = {Chen, Sitan and Chewi, Sinho and Lee, Holden and Li, Yuanzhi and Lu, Jianfeng and Salim, Adil},
  title     = {The Probability Flow {ODE} Is Provably Fast},
  booktitle = {Advances in Neural Information Processing Systems},
  volume    = {36},
  year      = {2023},
  eprint    = {2305.11798},
  archivePrefix = {arXiv},
  primaryClass  = {cs.LG}
}

@article{benton2023error,
  author  = {Benton, Joe and Deligiannidis, George and Doucet, Arnaud},
  title   = {Error Bounds for Flow Matching Methods},
  journal = {Transactions on Machine Learning Research},
  year    = {2024},
  eprint  = {2305.16860},
  archivePrefix = {arXiv},
  primaryClass  = {stat.ML}
}

@article{meng1996bridge,
  author  = {Meng, Xiao-Li and Wong, Wing Hung},
  title   = {Simulating Ratios of Normalizing Constants via a Simple Identity: A Theoretical Exploration},
  journal = {Statistica Sinica},
  volume  = {6},
  number  = {4},
  pages   = {831--860},
  year    = {1996}
}

@article{neal2001annealed,
  author  = {Neal, Radford M.},
  title   = {Annealed Importance Sampling},
  journal = {Statistics and Computing},
  volume  = {11},
  number  = {2},
  pages   = {125--139},
  year    = {2001},
  doi     = {10.1023/A:1008923215028}
}

@article{gelman1998normalizing,
  author  = {Gelman, Andrew and Meng, Xiao-Li},
  title   = {Simulating Normalizing Constants: From Importance Sampling to Bridge Sampling to Path Sampling},
  journal = {Statistical Science},
  volume  = {13},
  number  = {2},
  pages   = {163--185},
  year    = {1998},
  doi     = {10.1214/ss/1028905934}
}

@article{skilling2006nested,
  author  = {Skilling, John},
  title   = {Nested Sampling for General Bayesian Computation},
  journal = {Bayesian Analysis},
  volume  = {1},
  number  = {4},
  pages   = {833--859},
  year    = {2006},
  doi     = {10.1214/06-BA127}
}

\end{document}